\title{\textbf{Some toy models of self-organized criticality in percolation}}
\author{Rapha\"el Cerf\footnote{
\noindent
DMA, \'Ecole normale sup\'erieure, CNRS, PSL University, 75005, Paris.}
\footnote{
\noindent Universit\'e Paris-Saclay, CNRS, Laboratoire de math\'ematiques d'Orsay, 91405, Orsay, France.}
        \and
        Nicolas Forien\footnotemark[1]\ \footnotemark[2]\ \footnote{Aix Marseille Univ, CNRS, Centrale Marseille, I2M, Marseille, France}
        }

\documentclass{article}
\usepackage[utf8]{inputenc}
\usepackage[T1]{fontenc}
\usepackage[english]{babel}
\usepackage[scale=0.8]{geometry}
\usepackage{textcomp}
\usepackage{url}
\usepackage{soul}
\usepackage{centernot}
\usepackage[pdftex]{graphicx}
\usepackage{multicol}
\usepackage{bbold}
\usepackage{tikz}
\usetikzlibrary{patterns}
\usetikzlibrary{decorations.pathreplacing}
\usetikzlibrary{arrows,decorations.markings}
\usepackage{enumerate}

\usepackage{amsthm}
\usepackage{amsmath}
\usepackage{amssymb}
\usepackage{mathrsfs}
\usepackage{stmaryrd}
\usepackage{amsfonts}
\usepackage{dsfont}
\usepackage[mathscr]{euscript}
\usepackage[final]{pdfpages}
\usepackage[colorlinks=true,linkcolor=blue,linktocpage=true]{hyperref}
\hypersetup{
colorlinks=true, %colorise les liens
breaklinks=true, %permet le retour à la ligne dans les liens trop longs
urlcolor= blue, %couleur des hyperliens
linkcolor= red, %couleur des liens internes
citecolor=blue, %couleur des références 
}

\newcommand{\abs}[1]{\left| #1 \right|}
\newcommand{\norme}[1]{\left\| #1 \right\|}
\newcommand{\acc}[1]{\left\{ #1 \right\}}
\newcommand{\croch}[1]{\left[ #1 \right]}

\newtheorem{theorem}{Theorem}
\newtheorem{lemma}{Lemma}
\newtheorem*{remark}{Remark}

\newcommand{\N}{\mathbb{N}}

\newcommand{\Z}{\mathbb{Z}}
\newcommand{\Proba}{\mathbb{P}}

\newcommand{\Ed}{\mathbb{E}^d}
\newcommand{\En}{\mathbb{E}_n}

\newcommand{\Ent}[1]{\left\lfloor #1\right\rfloor}
\newcommand{\Ceil}[1]{\left\lceil #1\right\rceil}

\newcommand{\souspreuve}[1]{\noindent\textbf{#1}}
\newcommand{\point}{\noindent $\bullet$\ }

\newcommand{\demi}{\frac{1}{2}}

\newcommand{\cvninfty}{\stackrel{n\rightarrow\infty}{\longrightarrow}}
\newcommand{\cvloi}{\stackrel{\mathcal{L}}{\longrightarrow}}

\newcommand{\eqninfty}{\stackrel{n\rightarrow\infty}{\sim}}

\newcommand{\Mn}{\mathcal{M}_n}

\newcommand\numberthis{\addtocounter{equation}{1}\tag{\theequation}}
\newcommand{\diam}{{\rm diam}}
 
\newcommand{\limsupn}{\limsup\limits_{n\to\infty}}
\newcommand{\liminfn}{\liminf\limits_{n\to\infty}}
\newcommand{\quadet}{\quad\text{and}\quad}
\newcommand{\qquadet}{\qquad\text{and}\qquad}
\makeatletter
\newcommand*{\biggg}[1]{{\hbox{$\left#1\vbox to20.5\p@{}\right.\n@space$}}}
\newcommand*{\Biggg}[1]{{\hbox{$\left#1\vbox to23.5\p@{}\right.\n@space$}}}
\makeatother
\newcommand{\quadou}{\quad\text{where}\quad}
\newcommand{\qquadou}{\qquad\text{where}\qquad}
\newcommand{\quador}{\quad\text{or}\quad}

\newcommand{\quadimplique}{\quad\Rightarrow\quad}
\newcommand{\qquadimplique}{\qquad\Rightarrow\qquad}
\newcommand{\quadavec}{\quad\text{with}\quad}
\newcommand{\qquadavec}{\qquad\text{with}\qquad}

\newcommand{\connecte}{\stackrel{\omega}{\longleftrightarrow}}
\newcommand{\Edge}[1]{\mathbb{E}\croch{#1}}
\newcommand{\dprime}{^{\prime\prime}}
\newcommand{\guillemets}[1]{``#1''}

\begin{document}
\maketitle

\begin{abstract}
We consider the Bernoulli percolation model in a finite box and we introduce an automatic control of the percolation parameter, which is a function of the percolation configuration. For a suitable choice of this automatic control, the model is self-critical, i.e., the percolation parameter converges to the critical point~$p_c$ when the size of the box tends to infinity. We study here three simple examples of such models, involving the size of the largest cluster, the number of vertices connected to the boundary of the box, or the distribution of the cluster sizes.
\end{abstract}

\textbf{Keywords:} percolation, criticality, self-organized criticality
\\

\textbf{Mathematics Subject Classification (2020):} 82B43, 60K35, 82B20, 82B27

%\tableofcontents

\section{Introduction}

Our goal is to present a simple model of self-organized criticality built upon the classical Bernoulli percolation model in~$\Z^d$,
which is amenable to a rigorous mathematical analysis.
In the next subsection, we introduce a candidate model and we state
our main theorem, which shows that the parameter of our model converges automatically towards the critical parameter of the Bernoulli
percolation model.

\subsection{Construction of the model and convergence result}

Let~$\Lambda(n)$ be the box of side~$n$ centered at~$0$ in~$\Z^d$ with~$d\geqslant 2$, and let~$\En$ be the set of edges between nearest neighbours of~$\Lambda(n)$. Consider a sequence of increasing functions~\smash{$F_n:\acc{0,1}^{\En}\rightarrow\N$} and a parameter~$a>0$ and set, for~$\omega:\En\rightarrow\acc{0,1}$ a percolation configuration on the edges of the box,
$$p_n(\omega)\ =\ \varphi_n\big(F_n(\omega)\big)
\qquadou
\varphi_n(x)\ =\ \exp\left(-\frac{x}{n^a}\right)\,.$$
This function~$p_n$ will play the role of an automatic control of the percolation parameter, and in this paper we will study three examples of such a control, involving different functions~$F_n$ (see theorem~\ref{thm_clusters_convergence}). The model we consider is given by the following probability distribution on the configurations, which is obtained by replacing the parameter~$p$ of Bernoulli percolation with our feedback function~$p_n$, with the appropriate normalization. Let
\begin{equation}
\label{defmun}
\mu_n\ :\ \omega\in\acc{0,1}^{\En}\ \longmapsto\ \frac{1}{Z_n}\mathbb{P}_{p_n(\omega)}(\omega)
\end{equation}
where 
$$Z_n\ =\ \sum_{\omega\in \acc{0,1}^{\En}}{\mathbb{P}_{p_n(\omega)}(\omega)}$$
will be called the partition function, and~$\Proba_p$ is the Bernoulli percolation measure with parameter~$p$, namely
$$\forall\omega\in\acc{0,1}^{\En}\qquad
\Proba_p(\omega)\ =\ \prod_{e\in\En}{p^{\omega(e)}(1-p)^{1-\omega(e)}}\,.$$
For~$x\in\Lambda(n)$ and~$\omega:\En\rightarrow\acc{0,1}$, we write
$$C(x,\,\omega)\ =\ \Big\{\,y\in\Lambda(n)\ :\ x\connecte y\,\Big\}$$
for the open cluster of~$x$ in the configuration~$\omega$. We show the following convergence result, valid in any dimension~$d\geqslant 2$. The critical point of the Bernoulli percolation model is denoted by~$p_c$.
\begin{theorem}
\label{thm_clusters_convergence}
If~$F_n$ is one of the following sequences of functions:
\begin{align*}
&\,\,(i)\quad F_n : \omega\longmapsto\big|C_{max}(\omega)\big|\,=\,\max\limits_{x\in\Lambda(n)}\,\abs{C(x,\,\omega)}\quadavec 0<a<d\,;\\
&\,(ii)\quad F_n : \omega\longmapsto\big|\mathcal{M}_n(\omega)\big|\,=\,\abs{\Big\{\,x\in\Lambda(n)\ :\ x\connecte\partial\Lambda(n)\,\Big\}}\quadavec d-1<a<d\,;\phantom{\max\limits_{x\in\Lambda(n)}}\\
&(iii)\quad F_n : \omega\longmapsto \abs{B_n^b(\omega)}\,=\,\abs{\Big\{\,x\in\Lambda(n)\ :\  \abs{C(x,\,\omega)}\geqslant n^b\,\Big\}}\quadavec0<b<a<d\,,\phantom{\max\limits_{x\in\Lambda(n)}}
\end{align*}
then the law of~$p_n$ under~$\mu_n$ converges to~$\delta_{p_c}$ when~$n\rightarrow\infty$, and we have the following control:
$$\forall \varepsilon>0\qquad
-\infty\ <\ \liminf\limits_{n\rightarrow\infty}\,\,\frac{1}{(\ln n)n^{v}}\ln\mu_n\big(\,\abs{p_n-p_c}>\varepsilon\,\big)
\ \leqslant\ 
\limsup\limits_{n\rightarrow\infty}\,\,\frac{1}{n^{v}}\ln\mu_n\big(\,\abs{p_n-p_c}>\varepsilon\,\big)\ <\ 0\,,$$
where the~$(\ln n)$ factor can be dropped in the case~$(i)$, and where the exponent~\smash{$v$} is given by
$$\left\{\begin{array}{ll}
v=a\wedge(d-\frac{a}{d})&\text{in case }(i)\,;\phantom{\frac{1}{1}}\\
v=d-1&\text{in case }(ii)\,;\phantom{\frac{1}{1}}\\
v=a\wedge\left(d-\frac{b}{d}\right)&\text{in case }(iii)\,.
\end{array}\right.$$
\end{theorem}
We can see that, for a large interval of the parameter~$a$, the mass of~$\mu_n$ concentrates on the configurations~$\omega$ for which~$p_n(\omega)$ is very close to~$p_c$.
Hence, our model presents a phenomenon of self-organized criticality: the percolation parameter concentrates around the critical point without the need to finely tune a parameter to a precise value (see section~\ref{secSOC} about self-organized criticality).

\subsection{An estimate on the convergence speed}

In case~$(iii)$, an estimate on the convergence speed can be obtained, provided that we assume the existence of the critical exponents~$\beta$ and~$\gamma$.
Let us briefly recall the definition of these exponents (see~\cite{Grimmett}).

The exponent~$\beta$ is related to the percolation probability~$\theta(p)$, which is the probability that the origin belongs to an infinite cluster in a percolation configuration on~$\Z^d$, with percolation parameter~$p$. It is believed (but unproven in general up to now) that~$\theta(p_c)=0$ and that we have the power-law scaling~\smash{$\theta(p)=(p-p_c)^{\beta+o(1)}$} when~$p\to p_c$ with~$p>p_c$, for a certain exponent~$\beta>0$, which depends on the underlying graph~$\Z^d$.

The exponent~$\gamma$ is related to the mean finite cluster size~$\chi(p)$, which is defined as the mean size of the cluster of the origin, conditioned on the event that this cluster is finite.
It is conjectured that we have a power-law~\smash{$\chi(p)=\abs{p-p_c}^{-\gamma+o(1)}$} when~$p\to p_c$, for a certain exponent~$\gamma>0$.

The following theorem indicates which scaling could be deduced from the unproven existence of these critical exponents.
The existence of these exponents was proven in dimension~$2$ for the case of the triangular lattice~\cite{CriticalExponents}, with~$\beta=5/36$ and~$\gamma=43/18$, and our study could easily be adapted on the triangular lattice.

\begin{theorem}
\label{thm_clusters_vitesse}
Take~$F_n=\abs{B_n^b}$ (case~$(iii)$ of theorem~\ref{thm_clusters_convergence}). Assume that there exist real constants~$\beta,\,\gamma>0$ such that
$$\limsup\limits_{\substack{p\rightarrow p_c\\p>p_c}}\,\,\frac{\ln\theta(p)}{\ln(p-p_c)}\ \leqslant\ \beta
\qquadet
\liminf\limits_{\substack{p\rightarrow p_c\\p<p_c}}\,\,\frac{\ln\chi(p)}{\ln(p_c-p)}\ \geqslant\ -\gamma\,.$$
Then, for any real parameters~$a$,~$b$ and~$c$, we have
$$0\,<\,b\,<\,a\,<\,d\quadet c\ <\ {\rm min}\left(\frac{b}{2\gamma},\,\frac{a-b}{2\gamma},\,\frac{d-a}{\beta},\,\frac{d-bd-b}{\beta}\right)
\quad \Longrightarrow\quad 
n^c(p_n-p_c)\ \cvloi\ 0\,.$$
\end{theorem}

We do not believe the condition on~$c$ to be optimal, since the term~$(d-bd-b)/\beta$ comes from a quite rough estimate (see lemma~\ref{vitesse_surcrit}), and it does not allow to deal with~$b\geqslant d/(d+1)$.
It may be possible to improve our technique to get rid of this limitation, and to obtain a similar estimate on the convergence speed for the two first models.

\subsection{Self-organized criticality}
\label{secSOC}

Our model is intended as a toy model of self-organized criticality, a concept which was coined in by the physicists Bak, Tang and Wiesenfeld in their seminal paper~\cite{BTW}.
Many physical models present a phenomenon called phase transition: there is a critical point or a critical curve in the parameter space separating two distinct regions characterized by very different macroscopic behaviours. In such systems, the behaviour of the model at criticality is of particular interest and presents some general features (e.g., fractal geometry or power-law temporal and spatial correlations) which are universal across a wide range of systems and do not depend much on the microscopic details of the system. Bak, Tang and Wiesenfeld pointed out that these \guillemets{critical features} are very common in nature, which is rather surprising because it seems that the parameters need to be finely tuned for a system to be critical. To explain this paradox, they showed that some systems tend to be naturally attracted by critical points, without any fine tuning of the parameters. They call this phenomenon self-organized criticality.

To illustrate this idea, they defined a simple model inspired by the dynamics of a sandpile.
The balance between avalanches and accumulation of sand leads to a state where the system looks critical, with a self-similar distribution of the sizes of the avalanches and the slope self-adjusting to the critical slope, which is the slope at which large-scale avalanches appear.
But despite a very simple dynamics, their model turns out to be very difficult to analyze mathematically~\cite{Dhar06ASM,Jarai18ASM,Hutchcroft20ASM}.

In~\cite{CG}, Cerf and Gorny constructed a self-critical model as a variant of the generalized Ising-Curie-Weiss model, by replacing the temperature with a function depending on the spin configuration. In this paper, we implement the same principle of a feedback from the configuration to the parameter, but within the framework of Bernoulli percolation. This technique to obtain self-organized criticality by \guillemets{artificially} replacing the control parameter with a feedback function depending on the state of the model, which is explained in section~15.4.2 of~\cite{Sornette}, was implemented by physicists to imagine self-critical variants of percolation in~\cite{SornettePerco, FSS93, SocialPerco, CML03dynamical}. However, the understanding of such models often relies on computer simulations and few models are amenable to rigorous mathematical analysis.

\subsection{Self-critical models based on percolation}

There have been several attempts to build mathematical models of self-organized criticality in the percolation setup.

\paragraph{Forest-fire models.}
One strategy to obtain a self-critical model consists in modifying a process of dynamical percolation in order to burn down the large or infinite clusters.
In a model defined by D\"urre~\cite{Durre06existence, Durre06uniqueness}, trees grow with rate~$1$ on each site of the square lattice~$\Z^d$, and lightnings strike each site occupied by a tree with rate~$\lambda$, which makes the cluster of this tree instantaneously become vacant.
Thus, each cluster is burnt with a rate proportional to its size.
If~$\lambda=0$, the state of this model at time~$t$ corresponds to Bernoulli site percolation with parameter~$p=1-e^{-t}$.
The introduction of the lightning parameter~$\lambda>0$ is intended to prevent the appearance of too large clusters, and the most interesting behaviour is expected in a limit~$\lambda\to 0$, where finite clusters are almost never hit by lightnings, whereas no infinite cluster can survive without being immediately destroyed.

The study of this model proved quite challenging, but notable rigorous results have been derived for the one-dimensional case~\cite{BJ05forest, BF}.
In~\cite{TothErdos09}, R\'ath and T\'oth studied a similar model, but on the complete graph with~$n$ vertices, and where instead of trees growing on the sites, edges are added with a certain rate.
Then, the most interesting regime is when each edge is added with rate~$1/n$ and lightnings strike on each site with rate~$\lambda(n)$, with~$n^{-1}\ll \lambda(n)\ll 1$.
In this regime, the authors proved that, under certain conditions on the initial configuration, the stationary distribution of the cluster sizes converges when~$n$ tends to infinity to a power-law distribution, which shows that this mean-field model exhibits a phenomenon of self-organized criticality.
Heuristically, this model behaves as if an infinite cluster was about to appear, but the lightnings prevent it from effectively forming.

In view of this, a natural idea is to try to build a model on an infinite graph where trees grow with rate~$1$ and any cluster of trees which becomes infinite is instantaneously destroyed.
Such a model has been studied on non-amenable graphs~\cite{AST14destr} and in high dimension~\cite{ADK15seven}, but it turns out that such a model does not exist in dimension~$2$, as proved in~\cite{KMS15}, confirming a conjecture of~\cite{BB04self}.
The argument is based on the instructive fact that there exists~$\delta>0$ such that, if one takes a supercritical site percolation configuration on~$\Z^2$, closes all the sites belonging to the infinite open cluster, and reopens each closed site with probability~$\delta$, then almost surely there is still no infinite cluster.
Thus, after the destruction of an infinite cluster, it takes some incompressible time to reconstitute an infinite cluster.
This stands in contradiction with the fact that, in a model where infinite clusters are instantaneously destroyed, there would be an accumulation of such destruction events just after having reached a critical density of trees.

These forest-fire models can be seen as continuous variants of the Drossel-Schwabl forest fire model~\cite{DS92forest}, where instead of instantaneously destroying the clusters, lightnings trigger fires which then spread progressively from one tree to its neighbours, and so on.
This Drossel-Schwabl model has received much attention in the physics literature, with the hope to prove that it exhibits self-organized criticality, in the sense of power-law distributions for the cluster sizes and the duration and the sizes of the fires.
But, despite its quite simple definition, this process has been mainly studied through computer simulations and heuristic reasoning, which gave contradictory predictions about its large-scale behaviour, and few mathematically rigorous results have been obtained (see~\cite{Grassberger02forest} and the references therein).

\paragraph{Frozen percolation.}
Instead of burning large or infinite clusters, another technique consists in freezing clusters when they reach a certain size.
Once an open cluster is frozen, the closed sites on its boundary are forced to remain closed forever, preventing further growth of this cluster.

One may wish to freeze clusters when they become infinite.
Aldous defined such a model on the infinite binary tree and showed that, as soon as half of the sites are open, the system gets blocked in a critical-like state, where finite clusters look like critical percolation clusters~\cite{Aldous00frozen}.

On the square grid~$\Z^2$, such a process with freezing of the infinite clusters does not exist (see~\cite{BT01signal}, which explains an argument of Benjamini and Schramm).
Instead, one may consider diameter-frozen percolation, where clusters are frozen when their diameter exceeds~$N$~\cite{BLN12frozen} or volume-frozen percolation, where clusters freeze when they contain more than~$N$ vertices~\cite{BN17frozen, BKN18frozen}.
Then, some interesting properties arise in the~$N\to\infty$ limit.
On the binary tree, one recovers the behaviour observed by Aldous when only infinite clusters were frozen~\cite{BKN12frozen}. In diameter-frozen percolation on~$\Z^2$, when~$N\to\infty$ most frozen clusters freeze in a near-critical window around the critical time, and these clusters tend to look like critical percolation clusters~\cite{Kiss15frozen}.
Surprisingly, in the diameter-frozen case, this phenomenon of self-organized criticality turns out to be quite sensible to the rule imposed on the boundary of the frozen clusters (namely, the behaviour changes when one does not close the sites on the boundary of a frozen cluster~\cite{BergNolin17}).

In~\cite{Rath09mffp}, a mean-field variant of frozen percolation is studied, where clusters are frozen when they are hit by lightnings.
This model exhibits a similar behaviour to the mean-field forest-fire model described in~\cite{TothErdos09}, that we mentioned before.
For a large regime of the lightning rate, the process gets stuck in a state which looks like a critical Erd{\H o}s-{R}\'{e}nyi random graph, where unfrozen clusters look like critical Galton-Watson trees.

\paragraph{Invasion percolation.}
Invasion percolation is another process constructed as a variant of percolation which exhibits a phenomenon of self-organized criticality.
For each edge~$e$ of the lattice~$\Z^d$, we draw a random variable~$\tau_e$ uniformly distributed on~$[0,1]$, the variables~$(\tau_e)_{e}$ being independent.
Invasion percolation can be defined as a random increasing sequence~$(\mathcal{G}_t)_{t\in\N}$ of subgraphs of~$\Z^d$.
At time~$t=0$, we take~$\mathcal{G}_0$ to be the graph containing only the origin, and no edge.
At each step~$t\in\N$, we look at the edges which connect a vertex in~$\mathcal{G}_t$ to a vertex outside of~$\mathcal{G}_t$, and to obtain~$\mathcal{G}_{t+1}$ we add to~$\mathcal{G}_t$ the edge among these edges for which~$\tau_e$ is minimal (and we also add the corresponding new vertex).
Eventually, this exploration process gives an infinite tree~\smash{$\mathcal{G}=\cup_{n\in\N}\mathcal{G}_t$}, which turns out to look like the so-called incipient infinite cluster of critical percolation~\cite{WW83invasion}.

Heuristically, this can be understood by considering the~$p$-clusters of the underlying dynamical percolation, that is to say the clusters formed of all the edges for which~$\tau_e\leqslant p$.
When the exploration process reaches an infinite~$p$-cluster, then it stays inside this cluster forever and no more edge with~$\tau_e>p$ can be explored.
Thus, progressively, the invasion percolation will reach infinite~$p$-clusters for values~$p>p_c$ more and more close to~$p_c$.
On the contrary, for all~$p<p_c$, the (finite)~$p$-cluster of the origin will eventually be entirely explored.

The above heuristics were made rigorous by~\cite{CCN85invasion, HPS99invasion}, which confirmed that the invaded region asymptotically looks like the incipient infinite cluster.
These results were later precised in the two-dimensional case~\cite{Zhang95invasion, Jarai03invasion, DS12invasion}, but this similarity between planar invasion percolation and critical percolation has some limits: in particular, both measures turn out to be mutually singular~\cite{DSV09invasion}, and the scaling limit of invasion percolation shows rotational and scaling invariance, but it is conjectured that it is not conformal invariant~\cite{GPS18invasion}.

\paragraph{Our approach.}
The model presented in this article is defined in a different way, which may seem less natural but has some advantages.
First, by defining a probability measure on the percolation configurations in a finite box, we avoid the risk to have an ill-defined process (as can be the case when one tries to burn or freeze the infinite clusters).
Also, instead of dynamically adjusting the percolation parameter (like in invasion percolation or in the algorithmic models studied in the physics literature), we directly define this parameter as a function of the percolation configuration.
This function encapsulates the feedback mechanism from the configuration onto the control parameter, which is a key ingredient of self-organized criticality.
Thus, to investigate the self-critical behaviour of our model, we only need to study this feedback function, and in particular its behaviour in a near-critical window (see paragraph~\ref{parZn}).
As we will see, this behaviour is related to challenging problems of finite-size scaling of the cluster sizes, some of which remain unsolved even for the square lattice~$\Z^2$~\cite{NearCriticalPerco}.
In the end, our toy model of self-organized criticality, which is intended to be as simple as possible in its definition, already requires some work and raises some interesting problems.

\subsection{Heuristics for the construction of the model}

Let us explain the heuristics which lead to the choice of the sequences~$F_n$ which appear in the definition of our model. The role of the function~$p_n$ is to introduce a negative feedback which assigns low values~$p_n(\omega)\ll p_c$ to percolation configurations which are \guillemets{typical} of the supercritical phase~$p>p_c$, and high values~$p_n(\omega)\gg p_c$ to configurations which are \guillemets{typical} of the subcritical phase~$p<p_c$. For example, if~$F_n=\abs{C_{max}}$, a configuration~$\omega$ with a largest cluster containing a number of vertices of order~$n^d$ will be assigned a very low value~$p_n(\omega)\ll p_c$. Yet, for this value of the parameter~$p$ in Bernoulli percolation, it is very unlikely to have such a large cluster, which will give~$\omega$ a very low weight in the measure~$\mu_n$. Indeed, we will show that under~$\mu_n$, configurations which are either \guillemets{typically subcritical} or \guillemets{typically supercritical} have a very low probability.
Therefore, the mass of~$\mu_n$ concentrates on configurations~$\omega$ with~$p_n(\omega)$ sufficiently close to~$p_c$, hence the self-critical behaviour of our model.
In fact, the difficult point is to show that the weight of the \guillemets{typically} supercritical or subcritical configurations is much smaller than the weight of the quasi-critical configurations (see paragraph~\ref{parZn}).

Note that our parameter~$a$ does not need to be finely tuned for our result to hold, showing the robustness of the construction.
Indeed, one could expect a different behaviour depending on whether~$a$ is smaller or larger than the so-called fractal dimension~$d_f$ of the incipient infinite cluster (see for example~\cite{Birth}), but~$p_n$ tends to~$p_c$ regardless of~$a$.
In fact, one can conjecture that, if~$a>d_f$, then our~$p_n$ will tend to~$p_c$ \guillemets{from above}, and the configurations in our model might look slightly supercritical, while they might look slightly subcritical when~$a<d_f$. 
This is plausible because the definition of our model more or less amounts to forcing the size of the largest cluster (or~$\abs{\Mn}$, or~$\abs{B_n^b}$) to be of order~$n^a$.

Our list of three models is of course not comprehensive, since many variants could be defined using the same approach. For example, the case of the largest cluster can be extended to the largest cluster in the torus, which means we can set periodic boundary conditions on the box~$\Lambda(n)$. In the model defined with~$B_n^b$ (point~($iii$) of theorem~\ref{thm_clusters_convergence}), one could consider the distribution of the cluster diameters instead of the cluster sizes, by setting
\begin{equation}
\label{Bntilde}
\widetilde{B}_n^b(\omega)\ =\ \abs{\Big\{\,x\in \Lambda(n)\ :\ \diam\,C(x,\,\omega)\geqslant n^b\,\Big\}}\,,
\end{equation}
which gives exactly the same convergence result, under the same conditions for~$a$ and~$b$, and with a similar estimate on the convergence speed.

\subsection{Outline of the article}

The proof of each case of theorem~\ref{thm_clusters_convergence} requires two main steps.
Recall the definition~(\ref{defmun}) of our model: for every percolation configuration~$\omega$, we have~$\mu_n(\omega)=\Proba_{p_n(\omega)}(\omega)/Z_n$.
The first step is to prove that~$\Proba_{p_n(\omega)}(\omega)$ tends to~$0$ exponentially fast and uniformly over all configurations~$\omega$ for which~$p_n(\omega)\notin[p_c-\varepsilon,\,p_c+\varepsilon]$, for a fixed~$\varepsilon>0$.
This step, described in paragraph~\ref{parExp}, relies on classical large deviation estimates far from the critical point.
But this step is not sufficient to prove our result, because of the normalization constant~$Z_n$.
Therefore, the second step is to obtain an adequate lower bound on this partition function~$Z_n$.
This step relies on a monotone coupling of percolation configurations and the search for a fixed point of a certain function (see paragraph~\ref{parZn}).
The crucial tool to construct this fixed point is a geometric surgery procedure, which allows to cut finite subgraphs of~$\Z^d$ in pieces of a given size, without closing too many edges.
This geometric lemma is proved in section~\ref{section_intermede}, after some standard definitions and notations are given in section~\ref{section_definitions}.
The last three sections~\ref{sectionCmax},~\ref{sectionMn} and~\ref{sectionBn} are devoted to the proofs of the three items of theorem~\ref{thm_clusters_convergence}, each section containing the two steps described above (first the exponential decay estimates far from~$p_c$, and then the lower bound on~$Z_n$).
Eventually, theorem~\ref{thm_clusters_vitesse} is proved at the end of section~\ref{sectionBn}.

\subsubsection{Exponential decay estimates far from~$p_c$}
\label{parExp}

Let~$\varepsilon$ be such that~$0<\varepsilon<\min(p_c,\,1-p_c)$.
We start with an upper bound on the right tail of the law of~$p_n$. To this end, we define
\begin{equation}
\label{defTn}
t_n^+\ =\ \Ceil{n^a\big(-\ln(p_c+\varepsilon)\big)}\,.
\end{equation}
Grouping the configurations according to the value of~$F_n$, we can write
\begin{align*}
\mu_n\big(p_n> p_c+\varepsilon\big)
\ &=\ \frac{1}{Z_n}\sum_{\omega\in \acc{0,1}^{\En}}{\mathbb{1}_{\acc{p_n(\omega)> p_c+\varepsilon}}\Proba_{p_n(\omega)}(\omega)}
\ =\ \frac{1}{Z_n}\sum_{\omega\in \acc{0,1}^{\En}}{\mathbb{1}_{\acc{F_n(\omega)< t_n^+}}\Proba_{p_n(\omega)}(\omega)}\\
\ &=\ \frac{1}{Z_n} \sum_{t=0}^{t_n^+-1}{\ \sum_{\omega\in \acc{0,1}^{\En}}{\mathbb{1}_{\acc{F_n(\omega)=t}}\Proba_{p_n(\omega)}(\omega)}}
\ =\ \frac{1}{Z_n}\sum_{t=0}^{t_n^+-1}{\Proba_{\varphi_n(t)}\Big(F_n=t \Big)}\numberthis\label{munsurexact}\\
\ &\leqslant\ \frac{1}{Z_n}\sum_{t=0}^{t_n^+-1}{\Proba_{\varphi_n(t)}\Big(F_n< t_n^+\Big)}\,.
\end{align*}
Yet, the variables~$F_n$ are increasing, whence
\begin{equation}
\label{munsur}
\mu_n\big(p_n> p_c+\varepsilon\big)
\ \leqslant\ \frac{n^d}{Z_n}\Proba_{p_c+\varepsilon}\Big(F_n<\big(-\ln(p_c+\varepsilon)\big)n^a\Big)\,.
\end{equation}
Similarly, we can show that
\begin{equation}
\label{munsous}
\mu_n\big(p_n< p_c-\varepsilon\big)
\ \leqslant\ \frac{n^d}{Z_n}\Proba_{p_c-\varepsilon}\Big(F_n>\big(-\ln(p_c-\varepsilon)\big)n^a\Big)\,.
\end{equation}
Therefore, the first step is to obtain exponential decay estimates for
\begin{equation}
\label{probamun}
\Proba_{p_c+\varepsilon}\Big(F_n<\big(-\ln(p_c+\varepsilon)\big)n^a\Big)
\qquadet
\Proba_{p_c-\varepsilon}\Big(F_n>\big(-\ln(p_c-\varepsilon)\big)n^a\Big)
\end{equation}
with~$\varepsilon>0$ fixed.
This is done in subsections~\ref{subCmaxsous} and~\ref{subCmaxsur} in the case of~$F_n=\abs{C_{max}}$ (case~($i$) of theorem~\ref{thm_clusters_convergence}), in subsections~\ref{subMnsous} and~\ref{subMnsur} for~$F_n=\abs{\Mn}$ (case~($ii$) of theorem~\ref{thm_clusters_convergence}) and~\ref{subBnsous} and~\ref{subBnsur} with~$F_n=\abs{B_n^b}$ (case~($iii$) of theorem~\ref{thm_clusters_convergence}).
The estimates we obtain there are quite standard and follow from classical results in the literature about the behaviour of the cluster sizes in the subcritical and supercritical phases.

\subsubsection{Lower bound on the partition function}
\label{parZn}

The second step, which is the crucial and more interesting step, is to obtain a lower bound on the partition function~$Z_n$.
Indeed, to show that~(\ref{munsur}) and~(\ref{munsous}) tend to~$0$ as~$n$ tends to infinity, one must not only show that the two terms in~(\ref{probamun}) are small enough, but also that~$Z_n$ is not too small.
To obtain this lower bound, we rewrite the partition function as
$$Z_n\ =\ \sum_{\omega\in\acc{0,1}^{\En}}{\Proba_{p_n(\omega)}(\omega)}
\ =\ \sum_{t=0}^{n^d}{\sum_{\substack{\omega\in\acc{0,1}^{\En}\\F_n(\omega)=t}}{\Proba_{\varphi_n(t)}(\omega)}}
\ =\ \sum_{t=0}^{n^d}{\Proba_{\varphi_n(t)}{\big(F_n=t\big)}}\,.$$
To make this expression more concrete, we construct a decreasing coupling~$\omega(0)\geqslant\omega(1)\geqslant\cdots\geqslant\omega(n^d)$ of percolation configurations, such that for every~$t\in\acc{0,\,\ldots,\,n^d}$, the configuration~$\omega(t)$ is distributed according to~$\sim\Proba_{\varphi_n(t)}$.
Then~$Z_n$ rewrites as
\begin{equation}
\label{Cmax_expression_Z_n}
Z_n\ =\ \sum_{t=0}^{n^d}{\Proba\Big(F_n\big(\omega(t)\big)=t\Big)}
\ =\ \Proba\Big(\,\exists t\in\acc{0,\,\ldots,\,n^d}\quad F_n\big(\omega(t)\big)=t\,\Big)\,.
\end{equation}
Hence, the partition function~$Z_n$ is equal to the probability that the random non-increasing function~$t\mapsto F_n\big(\omega(t)\big)$ admits a fixed point.
This leads us to build the coupling step by step, and to consider a (random) stopping time~$T$ located just before this function goes under the first bisector (see figure~\ref{PointFixe}).
We then obtain a lower bound on the probability that the next steps of the coupling lead to a fixed point.
\begin{figure}
\begin{center}
\begin{tikzpicture}
\draw[->] (0,0) node[below left]{$0$} to (0,4) node[left]{$n^d$} to (0,4.5) node[right]{$F_n\big(\omega(t)\big)$};
\draw (-.1,4) to (.1,4);
\draw[->] (0,0) to (4,0) node[below]{$n^d$} to (4.5,0) node[right]{$t$};
\draw (4,-.1) to (4,.1);
\draw (0.5,3.95) node{$\bullet$};
\draw (1.0,3.85) node{$\bullet$};
\draw (1.5,3.65) node{$\bullet$};
\draw[gray] (3.0,1.1) node{$\bullet$};
\draw[gray] (3.5,0.4) node{$\bullet$};
\draw[gray] (4.0,0.2) node{$\bullet$};
\draw[dashed, red!50!black, thick] (0,0) to (3.3,3.3);
\draw[red!50!black] (2.5,2.5) node{$\bullet$};
\draw[red!50!black] (1.5,1.1) node[rotate=45]{$F_n\big(\omega(t)\big)=t$};
\draw[blue!70!black] (2.0,3.3) node{$\bullet$} node[above right] {$F_n\geqslant t$ but \guillemets{close} to~$t$};
\draw[blue!70!black,thick] (2,-0.1) -- node[midway,below]{$t=T$} (2,0.1); 
\draw[blue!70!black,->] (2.1,3.2) to[bend left] (2.5,2.6);
\end{tikzpicture}
\end{center}
\caption{\label{PointFixe}The partition function~$Z_n$ may be expressed as the probability that the random function~$t\mapsto F_n\big(\omega(t)\big)$ admits a fixed point. This allows for the construction of a scenario where we can force such a fixed point to appear, with a reasonable probabilistic cost.}
\end{figure}
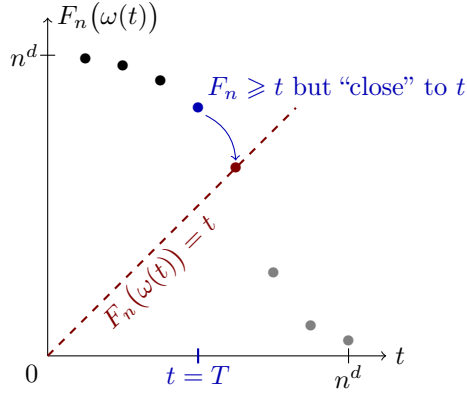

Because this instant~$T$ when we try to force a fixed point typically occurs for a percolation parameter close to~$p_c$, the classical estimates available in subcritical or supercritical percolation are of no use.
Indeed, we need to study the behaviour of~$F_n$ as~$p$ decreases towards~$p_c$, and to show that~$F_n$ does not vary too abruptly close to the critical point.
Our problem is therefore closely related to a question of finite-size scaling, i.e., the behaviour of the model when one takes~$n\rightarrow\infty$ and~$p\rightarrow p_c$ simultaneously (see~\cite{Birth, NearCriticalPerco}).

Yet, we are able to bypass the use of (unproven) scaling laws thanks to the geometric argument of section~\ref{section_intermede}, which is quite general and does not rely on the near-critical behaviour of~$F_n$.
Roughly speaking, this geometric result indicates that to cut a piece of a precise size out of a subgraph of~$\Z^d$ of size~$N$, one only needs to close~\smash{$O(N^{(d-1)/d})$} edges.
This geometric argument allows us to implement a surgery procedure on the configuration~$\omega(T)$ which leads to a fixed point by forcing a reasonable number of edges to be closed in the subsequent steps of the coupling.
The surgery procedure is different for each of the three considered models, but the core ingredient is always this graph separation result.

\begin{remark}
An important goal is to build a similar model of self-organized criticality associated with the Ising model.
A natural strategy consists in adapting the results presented here to the FK percolation model.
However, a major complication arises with the FK model.
Indeed, in a dynamical coupling of the FK processes, there is already a phenomenon of self-organized criticality in the way the edges become open when one approaches the critical point from below~\cite{NearCriticalFKIsing}.
Whereas in dynamical Bernoulli percolation, the opening times of the edges are independent, in FK percolation this independence property is lost, and groups of edges tend to become open simultaneously when~$p$ becomes close to~$p_c$.
As a consequence, our construction of the fixed point using the geometric surgery procedure does not work any more in FK percolation, because it would require to control this phenomenon of simultaneous openings of edges, which is not yet well understood.
Yet, in the article~\cite{Ising2DSOC}, we have
managed to bypass this problem in the particular setting of the planar FK-Ising model, using the estimates
about the near-critical regime proved by~\cite{CerfMessikh2011} in this context.
\end{remark}

\section{Definitions and notations}
\label{section_definitions}

\subsection{The box}

We fix an integer~$d\geqslant 2$ for the whole article. Let~$\Ed$ be the set of edges between nearest neighbours of~$\Z^d$ :
$$\Ed\ =\ \Big\{\,\acc{x,y}\subset\Z^d\ :\ \norme{x-y}_1=1\,\Big\}\,.$$
Let~$n\geqslant 1$. Let us consider the box centered at~$0$ and containing~$n^d$ vertices,
$$\Lambda(n)\ =\ \left[-\frac{n}{2},\,\frac{n}{2}\right[^d\cap\Z^d\ =\ \acc{-\Ent{\frac{n}{2}},\,\ldots,\,\Ent{\frac{n-1}{2}}}^d\,.$$
For~$V\subset\Z^d$ a set of vertices, we write
$$\Edge{V}\ =\ \Big\{\,\acc{x,y}\subset V\ :\ \norme{x-y}_1=1\,\Big\}$$
for the set of edges in~$\Ed$ connecting two vertices of~$V$, and we write in particular~$\En=\Edge{\Lambda(n)}$.
The inner boundary of the box~$\Lambda(n)$ will be denoted
$$\partial\Lambda(n)\ =\ \Big\{\,x\in\Lambda(n)\ :\ \exists\,y\in\Z^d\backslash\Lambda(n)\quad \norme{x-y}_1=1\,\Big\}\,.$$

\subsection{Bernoulli percolation}

For~$0\leqslant p\leqslant 1$, on the space~\smash{$\acc{0,1}^{\Ed}$} equipped with the~$\sigma$-field generated by events depending on finitely many edges, let~$\Proba_p$ be the product measure such that the state of each edge follows a Bernoulli law of parameter~$p$. An element~$\omega:\Ed\rightarrow\acc{0,1}$ is called a percolation configuration. Edges~$e\in\Ed$ such that~$\omega(e)=1$ are said open in~$\omega$, and the other edges are said closed in~$\omega$. Under the law~$\Proba_p$, each edge is open with probability~$p$ and the states of different edges are independent of each other.
For any configuration~$\omega:\Ed\rightarrow\acc{0,1}$ and any edge~$e\in\Ed$, we will write
$$
\omega_e\ :\ f\in\Ed\ \longmapsto\ \left\{
\begin{aligned}
&0&\text{ if }f=e\,,\\
&\omega(f)&\text{ otherwise}
\end{aligned}
\right.
$$
for the configuration obtained from~$\omega$ by closing the edge~$e$. Similarly, for any configuration~$\omega:\Ed\rightarrow\acc{0,1}$ and any set of edges~$H\subset\Ed$, we will write
$$
\omega^H\ :\ f\in\Ed\ \longmapsto\ \left\{
\begin{aligned}
&1&\text{ if }f\in H\,,\\
&\omega(f)&\text{ otherwise}
\end{aligned}
\right.
\qquadet
\omega_H\ :\ f\in\Ed\ \longmapsto\ \left\{
\begin{aligned}
&0&\text{ if }f\in H\,,\\
&\omega(f)&\text{ otherwise}
\end{aligned}
\right.
$$
for the configurations obtained from~$\omega$ by opening or closing all the edges of~$H$. These notations naturally extend to configurations~$\omega:\En\rightarrow\acc{0,1}$ on the edges of the box~$\Lambda(n)$.

\subsection{Clusters}

Let~$\omega:\Ed\rightarrow\acc{0,1}$ be a percolation configuration on~$\Z^d$. For~$x,y\in\Z^d$, we write~\smash{$x \connecte y$}
if there exists a path of open edges in the configuration~$\omega$ joining~$x$ and~$y$. For~$x\in\Z^d$, we will write
$$C(x)\ =\ C(x,\,\omega)\ =\ \acc{\,y\in\Z^d\ :\ x \connecte y\,}$$
for the connected component of~$x$, which is called the cluster of~$x$ in~$\omega$. If~$x\in\Z^d$ and~$Y\subset\Z^d$, we write
$$x \connecte Y\quad\Longleftrightarrow\quad \exists\, y\in Y\quad x \connecte y\,.$$
All these notations naturally extend to percolation configurations restricted to the box~$\Lambda(n)$. Thus, for~$\omega:\En\rightarrow\acc{0,1}$ and~$x\in\Lambda(n)$, we will write~$C(x,\,\omega)$ (or~$C(x)$) for the set of the vertices in~$\Lambda(n)$ which are connected to~$x$ in~$\Lambda(n)$ by an open path in the configuration~$\omega$. When it is not clear whether we consider paths which stay in the box or not, for example if~$\omega$ is defined on~$\Ed$, we will specify~$C_{\Lambda(n)}(x)$ to denote the set of the vertices which are connected to~$x$ by an open path with all its intermediate vertices belonging to~$\Lambda(n)$, i.e., the cluster of~$x$ in the configuration restricted to~$\En$.

\begin{figure}[hbtp]
\begin{center}

\end{center}
\caption{Percolation in the box~$\Lambda(35)$ with, left,~$p=0.48$ and right,~$p=0.52$. Open edges belonging to the largest cluster are drawn in solid lines, while other open edges are in dotted lines.}
\vspace{1cm}
\begin{center}
%
\end{center}
\caption{Percolation in the box~$\Lambda(35)$ with, left,~$p=0.48$ and right,~$p=0.52$. Open edges connected to the boundary of the box by an open path are drawn in solid lines, while other open edges are in dotted lines.}
\end{figure}

For a percolation configuration~$\omega:\En\rightarrow\acc{0,1}$ in the box~$\Lambda(n)$, we will denote by~$C_{max}(\omega)$, or sometimes~$C_{max}(\Lambda(n))$, the largest cluster in~$\omega$, speaking in terms of the number of vertices.
In case of equality between several maximal clusters, we choose one of them with an arbitrary order on subsets of~$\Lambda(n)$. For~$\omega:\En\rightarrow\acc{0,1}$, we also define
$$\mathcal{M}_n(\omega)\ =\ \Big\{\,x\in\Lambda(n)\ :\ x\stackrel{\omega}{\longleftrightarrow}\partial\Lambda(n)\,\Big\}
\qquadet
B_n^b(\omega)\ =\ \Big\{\,x\in \Lambda(n)\ :\ \abs{C_{\Lambda(n)}(x,\,\omega)}\geqslant n^b\,\Big\}\,,$$
where~$b>0$ is a fixed parameter. Given~$p\in[0,1]$, let
$$\theta(p)\ =\ \Proba_p\big(\abs{C(0)}=\infty\big)$$
be the probability that the origin lies in an infinite open cluster in a percolation configuration drawn according to~$\Proba_p$. We will write~$p_c$ for the critical point of Bernoulli percolation in dimension~$d$, defined by
$$p_c\ =\ \inf\Big\{\,p\in[0,1]\ :\ \theta(p)>0\,\Big\}\,.$$

\section{Geometrical interlude}
\label{section_intermede}

\subsection{Main result}

The purpose of this section is to show the following geometric inequality, which one could sum up as \guillemets{separating a cluster of a given size in a graph~$(V,\,E)$ requires at most~\smash{$O(\abs{V}^{(d-1)/d})$} edges}.

\begin{lemma}
\label{lemme_boucher}
There exists a constant~$K=K(d)$ such that, for any finite connected subgraph~$G=(V,\,E)$ of~$(\Z^d,\,\Ed)$, for any vertex~$x\in V$ and for any integer~$m$ such that~$1\leqslant m\leqslant \abs{V}$, there exists a subset~$E_0\subset E$ of edges of~$G$ with cardinality
$$\abs{E_0}\ \leqslant\ K\abs{V}^{\frac{d-1}{d}}$$
such that the connected component of~$x$ in the graph~$(V,\,E\backslash E_0)$ contains exactly~$m$ vertices.
\end{lemma}

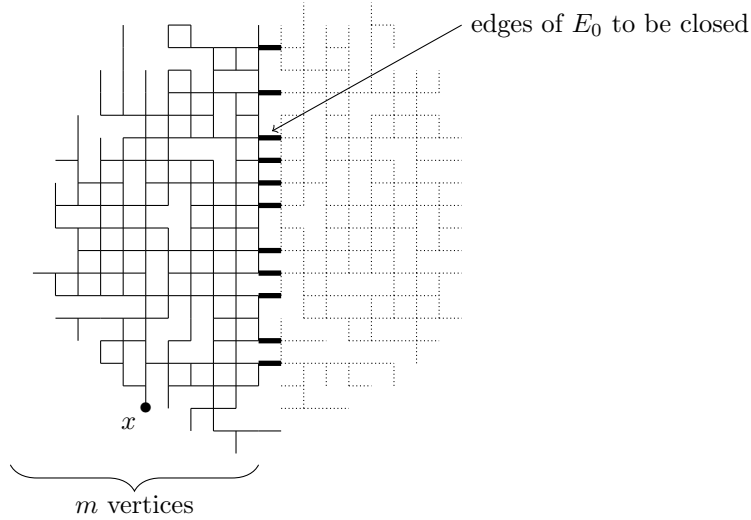
\begin{figure}[h!]
\begin{center}
\begin{tikzpicture}[scale=0.3]
\draw (1,8) to (2,8); \draw (2,6) to (3,6); \draw (2,7) to (3,7); \draw (2,7) to (2,8); \draw (2,8) to (3,8); \draw (2,10) to (3,10); \draw (2,10) to (2,11); \draw (2,11) to (3,11); \draw (2,11) to (2,12); \draw (2,13) to (3,13); \draw (3,5) to (3,6); \draw (3,6) to (4,6); \draw (3,7) to (4,7); \draw (3,7) to (3,8); \draw (3,8) to (4,8); \draw (3,8) to (3,9); \draw (3,9) to (4,9); \draw (3,9) to (3,10); \draw (3,10) to (4,10);  \draw (3,11) to (4,11); \draw (3,11) to (3,12); \draw (3,12) to (4,12); \draw (3,12) to (3,13); \draw (3,13) to (3,15); \draw (4,4) to (5,4); \draw (4,4) to (4,5); \draw (4,5) to (5,5); \draw (4,6) to (5,6); \draw (4,7) to (5,7); \draw (4,7) to (4,8); \draw (4,8) to (5,8); \draw (4,8) to (4,9); \draw (4,9) to (5,9); \draw (4,9) to (4,10); \draw (4,10) to (5,10); \draw (4,10) to (4,11); \draw (4,11) to (5,11); \draw (4,11) to (4,12); \draw (4,12) to (5,12); \draw (4,12) to (4,13); \draw (4,13) to (5,13); \draw (4,13) to (4,14); \draw (4,15) to (5,15); \draw (4,15) to (4,16); \draw (4,16) to (4,17); \draw (5,3) to (6,3); \draw (5,3) to (5,4); \draw (5,4) to (5,5); \draw (5,5) to (6,5); \draw (5,5) to (5,6); \draw (5,6) to (6,6); \draw (5,7) to (6,7); \draw (5,7) to (5,8); \draw (5,8) to (6,8); \draw (5,8) to (5,9); \draw (5,9) to (6,9); \draw (5,9) to (5,10); \draw (5,10) to (6,10); \draw (5,10) to (5,11); \draw (5,11) to (6,11); \draw (5,11) to (5,12); \draw (5,12) to (5,13); \draw (5,13) to (6,13); \draw (5,13) to (5,14); \draw (5,14) to (6,14); \draw (5,15) to (6,15); \draw (5,15) to (5,16); \draw (5,16) to (5,17); \draw (5,17) to (5,18); \draw (5,18) to (5,19); \draw (6,2) to (6,3); \draw (6,3) to (6,4); \draw (6,4) to (7,4); \draw (6,4) to (6,5); \draw (6,5) to (6,6); \draw (6,6) to (7,6); \draw (6,6) to (6,7); \draw (6,7) to (7,7); \draw (6,7) to (6,8); \draw (6,8) to (6,9); \draw (6,9) to (7,9); \draw (6,9) to (6,10); \draw (6,10) to (7,10); \draw (6,10) to (6,11); \draw (6,11) to (7,11); \draw (6,11) to (6,12); \draw (6,12) to (7,12); \draw (6,12) to (6,13); \draw (6,13) to (7,13); \draw (6,13) to (6,14); \draw (6,14) to (7,14); \draw (6,14) to (6,15); \draw (6,15) to (7,15); \draw (6,15) to (6,16); \draw (6,16) to (6,17); \draw (7,2) to (7,3); \draw (7,3) to (8,3); \draw (7,3) to (7,4); \draw (7,4) to (8,4); \draw (7,4) to (7,5); \draw (7,5) to (8,5); \draw (7,6) to (7,7); \draw (7,7) to (8,7); \draw (7,7) to (7,8); \draw (7,8) to (8,8); \draw (7,8) to (7,9); \draw (7,9) to (8,9); \draw (7,9) to (7,10); \draw (7,11) to (7,12); \draw (7,12) to (8,12); \draw (7,12) to (7,13); \draw (7,13) to (8,13); \draw (7,13) to (7,14); \draw (7,14) to (8,14); \draw (7,14) to (7,15); \draw (7,15) to (8,15); \draw (7,15) to (7,16); \draw (7,16) to (8,16); \draw (7,16) to (7,17); \draw (7,17) to (8,17); \draw (7,18) to (8,18); \draw (7,18) to (7,19); \draw (7,19) to (8,19); \draw (8,1) to (8,2); \draw (8,2) to (9,2); \draw (8,3) to (9,3); \draw (8,3) to (8,4); \draw (8,4) to (9,4); \draw (8,5) to (8,6); \draw (8,6) to (8,7); \draw (8,7) to (9,7); \draw (8,7) to (8,8); \draw (8,8) to (9,8); \draw (8,8) to (8,9); \draw (8,9) to (9,9); \draw (8,9) to (8,10); \draw (8,10) to (9,10); \draw (8,10) to (8,11); \draw (8,11) to (9,11); \draw (8,11) to (8,12); \draw (8,12) to (9,12); \draw (8,12) to (8,13); \draw (8,13) to (9,13); \draw (8,14) to (9,14); \draw (8,14) to (8,15); \draw (8,15) to (9,15); \draw (8,15) to (8,16); \draw (8,16) to (9,16); \draw (8,16) to (8,17); \draw (8,18) to (9,18); \draw (8,18) to (8,19); \draw (9,1) to (11,1); \draw (9,2) to (10,2); \draw (9,1) to (9,3); \draw (9,3) to (10,3); \draw (9,3) to (9,4); \draw (9,4) to (10,4); \draw (9,4) to (9,5); \draw (9,5) to (10,5); \draw (9,5) to (9,6); \draw (9,6) to (10,6); \draw (9,6) to (9,7); \draw (9,7) to (10,7); \draw (9,7) to (9,8); \draw (9,8) to (10,8); \draw (9,8) to (9,9); \draw (9,9) to (10,9); \draw (9,9) to (9,10); \draw (9,10) to (10,10); \draw (9,10) to (9,11); \draw (9,11) to (10,11); \draw (9,11) to (9,12); \draw (9,12) to (10,12); \draw (9,12) to (9,13); \draw (9,14) to (10,14); \draw (9,14) to (9,15); \draw (9,15) to (9,16); \draw (9,16) to (10,16); \draw (9,16) to (9,17); \draw (9,17) to (10,17); \draw (9,17) to (9,18); \draw (9,18) to (10,18); \draw (10,0) to (10,1); \draw (10,2) to (10,3); \draw (10,3) to (11,3); \draw (10,3) to (10,4); \draw (10,4) to (11,4); \draw (10,4) to (10,5); \draw (10,5) to (11,5); \draw (10,5) to (10,6); \draw (10,6) to (11,6); \draw (10,6) to (10,7); \draw (10,7) to (11,7); \draw (10,7) to (10,8); \draw (10,8) to (11,8); \draw (10,8) to (10,9); \draw (10,9) to (11,9); \draw (10,9) to (10,10); \draw (10,10) to (11,10); \draw (10,10) to (10,11); \draw (10,11) to (11,11); \draw (10,11) to (10,12); \draw (10,12) to (11,12); \draw (10,12) to (10,13); \draw (10,13) to (11,13); \draw (10,13) to (10,14); \draw (10,14) to (11,14); \draw (10,14) to (10,15); \draw (10,15) to (11,15); \draw (10,15) to (10,16); \draw (10,16) to (11,16); \draw (10,17) to (11,17); \draw (10,17) to (10,18); \draw (10,18) to (11,18);  \draw (11,1) to (12,1); \draw (11,3) to (11,4); \draw[line width=2pt] (11,4) to (12,4); \draw[line width=2pt] (11,5) to (12,5); \draw (11,5) to (11,6); \draw (11,6) to (11,7); \draw[line width=2pt] (11,7) to (12,7); \draw[line width=2pt] (11,8) to (12,8); \draw (11,8) to (11,9); \draw[line width=2pt] (11,9) to (12,9); \draw (11,9) to (11,10); \draw (11,10) to (11,11); \draw[line width=2pt] (11,11) to (12,11); \draw (11,11) to (11,12); \draw[line width=2pt] (11,12) to (12,12); \draw (11,12) to (11,13); \draw[line width=2pt] (11,13) to (12,13); \draw (11,13) to (11,14); \draw[line width=2pt] (11,14) to (12,14); \draw (11,14) to (11,15); \draw (11,15) to (11,16); \draw[line width=2pt] (11,16) to (12,16); \draw (11,16) to (11,17); \draw (11,17) to (11,18); \draw[line width=2pt] (11,18) to (12,18); \draw (11,18) to (11,19); \draw[densely dotted] (12,2) to (13,2); \draw[densely dotted] (12,3) to (13,3); \draw[densely dotted] (12,4) to (13,4); \draw[densely dotted] (12,4) to (12,5); \draw[densely dotted] (12,5) to (13,5); \draw[densely dotted] (12,5) to (12,6); \draw[densely dotted] (12,7) to (12,8); \draw[densely dotted] (12,8) to (13,8); \draw[densely dotted] (12,9) to (13,9); \draw[densely dotted] (12,9) to (12,10); \draw[densely dotted] (12,10) to (13,10); \draw[densely dotted] (12,10) to (12,11); \draw[densely dotted] (12,11) to (13,11); \draw[densely dotted] (12,11) to (12,12); \draw[densely dotted] (12,12) to (13,12); \draw[densely dotted] (12,12) to (12,13); \draw[densely dotted] (12,13) to (13,13); \draw[densely dotted] (12,13) to (12,14); \draw[densely dotted] (12,14) to (13,14); \draw[densely dotted] (12,14) to (12,15); \draw[densely dotted] (12,15) to (13,15); \draw[densely dotted] (12,16) to (13,16); \draw[densely dotted] (12,17) to (13,17); \draw[densely dotted] (12,17) to (12,18); \draw[densely dotted] (12,18) to (13,18); \draw[densely dotted] (12,18) to (12,19); \draw[densely dotted] (13,2) to (14,2); \draw[densely dotted] (13,2) to (13,3); \draw[densely dotted] (13,3) to (13,4); \draw[densely dotted] (13,4) to (14,4); \draw[densely dotted] (13,5) to (14,5); \draw[densely dotted] (13,6) to (14,6); \draw[densely dotted] (13,6) to (13,7); \draw[densely dotted] (13,7) to (14,7); \draw[densely dotted] (13,7) to (13,8); \draw[densely dotted] (13,8) to (14,8); \draw[densely dotted] (13,8) to (13,9); \draw[densely dotted] (13,9) to (14,9); \draw[densely dotted] (13,9) to (13,10);  \draw[densely dotted] (13,11) to (13,12); \draw[densely dotted] (13,12) to (14,12); \draw[densely dotted] (13,12) to (13,13); \draw[densely dotted] (13,13) to (13,14); \draw[densely dotted] (13,14) to (13,15); \draw[densely dotted] (13,15) to (13,16); \draw[densely dotted] (13,16) to (14,16); \draw[densely dotted] (13,16) to (13,17); \draw[densely dotted] (13,17) to (14,17); \draw[densely dotted] (13,17) to (13,18); \draw[densely dotted] (13,18) to (13,19); \draw[densely dotted] (13,19) to (13,20); \draw[densely dotted] (14,2) to (15,2); \draw[densely dotted] (14,3) to (14,4); \draw[densely dotted] (14,4) to (15,4); \draw[densely dotted] (14,6) to (15,6); \draw[densely dotted] (14,6) to (14,7); \draw[densely dotted] (14,7) to (15,7); \draw[densely dotted] (14,7) to (14,8); \draw[densely dotted] (14,8) to (15,8); \draw[densely dotted] (14,8) to (14,9); \draw[densely dotted] (14,9) to (15,9); \draw[densely dotted] (14,9) to (14,10); \draw[densely dotted] (14,10) to (15,10); \draw[densely dotted] (14,10) to (14,11); \draw[densely dotted] (14,11) to (15,11); \draw[densely dotted] (14,11) to (14,12); \draw[densely dotted] (14,12) to (14,13); \draw[densely dotted] (14,13) to (15,13); \draw[densely dotted] (14,13) to (14,14); \draw[densely dotted] (14,14) to (15,14); \draw[densely dotted] (14,14) to (14,15); \draw[densely dotted] (14,15) to (15,15); \draw[densely dotted] (14,15) to (14,16); \draw[densely dotted] (14,16) to (15,16); \draw[densely dotted] (14,16) to (14,17); \draw[densely dotted] (14,17) to (15,17); \draw[densely dotted] (14,17) to (14,18); \draw[densely dotted] (14,18) to (15,18); \draw[densely dotted] (14,18) to (14,19); \draw[densely dotted] (15,3) to (16,3); \draw[densely dotted] (15,3) to (15,4); \draw[densely dotted] (15,4) to (16,4); \draw[densely dotted] (15,5) to (16,5); \draw[densely dotted] (15,5) to (15,6); \draw[densely dotted] (15,6) to (16,6); \draw[densely dotted] (15,7) to (16,7); \draw[densely dotted] (15,8) to (16,8); \draw[densely dotted] (15,8) to (15,9); \draw[densely dotted] (15,9) to (16,9); \draw[densely dotted] (15,9) to (15,10); \draw[densely dotted] (15,10) to (16,10); \draw[densely dotted] (15,10) to (15,11); \draw[densely dotted] (15,11) to (16,11); \draw[densely dotted] (15,11) to (15,12); \draw[densely dotted] (15,12) to (15,13); \draw[densely dotted] (15,13) to (16,13); \draw[densely dotted] (15,13) to (15,14); \draw[densely dotted] (15,14) to (16,14); \draw[densely dotted] (15,14) to (15,15); \draw[densely dotted] (15,15) to (15,16); \draw[densely dotted] (15,16) to (16,16); \draw[densely dotted] (15,16) to (15,17); \draw[densely dotted] (15,17) to (15,18); \draw[densely dotted] (15,18) to (15,19); \draw[densely dotted] (16,4) to (17,4); \draw[densely dotted] (16,5) to (17,5); \draw[densely dotted] (16,5) to (16,6); \draw[densely dotted] (16,6) to (17,6); \draw[densely dotted] (16,6) to (16,7); \draw[densely dotted] (16,7) to (17,7); \draw[densely dotted] (16,8) to (17,8); \draw[densely dotted] (16,8) to (16,9); \draw[densely dotted] (16,9) to (17,9); \draw[densely dotted] (16,9) to (16,10); \draw[densely dotted] (16,10) to (17,10); \draw[densely dotted] (16,10) to (16,11); \draw[densely dotted] (16,11) to (17,11); \draw[densely dotted] (16,11) to (16,12); \draw[densely dotted] (16,12) to (17,12); \draw[densely dotted] (16,12) to (16,13); \draw[densely dotted] (16,13) to (17,13); \draw[densely dotted] (16,13) to (16,14); \draw[densely dotted] (16,14) to (17,14); \draw[densely dotted] (16,14) to (16,15); \draw[densely dotted] (16,15) to (17,15); \draw[densely dotted] (16,16) to (17,16); \draw[densely dotted] (16,16) to (16,17); \draw[densely dotted] (16,17) to (16,19); \draw[densely dotted] (16,19) to (17,19); \draw[densely dotted] (16,19) to (16,20); \draw[densely dotted] (17,3) to (17,4); \draw[densely dotted] (17,5) to (18,5); \draw[densely dotted] (17,6) to (18,6); \draw[densely dotted] (17,6) to (17,7); \draw[densely dotted] (17,7) to (18,7); \draw[densely dotted] (17,7) to (17,8); \draw[densely dotted] (17,8) to (18,8); \draw[densely dotted] (17,8) to (17,9); \draw[densely dotted] (17,9) to (17,10); \draw[densely dotted] (17,10) to (18,10); \draw[densely dotted] (17,10) to (17,11); \draw[densely dotted] (17,11) to (18,11); \draw[densely dotted] (17,11) to (17,13); \draw[densely dotted] (17,12) to (18,12);
\draw[densely dotted] (17,14) to (18,14); \draw[densely dotted] (17,14) to (17,15); \draw[densely dotted] (17,15) to (18,15); \draw[densely dotted] (17,15) to (17,16); \draw[densely dotted] (17,16) to (18,16); \draw[densely dotted] (17,16) to (17,17); \draw[densely dotted] (17,18) to (17,19); \draw[densely dotted] (18,4) to (18,5); \draw[densely dotted] (18,5) to (19,5); \draw[densely dotted] (18,5) to (18,6); \draw[densely dotted] (18,6) to (19,6); \draw[densely dotted] (18,6) to (18,7); \draw[densely dotted] (18,7) to (19,7); \draw[densely dotted] (18,7) to (18,8); \draw[densely dotted] (18,8) to (19,8); \draw[densely dotted] (18,8) to (18,9); \draw[densely dotted] (18,9) to (19,9); \draw[densely dotted] (18,9) to (18,10); \draw[densely dotted] (18,10) to (18,11); \draw[densely dotted] (18,11) to (19,11); \draw[densely dotted] (18,11) to (18,12); \draw[densely dotted] (18,12) to (19,12); \draw[densely dotted] (18,12) to (18,13); \draw[densely dotted] (18,13) to (19,13); \draw[densely dotted] (18,13) to (18,14); \draw[densely dotted] (18,14) to (19,14); \draw[densely dotted] (18,14) to (18,15); \draw[densely dotted] (18,15) to (19,15); \draw[densely dotted] (18,15) to (18,16); \draw[densely dotted] (18,16) to (19,16); \draw[densely dotted] (18,16) to (18,17); \draw[densely dotted] (19,6) to (20,6); \draw[densely dotted] (19,6) to (19,7); \draw[densely dotted] (19,7) to (19,8); \draw[densely dotted] (19,8) to (20,8); \draw[densely dotted] (19,9) to (20,9); \draw[densely dotted] (19,9) to (19,10); \draw[densely dotted] (19,10) to (20,10); \draw[densely dotted] (19,10) to (19,11); \draw[densely dotted] (19,11) to (20,11); \draw[densely dotted] (19,11) to (19,12); \draw[densely dotted] (19,12) to (20,12); \draw[densely dotted] (19,13) to (20,13); \draw[densely dotted] (19,14) to (20,14); \draw[densely dotted] (19,14) to (19,15); \draw[densely dotted] (19,16) to (19,17);
\draw[decorate, decoration={brace, mirror, amplitude=10pt}] (0,-0.5) to node[midway, yshift=-15pt]{$m$ vertices} (11,-0.5) ;
\draw (6,2) node{$\bullet$} node[below left]{$x$};
\draw[->] (20,19) node[right]{edges of~$E_0$ to be closed} to (11.6, 14.3);

\end{tikzpicture}
\end{center}
\caption{Closing the edges of~$E_0$ (drawn in thick lines) cuts the graph in several connected components, such that~$x$ lies in a component (drawn in normal lines) containing the required number of vertices. Lemma~\ref{lemme_boucher} states that, in dimension~$2$, the subset~$E_0$ can be chosen containing~$O\big(\sqrt{\abs{V}}\big)$~edges.}
\end{figure}

We decompose the proof of this lemma in two steps. In section~\ref{section_boucher}, we prove the \guillemets{butcher's lemma}, which  allows to cut a graph into small components, which may be too small, in particular the component of~$x$ might have a cardinality strictly smaller than the goal size~$m$. In section~\ref{section_chirurgien}, we prove the \guillemets{surgeon's lemma}, which involves an adequate algorithm to reopen some of the edges closed by the butcher's lemma in order to reach the goal size~$m$ for the cluster of~$x$.

\subsection{The \guillemets{butcher's lemma}}
\label{section_boucher}

We start with an upper bound on the number of edges that one needs to remove from a connected graph to divide it into pieces which are all smaller than half of the initial graph.

\begin{lemma}[The butcher's lemma]
\label{lemme_sandwich_jambon}
For every finite subgraph~$G=(V,\,E)$ of~$(\Z^d,\,\Ed)$, there exists a subset~$E_0\subset E$ of edges of~$G$ with cardinality
$$\abs{E_0}\ \leqslant\ 4^{d+1}d^2\abs{V}^{\frac{d-1}{d}}$$
such that any connected component of the graph~$(V,\,E\backslash E_0)$ contains at most~$\Ceil{\abs{V}/2}$ vertices.
\end{lemma}

This separation lemma, which can be summarized by
\guillemets{cutting a graph in two requires~\smash{$O(\abs{V}^{(d-1)/d})$} edges}, 
was proved in~\cite{BenjaminiSeparation12}, corollary 3.3. For completeness, we present here a self-contained proof of this geometric result for the case of~$\Z^d$. The more general technique of~\cite{BenjaminiSeparation12} would make it possible to extend our result to more general graphs, but we choose here to restrict our presentation to the~$d$-dimensional square grid.
For~$x\in\Z^d$, we will write its coordinates~$x=(x_1,\,\ldots,\,x_d)$.
For any finite non-empty subset~$V\subset\Z^d$ and any~$i\in\acc{1,\,\ldots,\,d}$, we define
$$\diam_i\,V\ =\ \max\limits_{x\in V}\,x_i-\min\limits_{x\in V}\,x_i
\qquadet
\diam\,V\ =\ \max\limits_{1\leqslant i\leqslant d}\,\diam_i\,V\,.$$
If~$i\in\acc{1,\,\ldots,\,d}$ and~$m\in\Z$, then
$$T_{i,m}\ =\ \Big\{\,e=\acc{x,y}\in\Ed\ :\ x_i=m\text{ and }y_i=m+1\,\Big\}$$
will denote the slice of edges cutting~$\Z^d$ in two parts in the direction~$i$ between abscissa~$m$ and~$m+1$.
We first prove an auxiliary lemma.
\begin{lemma}
\label{lemme_recurrence_graphes}
For every~$k\in\N$ and for any real number~$A\geqslant 4$, given a subgraph~$G=(V,E)$ of~$(\Z^d,\,\Ed)$ such that~$\abs{V}\leqslant A^d$ and
$$\diam\,V\ \leqslant\ \left(\frac{3}{2}\right)^k(A-1)\,,$$
there exists a subset~$E_0\subset E$ of edges of~$G$ with cardinality 
$$\abs{E_0}\ \leqslant\ 2A^{d-1}+36d^2\left(1-\left(\frac{2}{3}\right)^k\right)A^{d-1}$$
such that any connected component of the graph~$(V,\,E\backslash E_0)$ contains at most~$\Ceil{A^d/2}$ vertices.
\end{lemma}
\begin{remark}
In the sequel, this lemma will only be used with~$A=\abs{V}^{1/d}$ but it will be helpful for the proof to keep this parameter~$A$ fixed rather than have it depending on the graph.
\end{remark}
\begin{proof}
Fix~$A\geqslant 4$. We will proceed by induction on~$k$, and therefore we start with the case~$k=0$.
Let~$G=(V,\,E)$ be a subgraph of~$(\Z^d,\,\Ed)$ such that~$\abs{V}\leqslant A^d$ and~$\diam\,V\leqslant A-1$. Without loss of generality, we can assume that~$V\subset\Lambda(\diam\,V+1)$. Let us choose 
$$E_0\ =\ E\cap\big(T_{1,-1}\cup T_{1,0}\big)\,,$$
whose cardinality satisfies
$$\abs{E_0}
\ \leqslant\ 2\big(\diam\,V+1\big)^{d-1}
\ \leqslant\ 2A^{d-1}\,.$$
If~$C\subset V$ is a connected component of~$(V,\,E\backslash E_0)$, then we have
$$\abs{C}
\ \leqslant\ \max\left(\Ent{\frac{\diam\,V}{2}},\,\Ent{\frac{\diam\,V+1}{2}}\right)\big(\diam\,V+1\big)^{d-1}
\ \leqslant\ \frac{\big(\diam\,V+1\big)^{d}}{2}
\ \leqslant\ \frac{A^d}{2}\,.$$
We now perform the induction step.
Take~$k\geqslant 1$ such that the result holds for~$k-1$. Let~$G=(V,E)$ be a subgraph of~$(\Z^d,\,\Ed)$ such that~$\abs{V}\leqslant A^d$ and
$$\diam\,V\ \leqslant\  \left(\frac{3}{2}\right)^k(A-1)\,.$$
We are going to trim the graph~$G$ to decrease its diameter by a factor~$2/3$. To this end, we will remove slices of edges in the directions~$i$ in which the diameter is \guillemets{too big}. Consider
$$\mathcal{I}\ =\ \acc{\,i\in\acc{1,\,\ldots,\,d}\ :\ \diam_i\,V > \left(\frac{3}{2}\right)^{k-1}(A-1)\,}\,,$$
and take~$i\in\mathcal{I}$. Without loss of generality, one can assume that~\smash{$\min\limits_{x\in V}\,x_i=0$}. By the pigeonhole principle, there exists an integer~$k_i$ satisfying
$$\Ent{\frac{\diam_i\,V}{3}}\ <\ k_i\ \leqslant\ 2\Ent{\frac{\diam_i\,V}{3}}\qquadet
\abs{E\cap T_{i,k_i}}\ \leqslant\ \frac{\abs{E}}{\Ent{\frac{\diam_i\,V}{3}}}\,.$$
We choose such a~$k_i$ and we write, recalling that~$A\geqslant 4$,
\begin{align*}
\Ent{\frac{\diam_i\,V}{3}}
\ &\geqslant\ \frac{\diam_i\,V}{3}-\frac{2}{3}\\
\ &\geqslant\ \frac{1}{3}\left(\frac{3}{2}\right)^{k-1}(A-1)-\frac{2}{3}\\
\ &=\ \frac{1}{9}\left(\frac{3}{2}\right)^{k-1}(A-1)\ +\ \frac{2}{9}\left(\left(\frac{3}{2}\right)^{k-1}(A-1)-3\right)\\
\ &\geqslant\ \frac{1}{9}\left(\frac{3}{2}\right)^{k-1}(A-1)\\
\ &\geqslant\ \frac{1}{9}\left(\frac{3}{2}\right)^{k-1}\frac{3}{4}A\\
\ &=\ \frac{1}{12}\left(\frac{3}{2}\right)^{k-1}A\,.
\end{align*}
Noting that~$\abs{E}\leqslant d\abs{V}\leqslant d A^d$, we get
$$\abs{E\cap T_{i,k_i}}
\ \leqslant\ \left(\frac{2}{3}\right)^{k-1}\frac{12\abs{E}}{A}
\ \leqslant\ \left(\frac{2}{3}\right)^{k-1}\frac{12dA^d}{A}
\ =\ 12d\left(\frac{2}{3}\right)^{k-1}A^{d-1}\,.$$
Consider now
$$E_1\ =\ \bigcup_{i\in\mathcal{I}}{\left(E\cap T_{i,k_i}\right)}\,,$$
whose cardinality satisfies
$$\abs{E_1}\ \leqslant\ 12 d^2 \left(\frac{2}{3}\right)^{k-1}A^{d-1}\,.$$
Let~$G'=(V',E')$ be a maximal connected component of the graph~$(V,\,E\backslash E_1)$, in terms of number of vertices. By construction, we have that, for~$i\in\mathcal{I}$,
$$\diam_i\,V'
\ \leqslant\ \max\Big(k_i,\ \diam_i\,V-(k_i+1)\Big)
\ \leqslant\ \frac{2}{3}\diam_i\,V
\ \leqslant\ \left(\frac{3}{2}\right)^{k-1}(A-1)\,,$$
while for~$i\notin\mathcal{I}$, the definition of~$\mathcal{I}$ implies
$$\diam_i\,V'\ \leqslant\ \diam_i\,V\ \leqslant\ \left(\frac{3}{2}\right)^{k-1}(A-1)\,.$$
Taking the maximum over~$i$ yields
$$\diam\,V'\ \leqslant\ \left(\frac{3}{2}\right)^{k-1}(A-1)\,.$$
Besides, note that~$\abs{V'}\leqslant\abs{V}\leqslant A^d$. Hence, by the induction hypothesis applied to~$G'$, there exists~$E_2\subset E'$ such that
$$\abs{E_2}\ \leqslant\ 2A^{d-1}+36d^2\left(1-\left(\frac{2}{3}\right)^{k-1}\right)A^{d-1}\,,$$
and all connected components of the graph~$(V',\,E'\backslash E_2)$ contain at most~$\Ceil{A^d/2}$ vertices. Now take~$E_0=E_1\cup E_2$. We have
\begin{align*}
\abs{E_0}
\ &=\ \abs{E_1}\,+\,\abs{E_2}\\
\ &\leqslant\ 12 d^2 \left(\frac{2}{3}\right)^{k-1}A^{d-1}\ +\ 2A^{d-1}+36d^2\left(1-\left(\frac{2}{3}\right)^{k-1}\right)A^{d-1}\\
\ &=\ 2A^{d-1}+36d^2\left(1-\left(\frac{2}{3}\right)^k\right)A^{d-1}\,.
\end{align*}
If~$C$ is a connected component of the graph~$(V,\,E\backslash E_0)$, then either~$C\subset V\backslash V'$ which, by maximality of~$V'$, entails~$\abs{C}\leqslant\abs{V}/2\leqslant A^d/2$, or~$C\subset V'$ in which case~$C$ turns out to be a connected component of the graph~$(V',\,E'\backslash E_2)$, which implies~$\abs{C}\leqslant \Ceil{A^d/2}$.
\end{proof}

We can now prove the butcher's lemma, which is a mere rephrasing of lemma~\ref{lemme_recurrence_graphes}.

\begin{proof}[Proof of lemma~\ref{lemme_sandwich_jambon}]
If~$\abs{V}\geqslant 4^d$, this is a straightforward consequence of lemma~\ref{lemme_recurrence_graphes} with
$$A\ =\ \abs{V}^{1/d}
\qquadet
k\ =\ \Ceil{\frac{d\ln A-\ln(A-1)}{\ln 3-\ln 2}}$$
because we then have
$$\diam\,V\ \leqslant\ \abs{V}\ =\ \frac{A^d}{A-1}(A-1)\ \leqslant\ \left(\frac{3}{2}\right)^k(A-1)$$
and the lemma provides us with a subset~$E_0\subset E$ with cardinality satisfying
$$\abs{E_0}\ \leqslant\ \left(2+36d^2\right)A^{d-1}\ \leqslant\ 4^{d+1}d^2\abs{V}^{\frac{d-1}{d}}$$ 
such that all connected components of~$(V,\,E\backslash E_0)$ contain at most~$\Ceil{A^d/2}=\Ceil{\abs{V}/2}$ vertices.
Otherwise, if we have~$\abs{V}< 4^d$, then~$E_0=E$ answers the problem.
\end{proof}

\subsection{The \guillemets{surgeon's lemma}}
\label{section_chirurgien}

The application of the butcher's lemma allows us to separate a graph into connected components which are at least twice smaller than the original graph. If the connected component of~$x$ in the remaining graph still contains more vertices than the goal size~$m$, one can apply again the butcher's lemma to this component of~$x$, to obtain a connected component which contains at most a fourth of the initial number of vertices. This operation can be repeated until the connected component of~$x$ contains strictly less than~$m$ edges, which means that we have closed too many edges. The surgeon's lemma will fix this problem, by reopening some of the edges closed by the butcher's lemma.

\begin{lemma}[The surgeon's lemma]
\label{lemme_boucher_recurrence}
Let~$k\in\N$ and let~$G=(V,\,E)$ be a connected subgraph of~$(\Z^d,\,\Ed)$ with~$\abs{V}\leqslant 2^k$. Let~$x\in V$ and let~$m$ be an integer such that~$1\leqslant m\leqslant\abs{V}$. There exists a subset~$E_0\subset E$ of edges of~$G$ with cardinality satisfying
$$\abs{E_0}\ \leqslant\ \frac{1-a^k}{1-a}4^{d+1} d^2 \abs{V}^{\frac{d-1}{d}}\,,
\qquadou
a=\frac{1}{2^{\frac{d-1}{d}}}\,,$$
such that, in the graph~$(V,\,E\backslash E_0)$, the connected component of~$x$ contains exactly~$m$ vertices.
\end{lemma}

\begin{proof}
We proceed by induction on~$k$. The result is trivial if~$k=0$, so we perform next the induction step. Take~$k\geqslant 1$ such that the result holds for~$k-1$. Let~$G=(V,\,E)$ be a connected subgraph of~$(\Z^d,\,\Ed)$ with~$2^{k-1}<\abs{V}\leqslant 2^k$, let~$x\in V$ and let~$m$ be an integer such that~$1\leqslant m\leqslant\abs{V}$. According to lemma~\ref{lemme_sandwich_jambon}, we can choose a subset~$E_0\subset E$ of cardinality
$$\abs{E_0}\ \leqslant\ 4^{d+1} d^2 \abs{V}^{\frac{d-1}{d}}$$
such that any connected component of the graph~$(V,\,E\backslash E_0)$ contains at most~$2^{k-1}$ vertices. The idea is to reopen the edges of~$E_0$ one by one starting from the cluster of~$x$, in order to make this cluster grow until it reaches or exceeds the size~$m$. Then we will apply the induction hypothesis on the last piece added, which contains at most~$2^{k-1}$ vertices.
\begin{figure}[h]
\begin{center}
\begin{tikzpicture}
\filldraw[pattern = north west lines, pattern color=gray!70] (0,0) to[out=90, in=-20] (-1,2) to[out=160, in=40] (-3,1) to[out=-140, in=190] (-1,-1) to[out=10, in=-90] (0,0);
\filldraw[pattern = north west lines, pattern color=gray!70] (1,0) to[out=90, in=190] (3,2) to[out=10, in=170] (4,2) to[out=-10, in=110] (4,1) to[out=-70, in=-90] (1,0);
\draw[thick,dashed] (4,2) to node[midway, right] {$e_3$} (4,3);
\filldraw[pattern = north west lines, pattern color=gray!70] (3,3) to[out=150, in=170] (4,4) to[out=-10, in=45] (4,3) to[out=-130,  in=-30] (3,3);
\draw[thick,dashed] (4,4) to node[midway, above] {$e_4$} (5,4);
\filldraw[pattern = north west lines, pattern color=gray!70] (5,4) to[out=-90, in=210] (5.8,3.5) to[out=30, in=-20] (5.6,4.2) to[out=160, in=90] (5,4);
\draw[densely dotted] (4,4.1) to[out=0,in=45] (4,-0.5) to[out=-135, in=-30] (-3.4,-0.5) to[out=150, in=-170] (-0.5,4.3) to[out=10, in=180] (4,4.1);
\node at(4,-1.3) {$V_2=V_\sigma$};
\node at(-1.5,0.5) {$V_0$};
\node at(2.6,0.7) {$V_1\backslash V_0$};
\node at(1.1,3.7) {$V'=V_2\backslash V_1=V_\sigma\backslash V_{\sigma-1}$};
\node at(6.7,3.8) {$V_4\backslash V_3$};
\draw[thick] (0,0) to node[midway, above]{$e_1$} (1,0);
\draw (-1.3,-0.2) node{$\bullet$} node[below right]{$x$};
\draw[thick] (3,2) to node[midway, left] {$e_2=e_\sigma$} (3,3) node {$\bullet$} node[above right] {$x'$};
\end{tikzpicture}
\end{center}
\caption{Illustration of the proof of lemma~\ref{lemme_boucher_recurrence}: closing the edges of~$E_0=\acc{e_1,\,e_2,\,e_3,\,e_4}$ cuts the graph in pieces containing at most~$2^{k-1}$ vertices. We reopen the edges~$e_i$ in this order until the number of vertices in the cluster of~$x$ reaches or exceeds~$m$. In the case drawn here,~$\sigma=2$, and~$V_3=V_2$ because the edge~$e_3$ connects two vertices which already belong to~$V_2$.}
\end{figure}
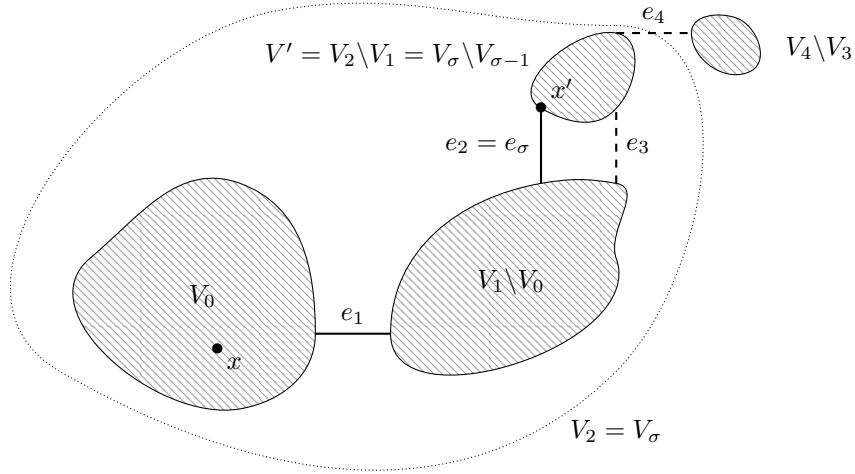

We are going to order the edges of~$E_0$ by exploring them one by one starting from the cluster of~$x$. We start by writing~$V_0$ for the connected component of~$x$ in the graph~$(V,\,E\backslash E_0)$. We have that~$\abs{V_0}\leqslant 2^{k-1}<\abs{V}$, hence~$V_0\subsetneq V$. Yet the graph~$(V,\,E)$ is connected, therefore we can choose an edge~$e_1\in E_0$ incident to this cluster~$V_0$. Assume now that we have defined~$e_1,\,\ldots,\,e_s\in E_0$ for some~$s\geqslant 1$. Let~$V_s$ be the connected component of~$x$ in the graph
$$\big(V,\,E\,\backslash\left(E_0\backslash\acc{e_1,\,\ldots,\,e_s}\right)\big)\,.$$
If~$s<\abs{E_0}$, then we can choose an edge~$e_{s+1}\in E_0$ incident to~$V_s$. Such an edge exists because~$(V,\,E)$ is connected. We proceed with this construction until all the edges of~$E_0$ are ordered in a sequence~$e_1,\,\ldots,\, e_r$ where~$r=\abs{E_0}$. We have then
$$x \in V_0\ \subset\ V_1\ \subset\ \ldots\ \subset\ V_{r} = V\,.$$
If we close all the edges of~$E_0$ and then reopen these edges one by one in the order~$e_1,\,\ldots,\, e_r$, then after having reopened~$s$ edges, the cluster of~$x$ is~$V_s$. Therefore, we introduce
$$\sigma\ =\ \min\Big\{\,s\in\acc{0,\,\ldots,\, r}\ :\ \abs{V_s}\geqslant m\,\Big\}$$
which is the number of reopened edges at which the size of the cluster of~$x$ reaches or exceeds the desired size~$m$. This number~$\sigma$ is well-defined because~$\abs{V_{r}}=\abs{V}\geqslant m$. Assume that~$\sigma\geqslant 1$. By minimality of~$\sigma$, we have~$\abs{V_{\sigma-1}}< m\leqslant\abs{V_\sigma}$, hence~$V_\sigma\neq V_{\sigma-1}$. In that case, the edge~$e_{\sigma}$ must connect a vertex of~$V_{\sigma-1}$ to a vertex~$x'\in V_\sigma\backslash V_{\sigma-1}$. Letting~$m'=m-\abs{V_{\sigma-1}}$, we have that
$$1\ \leqslant\ m'\ \leqslant\ \abs{V_\sigma}-\abs{V_{\sigma-1}}\ =\ \abs{V_\sigma\backslash V_{\sigma-1}}\,.$$
Otherwise, if~$\sigma=0$, we set~$x'=x$ and~$m'=m$, which entails~$1\leqslant m'\leqslant\abs{V_0}$.
\\

\noindent Let us consider the graph~$G'=(V',\,E')$ of the connected component of~$x'$ in~$(V,\,E\backslash E_0)$. The choice of~$E_0$ ensures that~$\abs{V'}\leqslant 2^{k-1}$. What's more, we have that~$V'=V_\sigma\backslash V_{\sigma-1}$ if~$\sigma\geqslant 1$ and~$V'=V_0$ otherwise, which in both cases leads to~$1\leqslant m'\leqslant\abs{V'}$.
The induction hypothesis applied to the graph~$G'=(V',\,E')$ gives us a subset~$E'_0\subset E'$ satisfying
$$\abs{E'_0}\ \leqslant\ \frac{1-a^{k-1}}{1-a}4^{d+1} d^2 \abs{V'}^{\frac{d-1}{d}}
\ \leqslant\ \frac{1-a^{k-1}}{1-a}4^{d+1} d^2 a\abs{V}^{\frac{d-1}{d}}$$
and such that the connected component of~$x'$ in~$(V',\,E'\backslash E'_0)$, which will be denoted~$V'_{x'}$, contains exactly~$m'$ vertices. Now, we consider the set
$$E_0\dprime\ =\ \acc{e_{\sigma+1},\,\ldots,\,e_r}\cup E'_0\,,$$
which is such that
\begin{align*}
\abs{E_0\dprime}
\ &=\ (r-\sigma)+\abs{E'_0}\\
\ &\leqslant\ 4^{d+1} d^2\abs{V}^{\frac{d-1}{d}}+\frac{a-a^k}{1-a}4^{d+1} d^2 \abs{V}^{\frac{d-1}{d}}\\
\ &=\ \frac{1-a^k}{1-a}4^{d+1} d^2 \abs{V}^{\frac{d-1}{d}}\,.
\end{align*}
If~$\sigma=0$, then the connected component of~$x$ in the graph~$(V,\,E\backslash E_0\dprime)$ is~$V'_{x'}$ and thus it contains exactly~$m'=m$~vertices. Otherwise, if~$\sigma\geqslant 1$, then this connected component is~$V_{\sigma-1}\cup V'_{x'}$, which contains~$\abs{V_{\sigma-1}}+m'=m$ vertices.
\end{proof}

\section{Proof of case (\textit{i}) of theorem~\ref{thm_clusters_convergence}}
\label{sectionCmax}

This section is devoted to the proof of the item~($i$) of theorem~\ref{thm_clusters_convergence}.
In this case, the function~$p_n$ is defined by~$p_n(\omega)=\exp(-\abs{C_{max}(\omega)}/n^a)$, where~$C_{max}(\omega)$ denotes the largest cluster in the box~$\Lambda(n)$ in the configuration~$\omega$.
As explained in the introduction, the first step is to show the exponential decay of the distribution of~$\abs{C_{max}}$ in the subcritical and supercritical phases.

\subsection{Exponential decay in the subcritical phase}
\label{subCmaxsous}
We first present a classical estimate about the size of the largest cluster below~$p_c$ :
\begin{lemma}
\label{Cmax_majo_souscritique}
For any~$a\in(0,\,d)$, for~$p<p_c$ and~$A>0$, we have
$$-\infty\ <\ 
\liminfn\,\frac{1}{n^a}\ln\Proba_p\Big(\abs{C_{max}\big(\Lambda(n)\big)}>An^a\Big)
\ \leqslant\ 
\limsupn\,\frac{1}{n^a}\ln\Proba_p\Big(\abs{C_{max}\big(\Lambda(n)\big)}>An^a\Big)\ <\ 0\,.$$
\end{lemma}

\begin{proof}
Let~$a>0$,~$p<p_c$ and~$A>0$. For all~$n\geqslant 1$, we have that
\begin{align*}
\Proba_p\Big(\abs{C_{max}\big(\Lambda(n)\big)}>An^a\Big)
\ =\ \Proba_p\left(\max\limits_{v\in\Lambda(n)}\,\abs{C_{\Lambda(n)}(v)}>An^a\right)
\ &\leqslant\ \Proba_p\left(\max\limits_{v\in\Lambda(n)}\,\abs{C(v)}>An^a\right)\\
\ &\leqslant\ n^d\Proba_p\Big(\abs{C(0)}>An^a\Big)\,.
\end{align*}
According to theorem~6.75 in~\cite{Grimmett}, there exists a constant~$\lambda(p)>0$ such that, for all~$m\geqslant 1$,
\begin{equation}
\label{eq675}
\Proba_p\Big(\abs{C(0)}\geqslant m\Big)\ \leqslant\ e^{-m\lambda(p)}\,.
\end{equation}
It follows that, for all~$n\geqslant 1$,
$$\Proba_p\Big(\abs{C_{max}\big(\Lambda(n)\big)}>An^a\Big)\ \leqslant\ n^d e^{-A\lambda(p)n^a}\,,$$
which implies the desired upper bound.
To create a cluster of size more than~$An^a$, one may simply open a self-avoiding path of~$\Ent{An^a}$ edges and~$\Ent{An^a}+1$ vertices, hence
$$\Proba_p\Big(\abs{C_{max}\big(\Lambda(n)\big)}>An^a\Big)
\ \geqslant\ p^{\Ent{An^a}}\,,$$
which shows the lower bound.
\end{proof}

\subsection{Exponential decay in the supercritical phase}
\label{subCmaxsur}
We establish a corresponding result in the supercritical regime:
\begin{lemma}
\label{Cmax_majo_surcritique}
For all~$a\in(0,\,d)$, for~$p>p_c$ and~$A>0$, we have
$$-\infty\ <\ 
\liminfn\,\frac{1}{n^{d-a/d}}\ln\Proba_p\Big(\abs{C_{max}\big(\Lambda(n)\big)}<An^a\Big)
\ \leqslant\ 
\limsupn\,\frac{1}{n^{d-a/d}}\ln\Proba_p\Big(\abs{C_{max}\big(\Lambda(n)\big)}<An^a\Big)
\ <\ 0\,.$$
\end{lemma}
The upper bound is a consequence of the following result, which easily follows from the classical literature:
\begin{lemma}
\label{CmaxMajoStronger}
For all~$p>p_c$, we have
$$\limsup\limits_{n\rightarrow\infty}\,\frac{1}{n^{d-1}}\ln\Proba_p\left(\abs{C_{max}\big(\Lambda(n)\big)}\leqslant\frac{\theta(p)n^d}{8}\right)\ <\ 0\,.$$
\end{lemma}
\begin{proof}
Assume first that~$d\geqslant 3$.
From theorem~1.2 of~\cite{PisztoraPGD}, it follows that, for~$d\geqslant 3$, for all~$p>\widehat{p_c}$ (where~$\widehat{p_c}$ denotes the slab-percolation threshold),
$$\limsup\limits_{n\rightarrow\infty}\,\frac{1}{n^{d-1}}\ln\Proba_p\left(\abs{C_{max}\left(\Lambda(n)\right)}\leqslant\frac{\theta(p)n^d}{2}\right)\ <\ 0\,.$$
In addition, Grimmett and Marstrand proved the identity~$p_c=\widehat{p_c}$ for~$d\geqslant 3$ in~\cite{GrimmettMarstrand}. The claim for~$d\geqslant 3$ thus follows immediately.
\\

\noindent Consider now the case~$d=2$. Theorem~6.1 of~\cite{Wulff2D} implies that, for all~$p>p_c$, if we consider a percolation configuration on~$\Z^d$ and write~$C_\infty\subset\Z^d$ for the unique infinite cluster of the configuration, then
$$\lim\limits_{n\rightarrow\infty}\,\frac{1}{n}\ln\Proba_p\left(\abs{C_{\infty}\cap\Lambda(n)}\leqslant \frac{\theta(p)n^2}{2}\right)\ <\ 0\,.$$
Thereby, there exists~$L>0$ such that, for all~$n\geqslant 1$,
$$\Proba_p\left(\abs{C_{\infty}\cap\Lambda(n)}\leqslant \frac{\theta(p)n^2}{2}\right)\ \leqslant\ e^{-Ln}\,.$$
Besides, if we set, for~$m\geqslant k\geqslant 1$,
$$L_{k,\,m}\ =\ \Big\{\,\text{The rectangle }\acc{0,\,\ldots,\,k}\times\acc{0,\,\ldots,\,m}\text{ is crossed by an open path in its long direction}\,\Big\}\,,$$
then it follows from equation~(7.110) in~\cite{Grimmett} that there exist positive constants~$C_2(p)$ and~$C_3(p)$ such that, for all~$m\geqslant k\geqslant 1$,
\begin{equation}
\label{rectangle_crossing}
\Proba_p\left(L_{k,\,m}\right)\ \geqslant\ 1-C_2 m e^{-C_3 k}\,.
\end{equation}
Define the rectangles
\begin{align*}
R_1\ =\ \Z^2\cap\left]\frac{n}{2},\,n\right[\times\left[-n,\,n\right[\,,&\qquad
R_2\ =\ \Z^2\cap\left[-n,\,n\right[\times\left]\frac{n}{2},\,n\right[\,,\\
R_3\ =\ \Z^2\cap\left[-n,\,-\frac{n}{2}\right[\times\left[-n,\,n\right[\,,&\qquad
R_4\ =\ \Z^2\cap\left[-n,\,n\right[\times\left[-n,\,-\frac{n}{2}\right[\,,
\end{align*}
which are represented in figure~\ref{dessin_contour}.
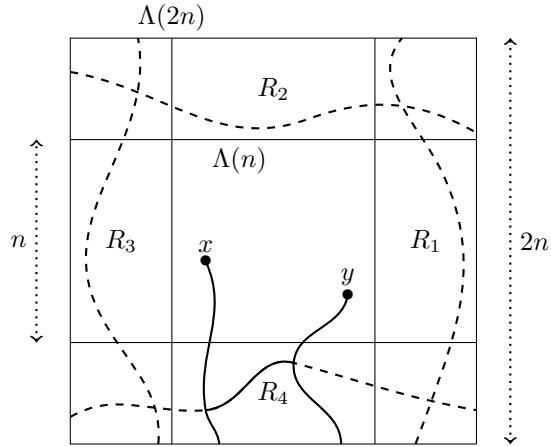
\begin{figure}[h]
\begin{center}
\begin{tikzpicture}[scale=0.9]
\draw (3,-3) to (3,3) to (-3,3) to (-3,-3) to (3,-3);
\draw (1.5,-3) to (1.5,3);
\draw (-1.5,-3) to (-1.5,3);
\draw (-3,1.5) to (3,1.5);
\draw (-3,-1.5) to (3,-1.5);
\node at(2.25,0) {$R_1$};
\node at(-2.25,0) {$R_3$};
\node at(0,2.25) {$R_2$};
\node at(0,-2.25) {$R_4$};
\draw[dashed, thick] (2.1,-3) to[out=70, in=-65] (2.5,1) to[out=115, in=-130] (1.9,3);
\draw[dashed, thick] (-1.7,-3) to[out=85, in=-65] (-2.6,-1) to[out=115, in=-80] (-2,3);
\draw[dashed, thick] (-3,2.5)  to[out=-10, in=200] (0.5,1.8) to[out=20, in=150] (3,1.6);
\draw[dashed, thick] (-3,-2.6) to[out=30, in=185] (-1,-2.5) to[out=5, in=165] (0.3,-1.8) to[out=-15, in=170] (3,-2.5);
\draw[thick] (-1,-0.3) node{$\bullet$} node[above]{$x$} to[out=-65, in=100] (-1,-2.5) to[out=-80, in=95] (-0.8,-3);
\draw[thick] (1.1,-0.8) node{$\bullet$} node[above]{$y$} to[out=-100, in=85] (0.3,-1.8) to[out=-95, in=90] (1,-3);
\draw[thick] (-1,-2.5) to[out=5, in=165] (0.3,-1.8);
\node at(-0.5,1.2) {$\Lambda(n)$};
\node at(-1.5,3.3) {$\Lambda(2n)$};
\draw[<->, dotted, thick] (3.5,-3) to node[midway, right]{$2n$} (3.5,3);
\draw[<->, dotted, thick] (-3.5,-1.5) to node[midway, left]{$n$} (-3.5,1.5);
\end{tikzpicture}
\end{center}
\caption{If each of the four rectangles~$R_1,\,R_2,\,R_3,\,R_4$ is crossed by an open path in its long direction, then~$\Lambda(n)$ is surrounded by an open path in~$\Lambda(2n)$, and thus any two vertices~$x$ and~$y$ in the box~$\Lambda(n)$ cannot be connected to~$\partial\Lambda(2n)$ without being connected to each other by an open path inside~$\Lambda(2n)$.}
\label{dessin_contour}
\end{figure}
Following a classical argument~(see the proof of theorem~7.61 in~\cite{Grimmett}), we consider the events
$$\mathcal{E}_n\,=\,\Big\{\,\text{There exists an open path in }\Lambda(2n)\backslash\Lambda(n)\text{ containing }\Lambda(n)\text{ in its interior}\,\Big\}$$
and
$$\mathcal{F}_n\,=\,\Big\{\,\text{Each of the rectangles }R_1,\ R_2,\ R_3,\ R_4\text{ is crossed by an open path in its long direction}\,\Big\}\,.$$
As illustrated on figure~\ref{dessin_contour}, we have the inclusion~$\mathcal{F}_n\subset\mathcal{E}_n$. In addition, by the FKG inequality, we have that
$$\Proba_p\left(\mathcal{F}_n\right)\ \geqslant\ \Proba_p\left(L_{\Ent{n/2},\,2n}\right)^4\,.$$
In combination with~(\ref{rectangle_crossing}), this yields
$$\Proba_p\left(\mathcal{E}_n\right)
\ \geqslant\ \Proba_p\left(\mathcal{F}_n\right)
\ \geqslant\ \Proba_p\left(L_{\Ent{n/2},\,2n}\right)^4
\ \geqslant\ \left(1-2C_2ne^{-C_3\Ent{n/2}}\right)^4\ \geqslant\ 1-8C_2ne^{-C_3\Ent{n/2}}\,.$$
Yet if the event~$\mathcal{E}_n$ occurs, then all the vertices of~$\Lambda(n)$ which are connected by an open path to the boundary of~$\Lambda(2n)$ must be connected to each other inside~$\Lambda(2n)$, which implies that~$\abs{C_{max}\left(\Lambda(2n)\right)}\geqslant\abs{C_\infty\cap\Lambda(n)}$. Therefore, we have the inclusion
$$\mathcal{E}_n\cap\acc{\abs{C_{\infty}\cap\Lambda(n)} > \frac{\theta(p)n^2}{2}}\ \subset\ \acc{\big|\,C_{max}\left(\Lambda(2n)\right)\big|>\frac{\theta(p)n^2}{2}}\,.$$
Considering complementary events leads to
\begin{align*}
\Proba_p\left(\abs{C_{max}\left(\Lambda(2n)\right)}\leqslant\frac{\theta(p)n^2}{2}\right)
\ &\leqslant\ 1-\Proba_p\left(\mathcal{E}_n\right)+\Proba_p\left(\abs{C_{\infty}\cap\Lambda(n)}\leqslant  \frac{\theta(p)n^2}{2}\right)\\
\ &\leqslant\ 8C_2ne^{-C_3\Ent{n/2}}+e^{-Ln}\\
\ &\leqslant\ e^{-L'n}
\end{align*}
for a certain constant~$L'>0$, which concludes the proof.
\end{proof}

We now briefly explain how to deduce lemma~\ref{Cmax_majo_surcritique} from lemma~\ref{CmaxMajoStronger}:

\begin{proof}[Proof of lemma~\ref{Cmax_majo_surcritique}]
We divide the box~$\Lambda(n)$ into smaller boxes of side
$$N_n\ =\ \Ceil{\left(\frac{8An^a}{\theta(p)}\right)^{1/d}}\,.$$
The box~$\Lambda(n)$ contains at least~$\Ent{n/N_n}^d$ disjoint boxes of side~$N_n$, so that we have
$$\Proba_p\Big(\abs{C_{max}\big(\Lambda(n)\big)}<An^a\Big)
\ \leqslant\ \Proba_p\Big(\abs{C_{max}\big(\Lambda(N_n)\big)}<An^a\Big)^{\Ent{n/N_n}^d}
\ \leqslant\ \Proba_p\left(\abs{C_{max}\big(\Lambda(N_n)\big)}<\frac{\theta(p)N_n^d}{8}\right)^{\Ent{n/N_n}^d}\,,$$
which implies that
\begin{multline*}
\limsup\limits_{n\rightarrow\infty}\,\frac{1}{n^{d-a/d}}\ln\Proba_p\Big(\abs{C_{max}\big(\Lambda(n)\big)}<An^a\Big)
\ \leqslant\ \limsup\limits_{n\rightarrow\infty}\,\frac{n^d}{N_n^dn^{d-a/d}}\ln\Proba_p\left(\abs{C_{max}\big(\Lambda(N_n)\big)}<\frac{\theta(p)N_n^d}{8}\right)\\
\ =\ \left(\frac{\theta(p)}{8A}\right)^{1/d}\limsup\limits_{N\rightarrow\infty}\,\frac{1}{N^{d-1}}\ln\Proba_p\left(\abs{C_{max}\big(\Lambda(N)\big)}<\frac{\theta(p)N^d}{8}\right)
\ <\ 0\,,
\end{multline*}
where the last inequality comes from lemma~\ref{CmaxMajoStronger}.
To obtain the lower bound, we divide the box~$\Lambda(n)$ into boxes of side~$N_n=\Ceil{(An^a)^{1/d}}-1$, which all contain strictly less than~$An^a$ vertices, and we consider the event that all the edges between two neighbouring boxes are closed. This leads to
$$\ln\Proba_p\Big(\abs{C_{max}\big(\Lambda(n)\big)}<An^a\Big)
\ \geqslant\ d N_n^{d-1} \Ceil{\frac{n}{N_n}}^d\ln(1-p)
\ \eqninfty\ \left(\frac{\theta(p)}{8A}\right)^{1/d}d\ln(1-p)n^{d-a/d}\,,$$
which shows that~$d-a/d$ is indeed the correct exponent.
\end{proof}

\subsection{Lower bound on the partition function}
We show here the following inequality on the normalization constant~$Z_n$ of our model:
\label{section_Cmax_mino_Zn}
\begin{lemma}
\label{Cmax_lemme_minoration_Z_n}
For any real number~$a$ such that~$0<a<d$, we have
$$\liminf\limits_{n\rightarrow\infty}\,\frac{\ln Z_n}{(\ln n)n^{a(d-1)/d}}\ >\ -\infty\,.$$
\end{lemma}

\begin{proof}
As explained in the introduction, we define a monotone coupling of the probability distributions~$\Proba_{\varphi_n(t)}$ for~\smash{$t\in\acc{0,\,\ldots,\,n^d}$}.
\\

\souspreuve{Construction of the coupling:}
Write~$\En = \acc{e_1,\,\ldots,\,e_r}$ with~$r=\abs{\En}$, and consider a collection of i.i.d.\ random variables
$$\left(X_{t,e}\right)_{t\in\acc{0,\,\ldots,\,n^d-1},\,e\in \En}$$
with Bernoulli law of parameter~$\exp(-1/n^a)$. For~$t_0\in\acc{0,\,\ldots,\,n^d}$, define a random configuration
$$\omega(t_0)\ :\ e\in \En\ \longmapsto\ \min\limits_{0\leqslant t<t_0}\,X_{t,e}\,.$$
Hence, for~$t_0\in\acc{0,\,\ldots,\,n^d}$ and~$e\in\En$, we see that
$$\Proba\Big(\omega(t_0)(e)=1\Big)
\ =\ \prod_{t=0}^{t_0-1}{\Proba\big(X_{t,e}=1\big)}
\ =\ \exp\left(-\frac{t_0}{n^a}\right)
\ =\ \varphi_n(t_0)\,,$$
therefore the configuration~$\omega(t_0)$ has distribution~$\Proba_{\varphi_n(t_0)}$. What's more, configurations are coupled in such a way that
$$\mathbb{1}_{\En}=\omega(0)\ \geqslant\ \omega(1)\ \geqslant\ \cdots\ \geqslant\ \omega(n^d)\,.$$
When going from the configuration~$\omega(t)$ to the configuration~$\omega(t+1)$, a certain number or edges are closed (these are the edges~$e$ such that~$\omega(t)(e)=1$ and~$X_{t,e}=0$). In order to control the edge closures one by one, we define intermediate configurations. For~$t\in\acc{0,\,\ldots,\,n^d-1}$ and~$s_0\in\acc{0,\,\ldots,\,r}$, we set
$$\omega(t,\,s_0)\ :\ e_s\in \En\ \longmapsto\ 
\left\{
\begin{aligned}
&\omega(t+1)(e_s)\ &\text{if }s\leqslant s_0\,,\\
&\omega(t)(e_s)\ &\text{otherwise.}
\end{aligned}
\right.$$
In this way, we have~$\omega(t,\,0)=\omega(t)$ and for~$s\geqslant 1$, the configuration~$\omega(t,\,s)$ is obtained from the configuration~$\omega(t,\,s-1)$ by closing the edge~$e_s$ if~$X_{t,e_s}=0$, and by keeping everything unchanged if~$X_{t,e_s}=1$. For~$s=r=\abs{\En}$, all edges have been updated, so~$\omega(t,\,r)=\omega(t+1)$. The configurations are therefore coupled in such a way that
$$(t,\,s)\leqslant(t',\,s')\ \Longrightarrow\ \omega(t,\,s)\geqslant\omega(t',\,s')\,,$$
where we use the lexicographic order on~$\{0,\,\ldots,\,n^{d}-1\}\times\acc{0,\,\ldots,\,r}$.
As we have shown in~(\ref{Cmax_expression_Z_n}), the partition function~$Z_n$ is equal to the probability that the non-increasing function~\smash{$t\mapsto\abs{C_{max}\big(\omega(t)\big)}$} admits a fixed point.
Thus, we now look for an instant~$t=T$ situated before this function goes under the first bisector, and we will study what is needed on the variables~$X_{t,e}$ for this function to actually cross the bisector at the instant~$t=T+2$.
\\

\souspreuve{Definition of the instant~$T$:}
Still considering the lexicographic order, we define a pair of random variables
$$(T,\,S)\ =\ \min\Big\{\,
(t,\,s)\in\acc{0,\,\ldots,\,n^{d}-2}\times\acc{0,\,\ldots,\,r}\quad :\quad
\exists e\in \En\quad \abs{C_{max}\big(\omega(t,\,s)_e\big)}\leqslant t+2\,\Big\}\,.$$
This minimum is well defined because one always has~\smash{$\abs{C_{max}\big(\omega(n^d-2,\,0)\big)}\leqslant n^d$}. In addition, for every~$(t_0,\,s_0)$, the event~$\acc{(T,\,S)=(t_0,\,s_0)}$ only depends on the variables~$X_{t,\,e_s}$ for~$(t,\,s)\leqslant(t_0,\,s_0)$, which means that~$(T,\,S)$ is a stopping time for the filtration generated by the variables~$X_{t,e_s}$.
Also, closing one single edge cannot divide the size of the largest cluster by more than two, whence
\begin{equation}
\label{Cmax_majorationBS}
\abs{C_{max}\big(\omega(T,\,S)\big)}\ \leqslant\ 2(T+2)\,.
\end{equation}
Let us prove that we also have
\begin{equation}
\label{Cmax_minorationBS}
\abs{C_{max}\big(\omega(T,\,S)\big)}\ \geqslant\ T+2\,.
\end{equation}
We distinguish several cases.

\point If~$S\geqslant 1$, then the minimality of~$(T,\,S)$ ensures that, for all~$e\in\En$,
$$\abs{C_{max}\big(\omega(T,\,S-1)_{e}\big)}\ >\ T+2\,.$$
Yet the configuration~$\omega(T,\,S)$ is obtained from~$\omega(T,\,S-1)$ by closing at most one edge, whence (\ref{Cmax_minorationBS}).

\point If~$S=0$ and~$T\geqslant 1$, then,~$(T,\,S)$ being minimal, we have that 
$$\abs{C_{max}\big(\omega(T-1,\,r)\big)}\ >\ T-1+2\ =\ T+1\,.$$
The configurations~$\omega(T-1,\,r)$ and~$\omega(T,\,0)$ being identical, inequality (\ref{Cmax_minorationBS}) is also satisfied.

\point The case~$(T,\,S)=(0,0)$ does not happen because all edges are open in the configuration~$\omega(0,0)$.

\smallskip
\noindent We build next a happy event, which implies the existence of the desired fixed point.
\\

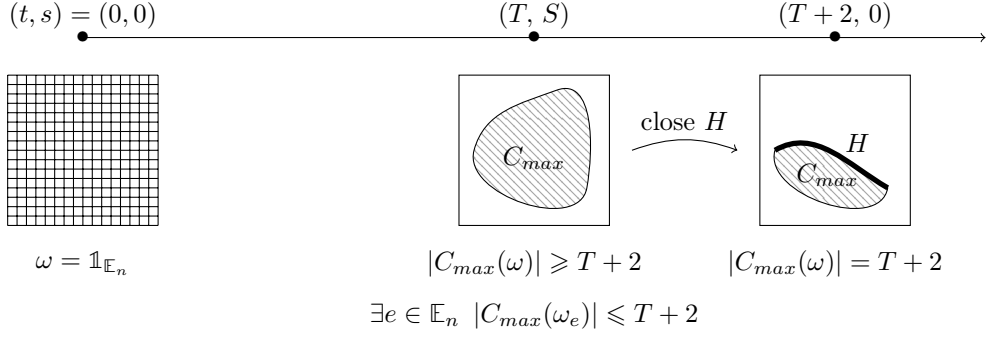
\begin{figure}
\begin{center}
\begin{tikzpicture}
\draw[->] (0,0) node{$\bullet$} node[above]{$(t,s)=(0,0)$} to (6,0) node{$\bullet$} node[above]{$(T,\,S)$} to (10,0) node{$\bullet$}node[above]{$(T+2,\,0)$} to (12,0);
\draw (-1,-2.5) grid[step=0.125] (1,-0.5);
\draw (0,-3) node{$\omega=\mathbb{1}_{\En}$};
\draw (5,-2.5) --++ (2,0) --++ (0,2) --++ (-2,0) -- cycle;
\draw (6,-3) node{$\abs{C_{max}(\omega)}\geqslant T+2$};
\draw (6,-3.7) node{$\exists e\in\En\ \abs{C_{max}(\omega_e)}\leqslant T+2$};
\filldraw[pattern = north west lines, pattern color=gray!70] (5.2,-1.5) to[out=75, in=200] (6.3,-0.7) to[out=20,in=80] (6.7,-2) to[out=-100, in=-105] (5.2,-1.5);
\draw (6,-1.6) node{$C_{max}$};
\draw (10,-3) node{$\abs{C_{max}(\omega)}=T+2$};
\draw[->] (7.3,-1.5) to[bend left=20] node[midway, above]{close~$H$} (8.7,-1.5);
\draw (9,-2.5) --++ (2,0) --++ (0,2) --++ (-2,0) -- cycle;
\filldraw[pattern = north west lines, pattern color=gray!70] (9.2,-1.5) to[out=30, in=150] (10.7,-2) to[out=-100, in=-105] (9.2,-1.5);
\draw[line width=2pt] (9.2,-1.5) to[out=30, in=150] (10.7,-2);
\draw (10.3,-1.4) node{$H$};
\draw (9.9,-1.83) node{$C_{max}$};
\end{tikzpicture}
\end{center}
\caption{Sketch of the proof: if~$\mathcal{E}$ occurs, i.e., between the instants~$(T,\,S)$ and~$(T+2,\,0)$, the edges~$H$ are closed but no other edges of~$C_{max}$ is closed, then the largest cluster in the configuration~$\omega(T+2,0)$ contains~$T+2$ vertices.}
\end{figure}

\souspreuve{Construction of the happy event:}
Let~$(V,\,E)$ be the graph associated to the largest cluster in~$\omega(T,\,S)$, that is to say~\smash{$V=C_{max}\big(\omega(T,\,S)\big)$} and~$E$ is the set of the edges between two vertices of~$V$ which are open in~$\omega(T,\,S)$.
Given~(\ref{Cmax_minorationBS}), it follows from lemma~\ref{lemme_boucher} that there exists a (random) set of edges
$$H\ =\ H\big(T,\,\omega(T,\,S)\big)\ \subset\ E\,,$$
satisfying
\begin{equation}
\label{Cmax_cardinal_H}
\abs{H}\ \leqslant\ K\abs{V}^{\frac{d-1}{d}}
\end{equation}
and such that the largest connected component of the graph~$(V,\,E\backslash H)$ contains exactly~$T+2$ vertices. Note that we have defined~$H=H(T,\,\omega(T,\,S))$ as a deterministic function of the variables~$T$ and~$\omega(T,\,S)$, this will be useful later. The existence of an edge~$e\in \En$ such that
$$\abs{C_{max}\big(\omega(T,\,S)_e\big)}\ \leqslant\ T+2$$
entails that, in~$\omega(T,\,S)$, there is at most one cluster containing strictly more than~$T+2$ vertices. Thus, closing the edges of~$H$ is enough to ensure that the remaining largest cluster contains exactly~$T+2$ vertices, i.e.,
$$\abs{C_{max}\big(\omega(T,\,S)_{H}\big)}\ =\ T+2\,.$$
Hence, closing the edges of~$H$ and no other edge of~\smash{$\Edge{C_{max}\big(\omega(T,\,S)\big)}$} between the instants~$(T,\,S)$ and~$(T+2,\,0)$ ensures that~\smash{$\abs{C_{max}\big(\omega(T+2)\big)}=T+2$}. However, the edges~$e_s\in H$ are not necessarily labeled with numbers~$s>S$. It is therefore not generally possible to close all the edges of~$H$ between the instants~$(T,\,S)$ and~$(T+1,\,0)$. For this reason, the event we consider is the one in which no edge of~\smash{$\Edge{C_{max}\big(\omega(T,\,S)\big)}$} is closed between~$(T,\,S)$ and~$(T+1,\,0)$, and the edges of~\smash{$\Edge{C_{max}\big(\omega(T,\,S)\big)}$} which are closed between~$(T+1,\,0)$ and~$(T+2,0)$ are precisely the edges of~$H$, that is to say
$$\mathcal{E}\ =\ \acc{
\begin{aligned}
&\forall s>S\qquad e_s\in \Edge{C_{max}\big(\omega(T,\,S)\big)}\ \Rightarrow\ X_{T,\,e_s}=1\\
&\forall e\in H\qquad X_{T+1,\,e}=0\\
&\forall e\in \Edge{C_{max}\big(\omega(T,\,S)\big)}\backslash H\qquad X_{T+1,\,e}=1
\end{aligned}
}\,.$$
If this event occurs, then in~$\omega(T+2)$, all the edges of~$H$ are closed, the other edges of~\smash{$\Edge{C_{max}\big(\omega(T,\,S)\big)}$} which were open in the configuration~$\omega(T,\,S)$ remain open, and all the other clusters contain at most~$T+2$ vertices, whence
$$\mathcal{E}\ \subset\ \Big\{\,\abs{C_{max}\big(\omega(T+2)\big)}\ =\ \abs{C_{max}\big(\omega(T,\,S)_{H}\big)}\ =\ T+2\,\Big\}\,.$$
\\

\souspreuve{Conditional probability of the happy event:}
Coming back to the expression~(\ref{Cmax_expression_Z_n}) of the partition function, we find that
\begin{equation}
\label{Cmax_Zn_E}
Z_n\ \geqslant\ \Proba\Big(\,\big|\,C_{max}\big(\omega(T+2)\big)\big|\ =\ T+2\,\Big)
\ \geqslant\ \Proba\left(\mathcal{E}\right)\,.
\end{equation}
Let~\smash{$(t_0,\,s_0)\in\big\{0,\,\ldots,\,n^d-2\big\}\times\acc{0,\,\ldots,\,r}$} and~$\omega_0:\En\rightarrow\acc{0,1}$ be such that
$$\Proba\big(\mathcal{C}_{t_0,\,s_0,\,\omega_0}\big)\ >\ 0
\qquadou
\mathcal{C}_{t_0,\,s_0,\,\omega_0}\ =\ \Big\{\,(T,\,S)=(t_0,\,s_0)\quadet\omega(T,\,S)=\omega_0\,\Big\}\,.$$
Having defined~$H$ as a deterministic function of~$T$ and~$\omega(T,\,S)$, we can consider the event
$$\widetilde{\mathcal{E}}_{t_0,\,s_0,\,\omega_0}\ =\ \acc{
\begin{aligned}
&\forall s>s_0\qquad e_s\in \Edge{C_{max}\left(\omega_0\right)}\Rightarrow X_{t_0,\,e_s}=1\\
&\forall e\in H(t_0,\,\omega_0)\qquad X_{t_0+1,\,e}=0\\
&\forall e\in \Edge{C_{max}\left(\omega_0\right)}\backslash H(t_0,\,\omega_0)\qquad X_{t_0+1,\,e}=1
\end{aligned}}\,,$$
which satisfies
\begin{equation}
\label{Cmax_proba_conditionnelle1}
\Proba\big(\mathcal{E}\,\big|\,\mathcal{C}_{t_0,\,s_0,\,\omega_0}\big)\ =\ \Proba\left(\widetilde{\mathcal{E}}_{t_0,\,s_0,\,\omega_0}\,\big|\,\mathcal{C}_{t_0,\,s_0,\,\omega_0}\right)\,.
\end{equation}
Now note that this event~$\widetilde{\mathcal{E}}_{t_0,\,s_0,\,\omega_0}$ depends only on the variables~$X_{t,\,e_s}$ with~$(t,\,s)>(t_0,\,s_0)$, whereas the event~$\mathcal{C}_{t_0,\,s_0,\,\omega_0}$ depends only on the variables~$X_{t,\,e_s}$ with~$(t,\,s)\leqslant(t_0,\,s_0)$. Thus, these two events are independent of each other, which allows us to write
\begin{align*}
\Proba\left(\widetilde{\mathcal{E}}_{t_0,\,s_0,\,\omega_0}\,\big|\,\mathcal{C}_{t_0,\,s_0,\,\omega_0}\right)
\ &=\ \Proba\left(\widetilde{\mathcal{E}}_{t_0,\,s_0,\,\omega_0}\right)\\
\ &=\ \prod_{\substack{s>s_0\\ e_s\in\Edge{C_{max}(\omega_0)}}}{\Proba\big(\,X_{t_0,\,e_s}=1\,\big)}\times\prod_{e\in H(t_0,\,\omega_0)}{\Proba\big(\,X_{t_0+1,\,e}=0\,\big)}\\
&\qquad\qquad\qquad\qquad\times\prod_{e\in\Edge{C_{max}(\omega_0)}\backslash H(t_0,\,\omega_0)}{\Proba\big(\,X_{t_0+1,\,e}=1\,\big)}\\
\ &\geqslant\ \left(e^{-1/n^a}\right)^{2\left|\Edge{C_{max}\left(\omega_0\right)}\right|}\left(1-e^{-1/n^a}\right)^{\abs{H(t_0,\,\omega_0)}}\,.
\end{align*}
Combining this with~(\ref{Cmax_proba_conditionnelle1}) yields
\begin{equation}
\label{Cmax_proba_conditionnelle2}
\Proba\big(\mathcal{E}\,\big|\,(T,\,S,\,\omega(T,\,S))\big)
\ \geqslant\ \left(e^{-1/n^a}\right)^{2\left|\Edge{C_{max}\left(\omega(T,\,S)\right)}\right|}\left(1-e^{-1/n^a}\right)^{\abs{H(T,\,\omega(T,\,S))}}\,.
\end{equation}
Yet, according to~(\ref{Cmax_majorationBS}), we have
$$\big|\,\Edge{C_{max}\left(\omega(T,\,S)\right)}\big|\ \leqslant\ d\abs{C_{max}(\omega(T,\,S))}\ \leqslant\ 2d(T+2)\,.$$
Furthermore, by convexity of~$x\mapsto e^{-x}$, we get
$$1-e^{-1/n^a}\ \geqslant\ \frac{1}{n^a}\left(1-e^{-1}\right)\ \geqslant\ \frac{1}{2n^a}\,.$$
In addition, combining~(\ref{Cmax_cardinal_H}) and~(\ref{Cmax_majorationBS}) leads to
$$\abs{H}\ \leqslant\ K\abs{C_{max}(\omega(T,\,S))}^{\frac{d-1}{d}}\ \leqslant\ K\big(2(T+2)\big)^{\frac{d-1}{d}}\ \leqslant\ 2K\big(T+2\big)^{\frac{d-1}{d}}\,.$$
Plugging the previous inequalities in equation~(\ref{Cmax_proba_conditionnelle2}), we obtain
$$\Proba\big(\mathcal{E}\,\big|\,(T,\,S,\,\omega(T,\,S))\big)
\ \geqslant\ \exp\left(-\frac{4d\left(T+2\right)}{n^a}\right)\left(\frac{1}{2n^a}\right)^{2K\left(T+2\right)^{\frac{d-1}{d}}}\,.$$
We take the conditional expectation with respect to~$T$, and we deduce that
\begin{equation}
\label{Cmax_E_sachant_T}
\Proba\big(\mathcal{E}\,\big|\, T\big)
\ \geqslant\ \exp\left(-\frac{4d\left(T+2\right)}{n^a}\right)\left(\frac{1}{2n^a}\right)^{2K\left(T+2\right)^{\frac{d-1}{d}}}\,.
\end{equation}
\\

\souspreuve{Upper bound on~$T$:}
Next, we need a control on~$T$ in order to obtain a lower bound on~$\Proba(\mathcal{E})$. Define
\begin{equation}
\label{defTau}
\tau_n^+\ =\ \Ceil{n^a\left(-\ln\left(\frac{p_c}{2}\right)\right)}\,.
\end{equation}
Lemma~\ref{Cmax_majo_souscritique} implies that
$$\Proba_{p_c/2}\big(\abs{C_{max}}\leqslant \tau_n^+\big)\ \cvninfty\ 1\,.$$
This entails that, for~$n$ large enough,
$$\Proba_{p_c/2}\big(\abs{C_{max}}\leqslant \tau_n^+\big)\ \geqslant\ \demi\,.$$
Given that 
$$\varphi_n\left(\tau_n^+\right)
\ \leqslant\ \varphi_n\left(n^a\left(-\ln\left(\frac{p_c}{2}\right)\right)\right)
\ =\ \frac{p_c}{2}\,,$$
we deduce that, for~$n$ large enough,
$$\Proba\big(T\leqslant \tau_n^+\big)
\ \geqslant\ \Proba\big(\abs{C_{max}\left(\omega(\tau_n^+)\right)}\leqslant \tau_n^++2\big)\\
\ =\ \Proba_{\varphi_n(\tau_n^+)}\big(\abs{C_{max}}\leqslant \tau_n^++2\big)\\
\ \geqslant\ \Proba_{p_c/2}\big(\abs{C_{max}}\leqslant \tau_n^++2\big)\\
\ \geqslant\ \demi\,.$$
Therefore, we can find~$\kappa\geqslant 2$ such that, for all~$n\geqslant 1$,
\begin{equation}
\label{Cmax_majo_T}
\Proba\big(T\leqslant \kappa n^a\big)\ \geqslant\ \demi\,.
\end{equation}
\\

\souspreuve{Conclusion:}
Combining~(\ref{Cmax_majo_T}) with~(\ref{Cmax_E_sachant_T}) gives
\begin{align*}
\Proba\left(\mathcal{E}\right)
\ &\geqslant\ \Proba\big(\mathcal{E}\cap\acc{T\leqslant\kappa n^a}\big)\\
\ &=\ \Proba\big(T\leqslant \kappa n^a\big)\Proba\big(\mathcal{E}\,\big|\, T\leqslant \kappa n^a\big)\\
\ &\geqslant\ \demi\exp\left(-\frac{4d(\kappa n^a+2)}{n^a}-2K(\kappa n^a+2)^{\frac{d-1}{d}}\ln(2n^a)\right)\\
\ &\geqslant\ \demi\exp\left(-8d\kappa-4K\kappa(\ln 2)n^{a(d-1)/d}-4 K\kappa a(\ln n) n^{a(d-1)/d}\right)\,,
\end{align*}
where we have used that~$2\leqslant \kappa n^a$.
Now recall inequality~(\ref{Cmax_Zn_E}) to deduce that
$$\liminf\limits_{n\rightarrow\infty}\,\frac{\ln Z_n}{(\ln n)n^{a(d-1)/d}}
\ \geqslant\ -4 K\kappa a\ >\ -\infty\,,$$
which is the required lower bound.
\end{proof}

\subsection{Proof of the convergence result}
\label{section_Cmax_preuve_resultat}

We are now in position to prove case~($i$) of theorem~\ref{thm_clusters_convergence}.

\begin{proof}[Proof of theorem~\ref{thm_clusters_convergence}, case~($i$)]
Let~$\varepsilon>0$ and~$a\in(0,d)$.
Given that~$a(d-1)/d<a$, the lower bound on~$Z_n$ we have obtained in lemma~\ref{Cmax_lemme_minoration_Z_n} implies that~\smash{$\liminfn Z_n/n^a\geqslant 0$}.
Combining this with the result of lemma~\ref{Cmax_majo_souscritique} and plugging it into the inequality~(\ref{munsous}) leads to
\begin{equation}
\label{prthm1}
\limsupn\,\frac{1}{n^a}\,\ln\mu_n\big(p_n<p_c-\varepsilon\big)
\ \leqslant\ 
\limsupn\,\frac{1}{n^a}\,\ln\Proba_{p_c-\varepsilon}\Big(\abs{C_{max}}>\big(-\ln(p_c-\varepsilon)\big)n^a\Big)
\ <\ 0\,.
\end{equation}
Similarly, using lemma~\ref{Cmax_majo_surcritique}, inequality~(\ref{munsur}) and the fact that~$a(d-1)/d<d-a/d$, we get
\begin{equation}
\label{prthm2}
\limsupn\,\frac{1}{n^{d-a/d}}\,\ln\mu_n\big(p_n>p_c+\varepsilon\big)
\ \leqslant\ 
\limsupn\,\frac{1}{n^{d-a/d}}\,\ln\Proba_{p_c+\varepsilon}\Big(\abs{C_{max}}<\big(-\ln(p_c+\varepsilon)\big)n^a\Big)
\ <\ 0\,.
\end{equation}
It remains to show that the exponent~$v=a\wedge(d-a/d)$ is optimal.
To this end, we go back to our computation~(\ref{munsurexact}) and we recall that~$Z_n$ was expressed as a probability in~(\ref{Cmax_expression_Z_n}), whence~$Z_n\leqslant 1$.
Therefore, with~$t_n^+$ as defined in~(\ref{defTn}) we have
$$\mu_n\big(p_n> p_c+\varepsilon\big)
\ =\ \frac{1}{Z_n}\sum_{t=0}^{t_n^+-1}\Proba_{\varphi_n(t)}\Big(\abs{C_{max}}=t \Big)
\ \geqslant\ \sum_{t=0}^{t_n^+-1}\Proba_{\varphi_n(t)}\Big(\abs{C_{max}}=t \Big)\,.$$
Using the notations of the last subsection, this becomes
\begin{equation}
\label{eq6397}
\mu_n\big(p_n> p_c+\varepsilon\big)
\ \geqslant\ \Proba\Big(\,\exists t\in\acc{0,\,\ldots,\,t_n^+-1}\ :\ \abs{C_{max}\big(\omega(t)\big)}=t\,\Big)
\ \geqslant\ \Proba\Big(\,\mathcal{E}\,\cap\,\acc{T\leqslant t_n^+-3}\,\Big)\,,
\end{equation}
since the occurrence of the event~$\mathcal{E}$ implies that~\smash{$\abs{C_{max}\big(\omega(T+2)\big)}=T+2$}.
As we did in the proof of lemma~\ref{Cmax_lemme_minoration_Z_n}, we can write
$$\Proba\big(T\leqslant t_n^+-3\big)
\ \geqslant\ \Proba\Big(\,\abs{C_{max}\big(\omega(t_n^+-3)\big)}\leqslant (t_n^+-3)+2\,\Big)
\ =\ \Proba_{\varphi_n(t_n^+-3)}\Big(\abs{C_{max}}<t_n^+\Big)\,.$$
Now notice that~$\varphi_n(t_n^+-3)\rightarrow p_c+\varepsilon$, whence~$\varphi_n(t_n^+-3)\leqslant p_c+2\varepsilon$ for~$n$ large enough.
Thus, for~$n$ large enough, we have
$$\Proba\big(T\leqslant t_n^+-3\big)\ \geqslant\ \Proba_{p_c+2\varepsilon}\Big(\abs{C_{max}}<t_n^+\Big)\,.$$
Plugging this into~(\ref{eq6397}) and using our lower bound~(\ref{Cmax_E_sachant_T}) on the conditional probability of~$\mathcal{E}$ with respect to~$T$ leads to
\begin{align*}
\frac{1}{n^{d-a/d}}\ln\mu_n\big(p_n> p_c+\varepsilon\big)
\ &\geqslant\ -\frac{4d\big(t_n^+-1\big)}{n^{a+d-a/d}}-\frac{2K\big(t_n^+-1\big)^{\frac{d-1}{d}}\ln\big(2n^a\big)}{n^{d-a/d}}+\frac{1}{n^{d-a/d}}\ln\Proba_{p_c+2\varepsilon}\Big(\abs{C_{max}}<t_n^+\Big)\\
\ &=\ O\left(\frac{1}{n^{d-a/d}}\right)+O\left(\frac{\ln n}{n^{d-a}}\right)+\frac{1}{n^{d-a/d}}\ln\Proba_{p_c+2\varepsilon}\Big(\abs{C_{max}}<t_n^+\Big)\,.
\end{align*}
Taking the infinimum limit and using the lower bound given by lemma~\ref{Cmax_majo_surcritique}, we obtain
\begin{equation}
\label{prthm3}
\liminfn\,\frac{1}{n^{d-a/d}}\ln\mu_n\big(p_n> p_c+\varepsilon\big)
\ >\ -\infty\,.
\end{equation}
To handle the other tail, we choose~$a'$ such that
$$a\ <\ a'\ <\ d\wedge\frac{da}{d-1}\,,$$
we define~$t_n^-=\Ent{n^a\big(-\ln(p_c-\varepsilon)\big)}$ and we write
\begin{align*}
\mu_n\big(p_n< p_c-\varepsilon\big)
\ &\geqslant\ \Proba\Big(\,\exists t\in\acc{t_n^-+1,\,\ldots,\,n^d}\ :\ \abs{C_{max}\big(\omega(t)\big)}=t\,\Big)\\
\ &\geqslant\ \Proba\Big(\,\mathcal{E}\,\cap\,\acc{T\geqslant t_n^--1}\,\Big)
\ \geqslant\ \Proba\Big(\,\mathcal{E}\,\cap\,\acc{t_n^--1\leqslant T\leqslant n^{a'}}\,\Big)\,.\numberthis\label{eq451}
\end{align*}
It follows from~(\ref{Cmax_majorationBS}) that
$$\Proba\big(T\geqslant t_n^--1\big)
\ \geqslant\ \Proba\Big(\,\abs{C_{max}\big(\omega(t_n^--1)\big)}>2\big(t_n^-+1\big)\,\Big)
\ \geqslant\ \Proba_{p_c-2\varepsilon}\Big(\,\abs{C_{max}}>2\big(t_n^-+1\big)\,\Big)\,.$$
Similarly, it follows from~(\ref{Cmax_minorationBS}) that
$$\Proba\big(T>n^{a'}\big)
\ \leqslant\ \Proba\Big(\,\abs{C_{max}\big(\omega\big(\big\lfloor n^{a'}\big\rfloor\big)\big)}\geqslant \big\lfloor n^{a'}\big\rfloor+2\,\Big)
\ \leqslant\ \Proba\Big(\,\abs{C_{max}\big(\omega\big(\big\lfloor n^{a'}\big\rfloor\big)\big)}>n^{a'}\,\Big)
\ \leqslant\ \Proba_{p_c/2}\Big(\,\abs{C_{max}}>n^{a'}\,\Big)\,.$$
Therefore, we have
\begin{align*}
\Proba\big(t_n^--1\leqslant T\leqslant n^{a'}\big)
\ &=\ \Proba\big(T\geqslant t_n^--1\big)-\Proba\big(T>n^{a'}\big)\\
\ &\geqslant\ \Proba_{p_c-2\varepsilon}\Big(\,\abs{C_{max}}>2\big(t_n^-+1\big)\,\Big)
-\Proba_{p_c/2}\Big(\,\abs{C_{max}}>n^{a'}\,\Big)\\
\ &\geqslant\ e^{-Cn^a}-e^{-C'n^{a'}}\ \geqslant\ \frac{e^{-Cn^a}}{2}\,,
\end{align*}
with~$C,\,C'>0$, using the exponential estimate of lemma~\ref{Cmax_majo_souscritique}.
Plugging this into~(\ref{eq451}) and using again~(\ref{Cmax_E_sachant_T}), we now obtain
\begin{align*}
\frac{1}{n^a}\ln\mu_n\big(p_n< p_c-\varepsilon\big)
\ &\geqslant\ -C-\frac{\ln 2}{n^a}-\frac{4d\big(n^{a'}+2\big)}{n^{2a}}-\frac{2K\big(n^{a'}+2\big)^{\frac{d-1}{d}}}{n^a}\ln\big(2n^a\big)\\
\ &=\ -C-\frac{\ln 2}{n^a}+O\left(\frac{1}{n^{2a-a'}}\right)+O\left(\frac{\ln n}{n^{a-a'+a'/d}}\right)
\ \cvninfty\ -C\ >\ -\infty\,.\numberthis\label{prthm4}
\end{align*}
The first case of theorem~\ref{thm_clusters_convergence} then follows from~(\ref{prthm1}),~(\ref{prthm2}),~(\ref{prthm3}) and~(\ref{prthm4}).
\end{proof}

\subsection{A variant on the torus}
One can define a similar model on the torus of side~$n$, which boils down to considering periodic boundary conditions on the box~$\Lambda(n)$. Clusters on the torus are at least as big as in the box, so the exponential decay in the supercritical phase for the model defined on the torus immediately follows from lemma~\ref{Cmax_majo_surcritique}. The analog of lemma~\ref{Cmax_majo_souscritique} can be proved by noting that the size of the cluster of the origin in the torus is stochastically dominated by the size of the cluster of the origin in a configuration on all~$\Z^d$. The same proof for the lower bound on the partition function applies in the case of the torus, by adapting our geometrical lemma to extend it to subgraphs of the torus. We therefore have the same convergence of~$p_n$ to~$p_c$ when~$n\rightarrow\infty$ for this alternative model.

\section{Proof of case (\textit{ii}) of theorem~\ref{thm_clusters_convergence}}
\label{sectionMn}

We prove here the point~($ii$) of theorem~\ref{thm_clusters_convergence}, namely the case of the model defined with~$F_n=\abs{\mathcal{M}_n}$, where~$\Mn$ is the set of the vertices connected by an open path to the boundary~$\partial\Lambda(n)$ of the box~$\Lambda(n)$.

\subsection{Exponential decay in the subcritical phase}
\label{subMnsous}

Following the same method as for the first model, we start with an upper bound on the law of~$\abs{\mathcal{M}_n}$ in the subcritical regime (the lower bound is straightforward, but we will not need it).
\begin{lemma}
\label{Mn_majo_souscritique}
For any~$a>d-1$, for~$p<p_c$ and~$A>0$, we have the upper bound
$$\limsup\limits_{n\rightarrow\infty}\,\frac{1}{n^a}\ln\Proba_p\Big(\abs{\mathcal{M}_n}>An^a\Big)\ <\ 0\,.$$
\end{lemma}
\begin{proof}
Take~$a>d-1$,~$p<p_c$ and~$A>0$. Write~$\partial\Lambda(n)=\acc{x_1,\,\ldots,\,x_t}$ with~$t=\abs{\partial\Lambda(n)}$. If~$A$ and~$T$ are two events, then~$A\circ T$ denotes the disjoint occurrence of these two events, which is defined in section~2.3 of~\cite{Grimmett}. Let~$\omega:\En\rightarrow\acc{0,1}$ be a configuration such that~$\abs{\mathcal{M}_n(\omega)}>An^a$. Define, for~$i\in\acc{1,\,\ldots,\,t}$,
$$n_i\ =\ \Big|\,C_{\Lambda(n)}(x_i)\,\backslash\,\bigcup_{j<i}{C_{\Lambda(n)}(x_j)}\,\Big|
\ =\ \left\{\begin{aligned}
&\,0\ \text{if there exists}\ j<i\ \text{such that}\ x_i\connecte x_j\,,\\
&\abs{C_{\Lambda(n)}(x_i)}\text{ otherwise.}
\end{aligned}\right.$$
We have that
$$\sum_{i=1}^t{n_i}\ =\ \abs{\bigcup_{i=1}^t{C_{\Lambda(n)}(x_i)}}\ =\ \abs{\mathcal{M}_n(\omega)}\ >\ An^a\,,$$
and
$$\omega
\ \in\ \acc{\abs{C_{\Lambda(n)}(x_1)}\geqslant n_1}\circ\ldots\circ\acc{\abs{C_{\Lambda(n)}(x_t)}\geqslant n_t}\,.$$
Indeed, if~$n_i=0$, then the event~$\acc{\abs{C_{\Lambda(n)}(x_i)}\geqslant n_i}$ is trivial, whereas if we have~$n_i>0$ and~$n_j>0$ for some~$i\neq j$, then the vertices~$x_i$ and~$x_j$ must belong to disjoint clusters. Whence the inclusion
$$\Big\{\,\abs{\mathcal{M}_n}>An^a\,\Big\}\ \subset\ \bigcup_{\substack{0\leqslant n_1,\,\ldots,\,n_t\leqslant n^d\\ n_1+\cdots+n_t>An^a}}{\acc{\abs{C_{\Lambda(n)}(x_1)}\geqslant n_1}\circ\cdots\circ\acc{\abs{C_{\Lambda(n)}(x_t)}\geqslant n_t}}\,.$$
Note that, for all~$i\in\acc{1,\,\ldots,\,t}$, the event~$\acc{\abs{C_{\Lambda(n)}(x_i)}\geqslant n_i}$ is an increasing event, thus by the BK inequality,
\begin{align*}
\Proba_p\Big(\abs{\mathcal{M}_n}>An^a\Big)
\ &\leqslant\ \sum_{\substack{0\leqslant n_1,\ldots,n_t\leqslant n^d\\ n_1+\cdots+n_t>An^a}}{\prod_{i=1}^t{\Proba_p\Big(\abs{C_{\Lambda(n)}(x_i)}\geqslant n_i\Big)}}\\
\ &\leqslant\ \sum_{\substack{0\leqslant n_1,\ldots,n_t\leqslant n^d\\ n_1+\cdots+n_t>An^a}}{\prod_{i=1}^t{\Proba_p\Big(\abs{C(0)}\geqslant n_i\Big)}}\,.
\end{align*}
Furthermore, according to theorem~6.75 in~\cite{Grimmett}, for~$p<p_c$, there exists a constant~$\lambda(p)>0$ such that, for all~$n\geqslant 1$,
$$\Proba_p\big(\abs{C(0)}\geqslant n\big)\ \leqslant\ e^{-n\lambda(p)}\,,$$
which is also true if~$n=0$.
It follows that
\begin{align*}
\Proba_p\Big(\abs{\mathcal{M}_n}>An^a\Big)
\ &\leqslant\ \sum_{\substack{0\leqslant n_1,\ldots,n_t\leqslant n^d\\ n_1+\cdots+n_t>An^a}}{\prod_{i=1}^t{\exp\big(-\lambda(p)n_i\big)}}\\
\ &\leqslant\ \sum_{0\leqslant n_1,\ldots,n_t\leqslant n^d}{\exp\big(-\lambda(p)An^a\big)}\\
\ &\leqslant\ \left(n^d+1\right)^t\exp\big(-\lambda(p)An^a\big)\\
\ &=\ \exp\left(\abs{\partial\Lambda(n)}\ln(n^d+1)-\lambda(p)An^a\right)\,.
\end{align*}
To conclude, note that
$$\abs{\partial\Lambda(n)}\ln(n^d+1)
\ =\ O\left((\ln n)n^{d-1}\right)
\ =\ o(n^a)\,.$$
This completes the proof of the lemma.
\end{proof}

\subsection{Exponential decay in the supercritical phase}
\label{subMnsur}

We now state a similar exponential decay property in the supercritical regime.
\begin{lemma}
\label{Mn_majo_surcritique}
For all~$a<d$, for~$p>p_c$ and~$A>0$, we have
$$\limsupn\,\frac{1}{n^{d-1}}\ln\Proba_p\Big(\abs{\mathcal{M}_n}<An^a\Big)\ <\ 0\,.$$
\end{lemma}
\begin{proof}
Let~$p>p_c$ and~$A>0$.
As in the proof of lemma~\ref{Cmax_majo_surcritique}, we show that
$$\limsup\limits_{n\rightarrow\infty}\,\frac{1}{n^{d-1}}\ln\Proba_p\left(\abs{\mathcal{M}_n}\leqslant\frac{\theta(p)n^d}{2}\right)\ <\ 0\,.$$
For~$d\geqslant 3$, the result follows from theorem~1.2 of~\cite{PisztoraPGD}, which proves it for~$p$ larger than~$\widehat{p_c}$, which was proved to be equal to~$p_c$ in~\cite{GrimmettMarstrand}. In dimension~$d=2$, the claim follows from theorem~6.1 in~\cite{Wulff2D}.
\end{proof}

\subsection{Lower bound on the partition function}

We now establish a lower bound on the normalization constant~$Z_n$.

\begin{lemma}
\label{Mn_lemme_minoration_Z_n}
For any real~$a$ such that~$d-1<a<d$, we have
$$\liminf_{n\rightarrow\infty}\,\frac{\ln Z_n}{(\ln n)n^{a(d-1)/d}}\ >\ -\infty\,.$$
\end{lemma}

\souspreuve{Heuristics of the proof:} We wish to apply the same technique as in the proof for the case of the largest cluster (section~\ref{section_Cmax_mino_Zn}), by constructing a decreasing coupling between the distributions~$\Proba_{\varphi_n(t)}$ for~$t$ varying from~$0$ (all edges open) to~$n^d$ (almost all edges closed). We monitor the evolution of the variable~$\abs{\mathcal{M}_n}$ until an instant~$t=T'$ when~$\abs{\mathcal{M}_n}$ is of order~$T'$. Then we find a set of edges~$H\subset\En$ whose closure would lead to~$\abs{\mathcal{M}_n}=T'+2$ at the instant~$T'+2$.

The hurdle is that, in order to find such a set~$H$ which is not too big (and thus whose closure is likely enough), we need a control on the size of the clusters which are connected to the boundary of the box at the instant~$T'$.
To obtain such a control, a natural idea is to monitor first the evolution of the size of the clusters connected to the boundary, to wait for an instant~$T$ when these clusters have become small enough, and then to define the instant~$T'$ in a way which ensures that it occurs later than~$T$.
However, unlike the size of the largest cluster~$C_{max}$, which can be at most divided by a factor~$2$ when closing an edge, the size of the largest cluster connected to the boundary can fall drastically with the closure of one edge.
To avoid this, we choose to monitor the size of the largest cluster on the torus, that is to say in the box~$\Lambda(n)$ but with periodic boundary conditions.
This variable has the advantage of being at most halved at each edge closure.

\begin{proof}

\souspreuve{Sketch of the proof:} We first define a decreasing coupling of configurations~$(\omega(t,\,s))_{t,\,s}$ but on the edges of the torus.
We then consider the first instant~$(T,\,S)$ when the largest cluster on the torus contains at most~$2T+3$ vertices.
In what follows, we will reason conditionally on the fact that, at this instant, the largest cluster on the torus touches the boundary of the box.
We will show that, at this instant, we have~$\abs{\mathcal{M}_n(\omega)}\geqslant T+2$.
Next we will construct a second instant~$(T',\,S')\geqslant(T,\,S)$ and a set of edges~$H$ such that, if the only edges of~$\mathcal{M}_n(\omega)$ which are closed between~$(T',\,S')$ and~$(T'+2,\,0)$ are the edges of~$H$, then we have~$\abs{\mathcal{M}_n(\omega(T'+2))}=T'+2$.
We will call this scenario the \guillemets{happy event}, and our aim is to obtain a lower bound on its probability.
To this end, we will show that, with sufficiently high probability, we have~$T'=O(n^a)$, which implies that, from the instant~$(T,\,S)$ onward, any of the clusters on the torus contains at most~$O(n^a)$ vertices.
This control will allow us to show that it is possible to find~$H$ small enough to ensure that the happy event is likely enough.
\\

\begin{figure}[h]
\begin{center}
\begin{tikzpicture}
\draw[->] (-3.5,0) node{$\bullet$} node[above]{$(0,\,0)$} to (0,0) node{$\bullet$} node[above]{$(T,\,S)$} to (5,0) node{$\bullet$} node[above]{$(T',\,S')$} to (9,0) node{$\bullet$} node[above]{$(T'+2,\,0)$} to (11,0);
\draw (-4.5,-2.5) grid[step=0.125] (-2.5,-0.5);
\draw (-3.5,-3) node{$\omega=\mathbb{1}_{\En}$};
\filldraw[pattern=north east lines, pattern color=gray!70] (-1,-0.5) to (-1,-2.5) to (1,-2.5) to (1,-0.5) -- cycle;
\filldraw[fill=white] (-0.1,-0.5) to[out=-20, in=135] (0.6,-0.9) to[out=-45, in=95] (0.8,-1.5) to[out=-85, in=25] (0.2,-2) to[out=205, in=-10] (-0.3,-2.1) to[out=170, in=-85] (-0.7,-1.4) to[out=95, in=200] (-0.3,-0.5) -- cycle;
\filldraw[fill=gray!20!white] (-0.6,-2.5) to[out=60, in=-90] (-0.1,-1.8) to[out=90, in=220] (-0.2,-1.3) to[out=40, in=120] (0.3,-1.2) to[out=-30, in=100] (0.1,-1.7) to[out=-80, in=140] (0.7,-2.5) -- cycle;
\fill[pattern=north east lines, pattern color=gray!70] (-0.6,-2.5) to[out=60, in=-90] (-0.1,-1.8) to[out=90, in=220] (-0.2,-1.3) to[out=40, in=120] (0.3,-1.2) to[out=-30, in=100] (0.1,-1.7) to[out=-80, in=140] (0.7,-2.5) -- cycle;
\draw (0.1,-2.25) node[scale=0.92]{$C_{max}^T$};
\filldraw[fill=gray!20!white] (-0.6,-0.5) to[out=-120,in=160] (-0.1,-0.8) to[out=-20,in=-40] (0.7,-0.5) -- cycle;
\fill[pattern=north east lines, pattern color=gray!70] (-0.6,-0.5) to[out=-120,in=160] (-0.1,-0.8) to[out=-20,in=-40] (0.7,-0.5) -- cycle;
\draw (-1,-2.5) --++ (2,0) --++ (0,2) --++ (-2,0) -- cycle;
\draw (0,-3) node{$\abs{C_{max}^T(\omega)}\leqslant 2T+3$};
\draw (0,-3.7) node{$\abs{\mathcal{M}_n(\omega)}\geqslant \abs{C_{max}^T(\omega)}\geqslant T+2$};
\filldraw[pattern=north east lines, pattern color=gray!70] (4.9,-0.5) to[out=-20, in=135] (5.8,-0.8) to[out=-45, in=95] (5.95,-1.5) to[out=-85, in=25] (5.3,-2.2) to[out=205, in=-80] (5.07,-1.7) to[out=100, in=-30] (5.25,-1.2) to[out=120, in=40] (4.8,-1.3) to[out=220, in=90] (4.96,-1.8) to[out=-90, in=-10] (4.6,-2.3) to[out=170, in=-85] (4.1,-1.4) to[out=95, in=200] (4.7,-0.5) to (4,-0.5) to (4,-2.5) to (6,-2.5) to (6,-0.5) -- cycle;
\draw (5.65,-2.3) node{$\mathcal{M}_n$};
\draw (4,-2.5) --++ (2,0) --++ (0,2) --++ (-2,0) -- cycle;
\draw (5,-3) node{$\abs{\mathcal{M}_n(\omega)}\geqslant T'+2$};
\draw (5,-3.7) node{$\exists e\in\En\ \abs{\mathcal{M}_n(\omega_e)}\leqslant T'+2$};
\draw[->] (6.3,-1.5) to[bend left=20] node[midway, above]{close~$H$} (7.7,-1.5);
\filldraw[pattern=north east lines, pattern color=gray!60] (8.9,-0.5) to[out=-20, in=135] (9.8,-0.8) to[out=-45, in=95] (9.95,-1.5) to[out=-85, in=25] (9.3,-2.2) to[out=205, in=-80] (9.07,-1.7) to[out=100, in=-30] (9.25,-1.2) to[out=230, in=-20] (8.8,-1.3) to[out=220, in=90] (8.96,-1.8) to[out=-90, in=-10] (8.6,-2.3) to[out=170, in=-85] (8.1,-1.4) to[out=95, in=200] (8.7,-0.5) to (8,-0.5) to (8,-2.5) to (10,-2.5) to (10,-0.5) -- cycle;
\draw (9.65,-2.3) node{$\mathcal{M}_n$};
\draw[line width=2pt] (9.25,-1.2) to[out=230, in=-20] (8.8,-1.3);
\draw (9,-1.05) node{$H$};
\draw (8,-2.5) --++ (2,0) --++ (0,2) --++ (-2,0) -- cycle;
\draw (9,-3) node{$\abs{\mathcal{M}_n(\omega)}= T'+2$};
\end{tikzpicture}
\end{center}
\caption{Illustration for the sketch of the proof of lemma~\ref{Mn_lemme_minoration_Z_n}.}
\end{figure}
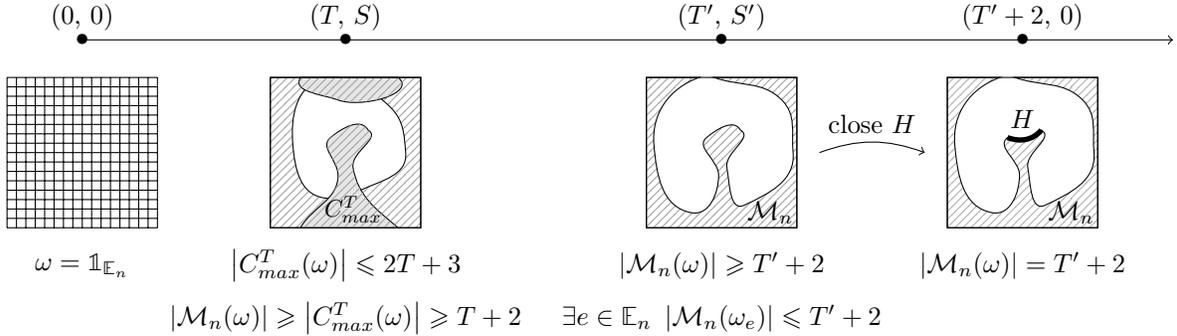

\souspreuve{Construction of the coupling and definition of~$T$:}
Take~$n\geqslant 2$.
We use the same notations and definitions as in the proof of lemma~\ref{Cmax_lemme_minoration_Z_n}, but we now consider configurations on the torus.
To define the torus, write~$p:\Z^d\to\Lambda(n)$ for the projection application, with is such that~$p(x)-x\in n\Z^d$ for every~$x\in\Z^d$.
The torus is the graph whose vertex set is~$\Lambda(n)$, and whose edge set is
$$\En^T\ =\ p\big(\Ed\big)
\ =\ \Big\{\,\big\{p(x),\,p(y)\big\}\ :\ \acc{x,y}\in\Ed\,\Big\}\,,$$
which amounts to adding edges between corresponding vertices on opposite faces of the box.
We then write~$\En^T=\acc{e_1,\,\ldots,\,e_r}$ with~$r=\abs{\En^T}$, and we consider a collection of i.i.d.\ random variables
$$\left(X_{t,e}\right)_{t\in\acc{0,\,\ldots,\,n^d-1},\,e\in \En^T}$$
all distributed with a Bernoulli law of parameter~$\exp(-1/n^a)$.
We set, for~$t_0\in\{0,\,\ldots,\,n^d\}$,
$$\omega(t_0)\ :\ e\in \En^T\ \longmapsto\ \min\limits_{0\leqslant t<t_0}\,X_{t,e}\,,$$
and for~$t\in\{0,\,\ldots,\,n^d-1\}$ and~$s_0\in\acc{0,\,\ldots,\,r}$, we define
$$\omega(t,\,s_0)\ :\ e_s\in \En^T\ \longmapsto\ 
\left\{
\begin{aligned}
&\omega(t+1)(e_s)\ &\text{if }s\leqslant s_0\,,\\
&\omega(t)(e_s)\ &\text{otherwise.}
\end{aligned}
\right.$$
For a configuration~$\omega:\En^T\rightarrow\acc{0,1}$ and~$v\in\Lambda(n)$, we denote by~$C_{\Lambda(n)}^T(v,\,\omega)\subset\Lambda(n)$ the cluster of~$v$ in the configuration~$\omega$ on the torus, that is to say the connected component of the vertex~$v$ in the graph
$$\Big(\Lambda(n),\,\acc{e\in\En^T\,:\,\omega(e)=1}\Big)\,.$$
For any~$\omega:\En^T\rightarrow\acc{0,1}$, we denote by~$C_{max}^T(\omega)$ the largest cluster on the torus in the configuration~$\omega$.
In case of equality between several clusters, we choose one with an arbitrary order on the subsets of~$\Lambda(n)$.
We consider the pair of random variables
$$(T,\,S)\ =\ \min\Big\{\,
(t,\,s)\in\acc{0,\,\ldots,\,n^{d}-2}\times\acc{0,\,\ldots,\,r}\ :\ 
\abs{C_{max}^T\big(\omega(t,\,s)\big)}\,\leqslant\,2t+3\,\Big\}\,,$$
which is well-defined because~\smash{$\abs{C_{max}^T\big(\omega(n^d-2,\,0)\big)}\leqslant n^d$}.
Let us show that, at this instant~$(T,\,S)$, we have
\begin{equation}
\label{Mn_mino_Cn}
\abs{C_{max}^T\big(\omega(T,\,S)\big)}\ \geqslant\ T+2\,.
\end{equation}
We distinguish several cases :

\point If~$S\geqslant 1$ then,~$(T,\,S)$ being minimal, we have~\smash{$\abs{C_{max}^T\big(\omega(T,\,S-1)\big)}\geqslant 2T+4$}. 
To obtain~(\ref{Mn_mino_Cn}), note that closing a single edge can at most divide~\smash{$\abs{C_{max}^T}$} by a factor two.

\point If~$T\neq 0$ and~$S=0$ then, by minimality of~$(T,\,S)$, we have that
$$\abs{C_{max}^T\big(\omega(T-1,\,r)\big)}\ \geqslant\ 2(T-1)+4\ =\ 2T+2\ \geqslant\ T+2\,,$$
which implies inequality~(\ref{Mn_mino_Cn}), because the configurations~$\omega(T-1,\,r)$ and~$\omega(T,0)$ are identical.

\point The case~$(T,\,S)=(0,\,0)$ never occurs because we have~\smash{$\abs{C_{max}^T\big(\omega(0,\,0)\big)}=n^d>3$}.

\noindent We have thus shown that~(\ref{Mn_mino_Cn}) holds.
\\

\souspreuve{Definition of the reference vertex :}
We now order the vertices of~$\Lambda(n)$ in a deterministic way (for instance the lexicographic order) and we denote by~$V$ the vertex of~\smash{$C_{max}^T\big(\omega(T,\,S)\big)$} which is minimal for this order.
Given that
$$\sum_{v\in\Lambda(n)}\Proba(V=v)\ =\ 1\,,$$
we can find a vertex~$v_0\in\Lambda(n)$ such that
$$\Proba(V=v_0)
\ \geqslant\ \frac{1}{\abs{\Lambda(n)}}
\ =\ \frac{1}{n^d}\,.$$
In what follows, we will reason conditionally on the event~$\acc{V=v_0}$.
Until now, everything took place on the torus, which is translation-invariant.
If~$v_0\notin \partial\Lambda(n)$, we can apply a translation on the torus so as to have~$v_0\in\partial\Lambda(n)$, modifying at the same time the orders considered on~$\En^T$, on~$\Lambda(n)$ and on the subsets of~$\Lambda(n)$.
Therefore, without loss of generality, we can assume that~$v_0\in\partial\Lambda(n)$.
Hence, if~$V=v_0$, then we have~\smash{$C_{max}^T\big(\omega(T,\,S)\big)\subset\Mn\big(\omega(T,\,S)\big)$}, whence
\begin{equation}
\label{Mn_mino_Mn_Cn}
V=v_0
\qquadimplique
\abs{\mathcal{M}_n\big(\omega(T,\,S)\big)}\ \geqslant\ 
\abs{C_{max}^T\big(\omega(T,\,S)\big)}\ \geqslant\ T+2\,,
\end{equation}
following~(\ref{Mn_mino_Cn}).
\\

\souspreuve{Construction of the second instant~$T'$:}
We now consider
$$(T',\,S')\ =\ \min\Big\{\,
(t,\,s)\geqslant(T,\,S)\quad :
\quad \exists e\in \En\quad \abs{\mathcal{M}_n\big(\omega(t,\,s)_{e}\big)}\ \leqslant\ t+2\,\Big\}\,.$$
The fact that~$T\leqslant n^d-2$ and~\smash{$\abs{\mathcal{M}_n\big(\omega(n^d-2,\,r)\big)}\leqslant n^d$} ensures that~$(T',\,S')$ is well-defined and that~$T'\leqslant n^d-2$. Let us show, by distinguishing several cases, that
\begin{equation}
\label{Mn_sup_biss}
V=v_0
\qquadimplique
\abs{\mathcal{M}_n\big(\omega(T',\,S')\big)}\ \geqslant\ T'+2\,.
\end{equation}

\point If~$(T',\,S')=(T,\,S)$, then the claim follows from~(\ref{Mn_mino_Mn_Cn}).

\point If~$(T',\,S')>(T,\,S)$ and~$S'=0$, then the minimality of~$(T',\,S')$ implies that
$$T'-1+2\ <\ \abs{\mathcal{M}_n\big(\omega(T'-1,\,r)\big)}\ =\ \abs{\mathcal{M}_n\big(\omega(T',\,S')\big)}\,.$$

\point Else if~$(T',\,S')>(T,\,S)$ and~$S'\neq 0$, then by minimality of~$(T',\,S')$, we know that for all~$e\in \En$, 
$$\abs{\mathcal{M}_n\big(\omega(T',\,S'-1)_{e}\big)}\ >\ T'+2\,,$$
which entails in particular that~\smash{$\abs{\mathcal{M}_n\big(\omega(T',\,S')\big)}>T'+2$}, because the configuration~$\omega(T',\,S')$ is obtained from the configuration~$\omega(T',\,S'-1)$ by closing at most one edge.

\noindent We conclude that~(\ref{Mn_sup_biss}) holds in all cases.
\\

\souspreuve{Construction of the happy event:}
We now wish to define a set of edges~$H$ that we want to be closed between the configuration~$\omega(T',\,S')$ and the configuration~$\omega(T'+2,\,0)$ in order to have
$$\abs{\mathcal{M}_n\big(\omega(T'+2)\big)}\ =\ T'+2\,.$$
We are only interested in situations where~$V=v_0$, thus we set arbitrarily~$H=\varnothing$ if the event~$\acc{V=v_0}$ does not occur.
We now assume that~$V=v_0$.
By definition of~$(T',\,S')$, there exists an edge~$e\in \En$ such that
\begin{equation}
\label{Mn_fermeture_e}
\abs{\mathcal{M}_n\big(\omega(T',\,S')_{e}\big)}\ \leqslant\ T'+2\,.
\end{equation}
We choose this edge~$e$ minimal (for the order~$e_1,\,\ldots,\,e_r$ we have considered on~$\En^T$) among the edges satisfying~(\ref{Mn_fermeture_e}), which ensures that~$e$ only depends on~$T'$,~$\omega(T',\,S')$ and~$V$.
We then construct the set~$H$ by distinguishing two cases depending on whether the inequality~(\ref{Mn_fermeture_e}) is strict or not.

\point In case there is equality in~(\ref{Mn_fermeture_e}), we take~$H=\acc{e}$.

\point Assume that~(\ref{Mn_fermeture_e}) is a strict inequality. It follows from~(\ref{Mn_sup_biss}) that
$$\abs{\mathcal{M}_n\big(\omega(T',\,S')_{e}\big)}\ <\ \abs{\mathcal{M}_n\big(\omega(T',\,S')\big)}\,,$$
which means that closing the edge~$e$ changes the number of vertices connected to the boundary of the box. Consequently, one end of the edge~$e$, say~$v$, must be disconnected from~$\partial\Lambda(n)$ when closing~$e$ in the configuration~$\omega(T',\,S')$. Write~$(C_v,\,E_v)$ for the graph of the open cluster of~$v$ in the configuration~$\omega(T',\,S')_{e}$.
We have,~using~(\ref{Mn_sup_biss}),
$$\abs{\mathcal{M}_n\big(\omega(T',\,S')_{e}\big)}
\ =\ \abs{\mathcal{M}_n\big(\omega(T',\,S')\big)}-\abs{C_v}
\ \geqslant\ T'+2-\abs{C_v}\,.$$
Combining this with the (strict) inequality~(\ref{Mn_fermeture_e}) yields
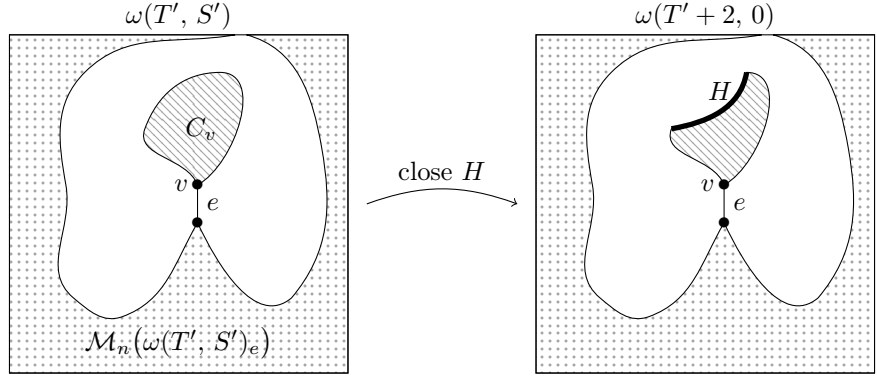
\begin{figure}[h!]
\begin{center}
\begin{tikzpicture}[scale=0.5]
\filldraw[pattern = dots, pattern color=gray!70] (6.3,9) to (9,9) to (9,0) to (0,0) to (0,9) to (6,9) to[out=-170, in=45] (2,8) to[out=-135, in=100] (1.5,5) to[out=-80, in=140] (1.7,2) to[out=-40, in=200] (3,1.5) to[out=20, in=-120] (5,4) to[out=-65, in=-140] (7.5,2) to[out=50, in=-80] (8.3,6) to[out=100, in=-15] (6.3,9);
\draw (0,0) --++ (0,9) --++ (9,0) --++ (0,-9) -- cycle;
\draw (4.5,0.8) node{$\mathcal{M}_n\big(\omega(T',\,S')_e\big)$};
\filldraw[pattern = north west lines, pattern color=gray!70] (5,5) to[out=30, in=0] (5.6,8) to[out=180, in=70] (3.6,6.5) to[out=-110, in=110] (5,5);
\draw (5,4) node{$\bullet$} to node[midway, right]{$e$} (5,5) node{$\bullet$} node[left]{$v$};
\draw (5.1,6.5) node{$C_v$};
\draw[->] (9.5,4.5) to[bend left=20] node[midway, above]{close~$H$} (13.5,4.5);
\filldraw[pattern = dots, pattern color=gray!70] (20.3,9) to (23,9) to (23,0) to (14,0) to (14,9) to (20,9) to[out=-170, in=45] (16,8) to[out=-135, in=100] (15.5,5) to[out=-80, in=140] (15.7,2) to[out=-40, in=200] (17,1.5) to[out=20, in=-120] (19,4) to[out=-65, in=-140] (21.5,2) to[out=50, in=-80] (22.3,6) to[out=100, in=-15] (20.3,9);
\filldraw[pattern = north west lines, pattern color=gray!70] (19,5) to[out=30, in=0] (19.6,8) to[out=260, in=10] (17.6,6.5) to[out=-110, in=110] (19,5);
\draw[line width=2pt] (19.6,8) to[out=260, in=10] node[midway, above]{$H$} (17.6,6.5);
\draw (19,4) node{$\bullet$} to node[midway, right]{$e$} (19,5) node{$\bullet$} node[left]{$v$};
\draw (14,0) --++ (0,9) --++ (9,0) --++ (0,-9) -- cycle;
\draw (4.5,9.5) node{$\omega(T',\,S')$};
\draw (18.5,9.5) node{$\omega(T'+2,\,0)$};
\end{tikzpicture}
\end{center}
\caption{If~(\ref{Mn_fermeture_e}) is a strict inequality, then closing the edge~$e$ in the configuration~$\omega(T',\,S')$ changes the number~$\abs{\mathcal{M}_n}$ of vertices connected to the boundary of the box. This means that one end of the edge~$e$, say~$v$, happens to be disconnected from the boundary when~$e$ is closed. We then choose a subset~$H$ of the edges of the cluster~$C_v$ which is disconnected by the closure of~$e$, such that closing all the edges of~$H$ and no other edges of~$\Edge{\mathcal{M}_n}$ between~$(T',\,S')$ and~$(T'+2,\,0)$ implies~\smash{$\abs{\mathcal{M}_n\big(\omega(T'+2)\big)}=T'+2$}.}
\end{figure}
$$1\ \leqslant\ T'+2-\abs{\mathcal{M}_n\big(\omega(T',\,S')_{e}\big)}
\ \leqslant\ \abs{C_v}\,.$$
Applying lemma~\ref{lemme_boucher} to the graph~$\left(C_v,\,E_v\right)$ and the vertex~$v$, we can choose a set~$H\subset E_v$ satisfying
\begin{equation}
\label{Mn_cardH}
\abs{H}\ \leqslant\ K\abs{C_v}^{\frac{d-1}{d}}
\end{equation}
and such that the cluster of~$v$ in the graph~$\left(C_v,\,E_v\backslash H\right)$ contains exactly~\smash{$T'+2-\abs{\mathcal{M}_n\big(\omega(T',\,S')_{e}\big)}$} vertices. We then have
$$\abs{\mathcal{M}_n\big(\omega(T',\,S')_{H}\big)}\ =\ T'+2\,.$$
The edge~$e$ (and thus the vertex~$v$) depends only on~$T'$,~$\omega(T',\,S')$ and~$V$, thus we can choose such a set~$H$ which also depends only on~$T'$,~$\omega(T',\,S')$ and~$V$. Besides, we have the following control over~$\abs{C_v}$:
\begin{equation}
\label{Mn_CV1}
\abs{C_v}
\ =\ \abs{C_{\Lambda(n)}\big(v,\,\omega(T',\,S')_e\big)}\\
\ \leqslant\ \abs{C_{\Lambda(n)}\big(v,\,\omega(T',\,S')\big)}\\
\ \leqslant\ \abs{C_{max}^T\big(\omega(T',\,S')\big)}\,.
\end{equation}
Note now that~$\abs{C_{max}^T(\omega)}$ is a decreasing function of~$\omega$ and that, by definition,~$(T',\,S')\geqslant(T,\,S)$, whence
\begin{equation}
\label{Mn_CV2}
\abs{C_{max}^T\big(\omega(T',\,S')\big)}\ \leqslant\ \abs{C_{max}^T\big(\omega(T,\,S)\big)}\ \leqslant\ 2T+3\,.
\end{equation}
Combining~(\ref{Mn_CV1}) and~(\ref{Mn_CV2}), we get~$\abs{C_v}\leqslant 2T+3$, and therefore the upper bound~(\ref{Mn_cardH}) becomes
\begin{equation}
\label{Mn_cardH2}
\abs{H}\ \leqslant\ K\big(2T+3\big)^{\frac{d-1}{d}}\,.
\end{equation}

\noindent To sum up these two cases, we have defined a (random) set of edges~$H\subset\En$ whose size is controlled by~(\ref{Mn_cardH2}) and which satisfies
$$\abs{\mathcal{M}_n\big(\omega(T',\,S')_{H}\big)}\ =\ T'+2\,.$$
Therefore, conditionally on~$\acc{V=v_0}$, if the edges belonging to~$H$ and no other edges of~\smash{$\Edge{\mathcal{M}_n\big(\omega(T',\,S')\big)}$} are closed between the configurations~$\omega(T',\,S')$ and~$\omega(T'+2,\,0)$, then we have
\begin{equation}
\label{Mn_egalite}
\abs{\mathcal{M}_n\big(\omega(T'+2)\big)}\ =\ T'+2\,.
\end{equation}
This leads us to consider the event
$$\mathcal{E}\ =\ \acc{\,
\begin{aligned}
&\forall s>S'\qquad e_s\in \Edge{\mathcal{M}_n\big(\omega(T',\,S')\big)}\Rightarrow X_{T',\,e_s}=1\\
&\forall e\in H\qquad X_{T'+1,\,e}=0\\
&\forall e\in \Edge{\mathcal{M}_n\big(\omega(T',\,S')\big)}\backslash H\qquad X_{T'+1,\,e}=1
\end{aligned}\,
}$$
which, if it occurs and if~$V=v_0$, implies~(\ref{Mn_egalite}). Also, our expression~(\ref{Cmax_expression_Z_n}) becomes
\begin{equation}
\label{lienEnZn}
Z_n\ =\ \Proba\Big(\,\exists t\in\acc{0,\,\ldots,\,n^d}\quad \abs{\mathcal{M}_n\big(\omega(t)\big)}=t\,\Big)
\ \geqslant\ \Proba\Big(\,\mathcal{E}\cap\acc{V=v_0}\,\Big)\,.
\end{equation}
\\

\souspreuve{Conditional probability of the happy event:}
As in the proof of lemma~\ref{Cmax_lemme_minoration_Z_n}, we consider~$(t_0,\,t_0',\,s_0')$ and a configuration~$\omega_0:\En^T\rightarrow\acc{0,1}$ such that
$$\Proba\big(\mathcal{C}_{t_0,\,t_0',\,s_0',\,\omega_0}\big)\ >\ 0
\quadou
\mathcal{C}_{t_0,\,t_0',\,s_0',\,\omega_0}
\ =\ \Big\{\,(T,\,T',\,S')=(t_0,\,t_0',\,s_0')\,\Big\}
\cap\Big\{\,\omega(T',\,S')=\omega_0\,\Big\}
\cap\Big\{\,V=v_0\,\Big\}\,.$$
By definition of~$T$ and~$T'$, we have
\begin{equation}
\label{n6}
\abs{\mathcal{M}_n\left(\omega_0\right)}
\ \leqslant\ (T'+2)+(2T+3)
\ \leqslant\ 3T'+5\,.
\end{equation}
The event~$\mathcal{C}_{t_0,\,t_0',\,s_0',\,\omega_0}$ depends only on the variables~$X_{t,\,e_s}$ with~$(t,\,s)\leqslant(t_0',\,s_0')$ and, conditionally on this event, the event~$\mathcal{E}$ only depends on the variables~$X_{t,\,e_s}$ for~$(t_0',\,s_0')<(t,\,s)<(t_0'+2,\,0)$. What's more, the set~$H$ only depends on~$T'$,~$\omega(T',\,S')$ and~$V$, which allows us to write~$H=H\left(T',\,\omega(T',\,S'),\,V\right)$. Therefore, we have
\begin{align*}
\Proba\big(\,\mathcal{E}\ \big|\ \mathcal{C}_{t_0,\,t_0',\,s_0',\,\omega_0}\,\big)
\ &=\ \prod_{\substack{s>s_0'\\ e_s\in\Edge{\mathcal{M}_n(\omega_0)}}}{\Proba\big(\,X_{t_0',\,e_s}=1\,\big)}\times\prod_{e\in H(t_0',\,\omega_0,\,v_0)}{\Proba\big(\,X_{t_0'+1,\,e}=0\,\big)}\\
&\qquad\qquad\qquad\qquad\times\prod_{e\in\Edge{\mathcal{M}_n(\omega_0)}
\backslash H(t_0',\,\omega_0,\,v_0)}{\Proba\big(\,X_{t_0'+1,\,e}=1\,\big)}\\
\ &\geqslant\ \left(e^{-1/n^a}\right)^{2\left|\Edge{\mathcal{M}_n\left(\omega_0\right)}\right|}\left(1-e^{-1/n^a}\right)^{\abs{H(t_0',\,\omega_0,\,v_0)}}\,.
\end{align*}
Using the upper bound~(\ref{n6}) on~$\abs{\Mn(\omega_0)}$ and the upper bound~(\ref{Mn_cardH2}) on~$\abs{H}$ leads to
$$\Proba\Big(\,\mathcal{E}\ \Big|\ T,\,T',\,S',\,\omega(T',\,S'),\,V\,\Big)
\ \geqslant\ \mathbb{1}_{V=v_0}\exp\left(-\frac{2d(3T'+5)}{n^a}-2K a(\ln n)\left(2T+3\right)^{\frac{d-1}{d}}\right)\,.$$
Taking the conditional expectation with respect to~$(T',\,V)$ and using the fact that~$T\leqslant T'$, we obtain
\begin{equation}
\label{Mn_proba_cond}
\Proba\big(\,\mathcal{E}\,\big|\,T',\,V\,\big)
\ \geqslant\ \mathbb{1}_{V=v_0}\exp\left(-\frac{2d(3T'+5)}{n^a}-2K a(\ln n)\left(2T'+3\right)^{\frac{d-1}{d}}\right)\,.
\end{equation}
\\

\souspreuve{Upper bound on~$T'$:}
It follows from lemma~\ref{Mn_majo_souscritique} that
$$\limsupn\,\frac{1}{n^a}\,\ln \Proba_{p_c/2}\Bigg[\,\abs{\mathcal{M}_n}\,\geqslant\,n^a\bigg(-\ln\left(\frac{p_c}{2}\right)\bigg)\,\Bigg]\ <\ 0\,.$$
Therefore, we have
$$\Proba_{p_c/2}\Bigg[\,\abs{\mathcal{M}_n}\,\geqslant\, n^a\bigg(-\ln\left(\frac{p_c}{2}\right)\bigg)\,\Bigg]
\ =\ o\left(\frac{1}{n^d}\right)\,,$$
and thus, if we take~$\tau_n^+$ defined as in~(\ref{defTau}) then, for~$n$ large enough,
$$\Proba\Big(\,\abs{\Mn\left(\omega(\tau_n^+)\right)}\,\geqslant\,\tau_n^+\Big)
\ \leqslant\ 
\Proba_{p_c/2}\Bigg[\,\abs{\mathcal{M}_n}\,\geqslant\,n^a\bigg(-\ln\left(\frac{p_c}{2}\right)\bigg)\,\Bigg]
\ \leqslant\ \frac{1}{2n^d}\,.$$
We then have, using the fact that~$\Proba\big(V=v_0\big)\geqslant 1/n^d$,
$$\Proba\Big(\,V=v_0\ \text{and}\ \abs{\Mn\left(\omega(\tau_n^+)\right)}\,<\,\tau_n^+\,\Big)
\ \geqslant\ \frac{1}{n^d}-\frac{1}{2n^d}
\ =\ \frac{1}{2n^d}\,.$$
Yet, if~$V=v_0$ and~$\abs{\Mn\left(\omega(\tau_n^+)\right)}< \tau_n^+$, then inequality~(\ref{Mn_sup_biss}) entails that~$T'< \tau_n^+$. 
From this we can deduce that, for~$n$ large enough,
$$\Proba\Big(\,V=v_0\ \text{and}\ T'\,<\,\tau_n^+\,\Big)
\ \geqslant\ \frac{1}{2n^d}\,.$$
Therefore, we can find~$\kappa\geqslant 2$ such that, for~$n$ large enough,
\begin{equation}
\label{Mn_kappa}
\Proba\Big(\,V=v_0\ \text{and}\ T'\,\leqslant\,\kappa n^a\,\Big)\ \geqslant\ \frac{1}{2n^d}\,.
\end{equation}
\\

\souspreuve{Conclusion:}
Combining~(\ref{Mn_proba_cond}) and~(\ref{Mn_kappa}) yields
\begin{align*}
\Proba\Big(\,\mathcal{E}\cap\acc{V=v_0}\,\Big)
\ &\geqslant\ \Proba\Big(\,V=v_0\ \text{and}\ T'\,\leqslant\,\kappa n^a\,\Big)\,\Proba\Big(\,\mathcal{E}\ \Big|\ V=v_0\ \text{and}\ T'\,\leqslant\,\kappa n^a\,\Big)\\
\ &\geqslant\ \frac{1}{2n^d}\exp\left(-\frac{2d(3\kappa n^a+5)}{n^a}-2K a(\ln n)\left(2\kappa n^a+3\right)^{\frac{d-1}{d}}\right)\\
\ &\geqslant\ \frac{1}{2n^d}\exp\left(-6d\kappa-\frac{10d}{n^a}-8K \kappa a(\ln n) n^{a(d-1)/d}\right)\,.
\end{align*}
Given~(\ref{lienEnZn}), we obtain
$$\liminf_{n\rightarrow\infty}\,\frac{\ln Z_n}{(\ln n)n^{a(d-1)/d}}
\ \geqslant\ \liminf_{n\rightarrow\infty}\,\frac{\ln \Proba\big(\acc{V=v_0}\cap\mathcal{E}\big)}{(\ln n)n^{a(d-1)/d}}
\ \geqslant\ -8K\kappa a
\ >\ -\infty\,,$$
which is the required lower bound.
\end{proof}

\subsection{Proof of the convergence result}

We now explain how the case~($ii$) of theorem~\ref{thm_clusters_convergence} follows from the above lemmas.

\begin{proof}[Proof of theorem~\ref{thm_clusters_convergence}, case~($ii$)]
Let~$\varepsilon>0$ and~$a\in(d-1,\,d)$.
Exactly as in section~\ref{section_Cmax_preuve_resultat}, the upper bounds
$$\limsupn\,\frac{1}{n^a}\,\ln\mu_n\big(p_n<p_c-\varepsilon\big)
\ <\ 0\,,
\qquad
\limsupn\,\frac{1}{n^{d-1}}\,\ln\mu_n\big(p_n>p_c+\varepsilon\big)
\ <\ 0\,,$$
follow from our lower bound on~$Z_n$ (lemma~\ref{Mn_lemme_minoration_Z_n}) and from the results of exponential decay in the subcritical (lemma~\ref{Mn_majo_souscritique}) and supercritical phases (lemma~\ref{Mn_majo_surcritique}).
To obtain the lower bound, we go back to our computation~(\ref{munsurexact}) to write
\begin{equation}
\label{eq491}
\mu_n\big(p_n>p_c+\varepsilon\big)
\ =\ \frac{1}{Z_n}\sum_{t=0}^{t_n^+-1}\Proba_{\varphi_n(t)}\Big(\abs{\Mn}=t \Big)
\ \geqslant\ \Proba_{\varphi_n(t_n^+-1)}\Big(\abs{\Mn}=t_n^+-1 \Big)\,,
\end{equation}
where~$t_n^+$ is still given by~(\ref{defTn}).
We implement now a simplified surgery procedure to force~$\abs{\Mn}=t_n^+-1$ starting from a configuration~$\omega$ such that~$\abs{\Mn(\omega)}>t_n^+-1$.
According to lemma~\ref{Mn_majo_surcritique}, we have
\begin{equation}
\label{probapriori}
\Proba_{\varphi_n(t_n^+-1)}\Big(\abs{\Mn}>t_n^+-1 \Big)
\ \geqslant\ \Proba_{p_c+\varepsilon}\Big(\abs{\Mn}>t_n^+-1 \Big)
\ \cvninfty\ 1\,.
\end{equation}
Let~\smash{$\omega\in\acc{0,1}^{\En}$} be a configuration such that~$\abs{\Mn(\omega)}>t_n^+-1$.
Consider the set~$E$ of the edges of~$\En$ which have exactly one endpoint in~$\partial\Lambda(n)$, which is such that
\begin{equation}
\label{numbord}
\abs{\Mn(\omega_E)}
\ =\ \abs{\partial\Lambda(n)}
\ \leqslant\ 2dn^{d-1}
\ <\ t_n^+-1
\end{equation}
for~$n$ large enough, because~$a>d-1$.
We write~$E=\acc{e_1,\,\ldots,\,e_{\abs{E}}}$ with~$\abs{E}\leqslant 2dn^{d-1}$, and we let
$$B\ =\ \max\Big\{\,b\in\acc{1,\,\ldots,\,\abs{E}}\ :\ \abs{\Mn\big(\omega_{\acc{e_1,\,\ldots,\,e_b}}\big)}\,\geqslant\,t_n^+-1\,\Big\}\,.$$
It follows from~(\ref{numbord}) that~$B<\abs{E}$, whence by maximality of~$B$,
$$\abs{\Mn\big(\omega_{\acc{e_1,\,\ldots,\,e_{B+1}}}\big)}
\ <\ t_n^+-1\,.$$
Therefore, if we write~$e_{B+1}=\acc{x,y}$ with~$x\in\partial\Lambda(n)$ and~$y\notin\partial\Lambda(n)$, and if we consider the cluster which is disconnected from the boundary when closing this edge~$e_{B+1}$, namely :
$$C_y\ =\ C_{\Lambda(n)}\big(y,\,\omega_{\acc{e_1,\,\ldots,\,e_{B+1}}}\big)\,,$$
we have
$$\abs{\Mn\big(\omega_{\acc{e_1,\,\ldots,\,e_{B+1}}}\big)}
\ =\ \abs{\Mn\big(\omega_{\acc{e_1,\,\ldots,\,e_B}}\big)}-\abs{C_y}\,,$$
so that~$m=(t_n^+-1)-\abs{\Mn\big(\omega_{\acc{e_1,\,\ldots,\,e_{B+1}}}\big)}$ satisfies~$1\leqslant m\leqslant \abs{C_y}$.
Hence, lemma~\ref{lemme_boucher} provides us with~$H_1\subset\En$, of size~\smash{$\abs{H_1}\leqslant \abs{C_y}^{(d-1)/d}\leqslant n^{d-1}$}, such that
$$\abs{C_{\Lambda(n)}\big(y,\,\omega_{\acc{e_1,\,\ldots,\,e_{B+1}}\cup H_1}\big)}\ =\ m\,.$$
Writing~$H=H(\omega)=\acc{e_1,\,\ldots,\,e_{B}}\cup H_1$, we then have~$\abs{H}\leqslant (2d+1)n^{d-1}$ and
$$\abs{\Mn\big(\omega_H\big)}\ =\ \abs{\Mn\big(\omega_{\acc{e_1,\,\ldots,\,e_{B+1}}}\big)}+m
\ =\ t_n^+-1\,.$$
Using lemma~6.3 in~\cite{CerfPisztora2000}, we can deduce that
$$\Proba_{\varphi_n(t_n^+-1)}\Big(\abs{\Mn}=t_n^+-1 \Big)
\ \geqslant\ \left(\frac{1}{C_n}\right)^{(2d+1)n^{d-1}}\times\Proba_{\varphi_n(t_n^+-1)}\Big(\abs{\Mn}>t_n^+-1 \Big)\,,$$
where
$$C_n\ =\ \left(1\vee\frac{\varphi_n(t_n^+-1)}{1-\varphi_n(t_n^+-1)}\right)\abs{E_n}
\ =\ O\big(n^d\big)\,.$$
Plugging this into~(\ref{eq491}) and using~(\ref{probapriori}) then yields
$$\liminfn\,\frac{1}{(\ln n)n^{d-1}}\,\ln\,\mu_n\big(\abs{p_n-p_c}>\varepsilon\big)
\ \geqslant\ \liminfn\,\frac{1}{(\ln n)n^{d-1}}\,\ln\,\mu_n\big(p_n>p_c+\varepsilon\big)
\ \geqslant\ -(2d+1)d
\ >\ -\infty\,,$$
as announced in theorem~\ref{thm_clusters_convergence}.
\end{proof}

\section{Proof of case (\textit{iii}) of theorem~\ref{thm_clusters_convergence}}
\label{sectionBn}

The goal of this section is to prove the remaining part of theorem~\ref{thm_clusters_convergence}, namely the case~($iii$), where the function~$p_n$ is defined by
$$p_n(\omega)\ =\ \exp\left(-\frac{\abs{B_n^b(\omega)}}{n^a}\right)\,,$$
where~$a$ and~$b$ are two fixed parameters such that~$0<b<a<d$, and
$$B_n^b(\omega)\,=\,\Big\{\,x\in\Lambda(n)\ :\  \abs{C(x,\,\omega)}\geqslant n^b\,\Big\}\,.$$
In subsection~\ref{prthmvitesse}, we will also obtain the estimate on the convergence speed announced in theorem~\ref{thm_clusters_vitesse}, which depends on the existence of the critical exponents~$\beta$ and~$\gamma$.
The rest of the section is organized as the two previous sections, with first the large deviation estimates far from~$p_c$ and then the lower bound on~$Z_n$ (with, this time, two different lower bounds).

\subsection{Exponential decay in the subcritical phase}
\label{subBnsous}

We now prove the following exponential decay in the subcritical regime :

\begin{lemma}
\label{lemme_majo_souscritique}
For every~$p<p_c$ and any~$A>0$, we have
$$-\infty\ <\ 
\liminf\limits_{n\rightarrow\infty}\,\frac{1}{n^{a}}\ln\Proba_p\Big(\abs{B_n^b}>An^a\Big)
\ \leqslant\ 
\limsup\limits_{n\rightarrow\infty}\,\frac{1}{n^{a}}\ln\Proba_p\Big(\abs{B_n^b}>An^a\Big)\ <\ 0\,.$$
\end{lemma}

\begin{proof}
Let~$p<p_c$ and~$A>0$.
Writing~$N_n=1+\Ent{An^{a-b}}$ and using the BK inequality as in the proof of lemma~\ref{Mn_majo_souscritique}, we get
\begin{align*}
\Proba_p\Big(\abs{B_n^b}>An^a\Big)
\,&\leqslant\,
\sum_{k=1}^{N_n}\,
\sum_{x_1,\,\ldots,\,x_k\in\Lambda(n)}
\sum_{\substack{n^b\leqslant n_1,\,\ldots,\,n_k\leqslant n^d\\n_1+\cdots+n_k>An^a}}
\Proba_p\Big(\acc{\abs{C_{\Lambda(n)}(x_1)}\geqslant n_1}\circ\cdots\circ\acc{\abs{C_{\Lambda(n)}(x_k)}\geqslant n_k}\Big)\\
\ &\leqslant\ 
\sum_{k=1}^{N_n}\,
\sum_{x_1,\,\ldots,\,x_k\in\Lambda(n)}
\,\sum_{\substack{n^b\leqslant n_1,\,\ldots,\,n_k\leqslant n^d\\n_1+\cdots+n_k>An^a}}
\prod_{i=1}^ke^{-\lambda(p)n_i}\\
\ &\leqslant\ N_n\big(n^d\big)^{2N_n} e^{-\lambda(p)An^a}\,.
\end{align*}
Therefore, we obtain
$$\frac{1}{n^{a}}\ln\Proba_p\Big(\abs{B_n^b}>An^a\Big)
\ \leqslant\ \frac{\ln N_n}{n^a}+\frac{2N_nd\ln n}{n^a}-\lambda(p)A
\ \cvninfty\ -\lambda(p)A\ <\ 0\,,$$
which proves the upper bound. The lower bound follows from the lower bound given by lemma~\ref{Cmax_majo_souscritique}, since~$An^a\geqslant n^b$ for~$n$ large enough.
\end{proof}

\subsection{Exponential decay in the supercritical phase}
\label{subBnsur}

We now deal with the deviations in the regime~$p>p_c$.
We wish to thank an anonymous referee for having improved our proof, leading to a better (and in fact optimal) exponent.
\begin{lemma}
\label{lemme_majo_surcritique}
We have the upper bound
$$\forall p>p_c\quad
\forall A>0\qquad
\limsup\limits_{n\rightarrow\infty}\,\frac{1}{n^{d-b/d}}\ln\Proba_p\Big(\abs{B_n^b}< An^a\Big)\ <\ 0\,.$$
\end{lemma}

\begin{proof}
Let~$p>p_c$ and~$A>0$.
We shall partition the box~$\Lambda(n)$ into hypercubic boxes of side
$$N_n\ =\ \Ceil{\left(\frac{8n^b}{\theta(p)}\right)^{1/d}}\,.$$
We let
$$M_n\ =\ \min\Bigg\{\,m\in\N\ :\ \Lambda(n)\subset\bigcup_{j\in\Lambda(m)}\big(N_nj+\Lambda(N_n)\big)\,\Bigg\}\,,$$
so that we have a partition
$$\Lambda(n)\ =\ \bigsqcup_{j\in\Lambda(M_n)}\Big[\big(N_nj+\Lambda(N_n)\big)\cap\Lambda(n)\Big]\,.$$
By definition of~$B_n^b$, we have
\begin{align*}
\abs{B_n^b}
\ &=\ \abs{\Big\{x\in \Lambda(n)\ :\ \abs{C_{\Lambda(n)}(x)}\geqslant n^b\Big\}}\\
\ &\geqslant\ n^b\abs{\bigg\{\,j\in\Lambda(M_n)\ :\ \abs{C_{max}\Big[\big(N_nj+\Lambda(N_n)\big)\cap\Lambda(n)\Big]}\,\geqslant\, n^b\,\bigg\}}\,.
\end{align*}
Now note that
$$\frac{n^b\abs{\Lambda(M_n)}}{2}
\ =\ \frac{n^bM_n^d}{2}
\ \eqninfty\ \frac{n^b}{2}\left(\frac{n}{N_n}\right)^d
\ \eqninfty\ \frac{n^{b+d}\theta(p)}{16n^b}
\ =\ \frac{\theta(p)n^d}{16}\,.$$
Given that~$An^a=o(n^d)$, this implies that, for~$n$ large enough,
$$An^a\ <\ \frac{n^b\abs{\Lambda(M_n)}}{2}\,.$$
Therefore, we have the following implication :
\begin{align*}
\abs{B_n^b}\,<\,An^a
&\quadimplique \abs{\bigg\{\,j\in\Lambda(M_n)\ :\ \abs{C_{max}\Big[\big(N_nj+\Lambda(N_n)\big)\cap\Lambda(n)\Big]}\,\geqslant\, n^b\,\bigg\}}\,<\,\frac{\abs{\Lambda(M_n)}}{2}\\
&\quadimplique \abs{\bigg\{\,j\in\Lambda(M_n)\ :\ \abs{C_{max}\Big[\big(N_nj+\Lambda(N_n)\big)\cap\Lambda(n)\Big]}\,<\, n^b\,\bigg\}}\,\geqslant\,\frac{\abs{\Lambda(M_n)}}{2}\,.
\end{align*}
The problem now is that the boxes on the boundaries might be truncated. However, the inside boxes are full, that is to say
$$\forall j\in\Lambda(M_n-2)\qquad
N_nj+\Lambda(N_n)\ \subset\ \Lambda(n)\,.$$
Yet the number of boxes on the boundaries is
$$\big|\Lambda(M_n)\backslash\Lambda(M_n-2)\big|\ =\ o\big(\abs{\Lambda(M_n)}\big)\,,$$
so that we have, for~$n$ large enough,
$$\abs{B_n^b}\,<\,An^a
\quadimplique \abs{\Big\{\,j\in\Lambda(M_n-2)\ :\ \abs{C_{max}\big(N_nj+\Lambda(N_n)\big)}\,<\, n^b\,\Big\}}\,\geqslant\,\frac{\abs{\Lambda(M_n)}}{4}\,.$$
Using the independence of the sizes of the largest cluster inside disjoint boxes and the fact that the number of choices of at least~$\abs{\Lambda(M_n)}/4$ boxes is at most~\smash{$2^{\abs{\Lambda(M_n)}}$}, we get
$$\Proba_p\Big(\abs{B_n^b}< An^a\Big)
\ \leqslant\ 2^{M_n^d}\Proba_p\Big(\abs{C_{max}\big(\Lambda(N_n)\big)}\,<\,n^b\,\Big)^{M_n^d/4}\,,$$
which implies that
$$\limsupn\,\frac{1}{n^{d-b/d}}\ln\Proba_p\Big(\abs{B_n^b}< An^a\Big)
\ \leqslant\ \limsupn\,\left[\,\frac{M_n^d\ln 2}{n^{d-b/d}}+\frac{M_n^d}{4n^{d-b/d}}\,\ln\Proba_p\Big(\abs{C_{max}\big(\Lambda(N_n)\big)}\,<\,n^b\,\Big)\,\right]\,.$$
Now note that
$$\frac{M_n^d}{n^{d-b/d}}\ \eqninfty\ \frac{(n/N_n)^d}{n^{d-b/d}}
\ \eqninfty\ \frac{n^{b/d}}{N_n}\frac{1}{N_n^{d-1}}
\ \eqninfty\ \left(\frac{\theta(p)}{8}\right)^{1/d}\frac{1}{N_n^{d-1}}\,.$$
Therefore, we obtain
$$\limsupn\,\frac{1}{n^{d-b/d}}\ln\Proba_p\Big(\abs{B_n^b}< An^a\Big)
\ \leqslant\ \frac{1}{4}\left(\frac{\theta(p)}{8}\right)^{1/d}\limsupn\,\frac{1}{N_n^{d-1}}\,\ln\Proba_p\Big(\abs{C_{max}\big(\Lambda(N_n)\big)}\,<\,n^b\,\Big)\,.$$
The result then follows from lemma~\ref{CmaxMajoStronger}, noting that~$n^b\leqslant\theta(p)N_n^d/8$.
\end{proof}

\subsection{Two lower bounds on the partition function}
\label{sectionBnminoZn}

It remains to prove a lower bound on the normalization constant~$Z_n$.
Adapting the technique of lemmas~\ref{Cmax_lemme_minoration_Z_n} and~\ref{Mn_lemme_minoration_Z_n}, we can easily obtain a bound with an exponent~$b$.
This is done in lemma~\ref{lemme_minoration_Z_n}, and the proof is much simpler than in the previous sections because, instead of performing a surgery step, we only \guillemets{freeze} the edges of~$B_n^b$ during a certain number of steps.
But this bound may not be sufficient to outweigh the bound in the supercritical phase, since it may be the case that~$b>d-b/d$.
To solve this problem, we show an other lower bound in lemma~\ref{lemme_minoration_Z_n_bis} with a different exponent, using a more geometrical technique.

\begin{lemma}
\label{lemme_minoration_Z_n}
For every~$a$ and~$b$ such that~$0<b<a<d$, we have
$$\liminf\limits_{n\rightarrow\infty}\,\,\frac{\ln Z_n}{n^{b}}\ >\ -\infty\,.$$
\end{lemma}

\begin{proof}
We use the same monotone coupling~$(\omega(t,s))_{0\leqslant t\leqslant n^d,\,0\leqslant s\leqslant r}$ as in the proof of lemma~\ref{Cmax_lemme_minoration_Z_n}.
Following a strategy similar to that of lemma~\ref{Cmax_lemme_minoration_Z_n}, we define
\begin{equation}
\label{defTS}
(T,\,S)\ =\ \min\Big\{\,
(t,\,s)\in\acc{0,\,\ldots,\,n^{d}-2}\times\acc{0,\,\ldots,\,r}\quad :\quad
\abs{B_n^b\big(\omega(t,\,s)\big)}\leqslant t+1+2n^b\,\Big\}\,.
\end{equation}
When closing one single edge,~$\abs{B_n^b}$ cannot decrease by more than~$2n^b$ (in the worst case, the edge cuts a cluster of~\smash{$2\Ceil{n^b}-2$} vertices in two equal parts). Therefore, we always have~$\abs{B_n^b\big(\omega(T,\,S)\big)}\geqslant T+1$.
Thus, if we consider the instant~$T'=\abs{B_n^b\big(\omega(T,\,S)\big)}$, we have~$T+1\leqslant T'\leqslant T+1+2n^b$.
In view of this, our strategy is to force all the edges of~$B_n^b\big(\omega(T,\,S)\big)$ to remain open until the configuration~$\omega(T',\,0)$.
This idea is much simpler than the strategy of the previous sections, because we do not perform any surgery step.
Considering the event
$$\mathcal{E}\ =\ \acc{
\begin{aligned}
&\forall s>S\qquad e_s\in \Edge{B_n^b\big(\omega(T,\,S)\big)}\ \Rightarrow\ X_{T,\,e_s}=1\\
&\forall t\in\acc{T+1,\,\ldots,\,T'-1}\quad \forall e\in \Edge{B_n^b\big(\omega(T,\,S)\big)}\qquad  X_{t,\,e}=1
\end{aligned}
}\,,$$
equation~(\ref{Cmax_expression_Z_n}) becomes
\begin{equation}
\label{Zn874}
Z_n
\ \geqslant\ \Proba\Big(\,\abs{B_n^b\big(\omega(T',\,0)\big)}\ =\ \abs{B_n^b\big(\omega(T,\,S)\big)}\ =\ T'\,\Big)
\ \geqslant\ \Proba\big(\mathcal{E}\big)\,.
\end{equation}
A lower bound on the probability of~$\mathcal{E}$ is easily obtained by writing
\begin{align*}
\Proba\big(\mathcal{E}\,\big|\,(T,\,S,\,\omega(T,\,S))\big)
\ &=\ \prod_{\substack{s>S\\e_s\in \Edge{B_n^b\left(\omega(T,\,S)\right)}}}\Proba\big(X_{T,\,e_s}=1\big)
\prod_{t=T+1}^{T'-1}\prod_{e\in \Edge{B_n^b\left(\omega(T,\,S)\right)}}\Proba\big(X_{t,\,e}=1\big)\\
\ &\geqslant\ \Big(e^{-1/n^a}\Big)^{(T'-T)\abs{\Edge{B_n^b\left(\omega(T,\,S)\right)}}}
\ \geqslant\ \Big(e^{-1/n^a}\Big)^{(2n^b+1)d(T+1+2n^b)}\,.\numberthis\label{eq761}
\end{align*}
We then show an upper bound on~$T$, using the same technique as in the proof of lemma~\ref{Cmax_lemme_minoration_Z_n}. With~$\tau_n^+$ defined as in~(\ref{defTau}), we can write
$$\Proba\big(T\leqslant \tau_n^+\big)
\ \geqslant\ \Proba\Big(\abs{B_n^b\big(\omega(\tau_n^+)\big)}\leqslant \tau_n^++1+2n^b\Big)
\ \geqslant\ \Proba_{p_c/2}\Big(\abs{B_n^b}\leqslant \tau_n^++1+2n^b\Big)
\ \cvninfty\ 1\,,$$
thanks to lemma~\ref{lemme_majo_souscritique}.
Combining this with~(\ref{Zn874}) and~(\ref{eq761}) then leads to
$$\frac{\ln Z_n}{n^b}
\ \geqslant\ \frac{\ln \Proba\big(T\leqslant \tau_n^+\big)}{n^b}+\frac{\ln\Proba\big(\mathcal{E}\,\big|\,T\leqslant \tau_n^+\big)}{n^b}
\ \geqslant\ o(1)-\frac{(2n^b+1)d(\tau_n^++1+2n^b)}{n^{a+b}}
\ \cvninfty\ -2d\left(-\ln\left(\frac{p_c}{2}\right)\right)\,,$$
using that~$\tau_n^+\sim(-\ln(p_c/2))n^a$.
\end{proof}

We now state the other lower bound we obtain using a more geometrical technique:

\begin{lemma}
\label{lemme_minoration_Z_n_bis}
For every~$a$ and~$b$ such that~$0<b<a<d$, we have
$$\liminf\limits_{n\rightarrow\infty}\,\,\frac{\ln Z_n}{(\ln n)n^{c}}\ >\ -\infty
\qquadou
c\ =\ \left(1-\frac{a}{d}+\frac{b}{d}\right)\vee\left(a-\frac{a}{d}\right)\,.$$
\end{lemma}

\souspreuve{Proof outline:}
We use again the same coupling~$(\omega(t,s))_{t,s}$ as in the previous proof, and the same instant~$(T,\,S)$.
We then want to close edges to reach a fixed point, but the problem is that it is not always possible to do so by only closing edges.
Imagine for example that, in the configuration~$\omega(T,\,S)$, no cluster contains more than~$\Ceil{n^b}$ vertices, meaning that~$B_n^b$ is only made of clusters containing exactly~$\Ceil{n^b}$ vertices.
Then, in this very unfavourable situation, closing edges either does not affect~$\abs{B_n^b}$ or it diminishes~$\abs{B_n^b}$ by~$\Ceil{n^b}$, thus we cannot finely tune~$B_n^b$ only by closing edges.
To circumvent this problem, we will change what happened before the instant~$(T,\,S)$ so as to ensure that, at this instant, we have at our disposal a cluster containing at least~$2n^b$ vertices.
This will enable us to use our surgery procedure on this cluster, in order to reach the exact desired value for~$\abs{B_n^b}$.
However, to do so, we need to intervene on the past of the instant~$(T,\,S)$, which will make notations more complicated.
Namely, we will define a second coupling of configurations~$(\omega'(t,s))_{t,s}$ which is a copy of the first coupling, except that the closure times of a certain number of edges are drawn again, allowing us to close or to open these edges at different times.
Before diving into the proof, we precise our surgery procedure in the following two lemmas.
The first one is a lower bound on the number of edges we need to reopen to create a cluster of size at least~$2n^b$.

\begin{lemma}
\label{lemmeCmaxdouble}
Let~$b\in(0,d)$.
For every~$n\geqslant 1$ and for any configuration~\smash{$\omega\in\acc{0,1}^{\En}$} such that~\smash{$\abs{B_n^b(\omega)}\geqslant 2^{d+1}n^b$}, there exists a set of edges~$H\subset\En$, such that
$$\abs{C_{max}\big(\omega^H\big)}\ \geqslant\ 2n^b
\qquadet
\abs{H}\ \leqslant\ 4d\frac{n^{1+b/d}}{\abs{B_n^b(\omega)}^{1/d}}\,.$$
\end{lemma}

\begin{proof}
Let~$b\in(0,d)$ and~\smash{$\omega\in\acc{0,1}^{\En}$} such that~\smash{$\abs{B_n^b(\omega)}\geqslant 2^{d+1}n^b$}.
If~$\abs{C_{max}(\omega)}\ \geqslant\ 2n^b$, then we choose~$H=\varnothing$.
Let us now assume that~$\abs{C_{max}(\omega)}<2n^b$.
Then all the clusters in~$B_n^b(\omega)$ contain between~$n^b$ and~$2n^b$ vertices.
Therefore, there are at least~\smash{$\abs{B_n^b(\omega)}/(2n^b)$} such clusters in the configuration~$\omega$.
We divide the box~$\Lambda(n)$ into hypercubic boxes of side
$$L_n\ =\ \Ceil{2n\left(\frac{2n^b}{\abs{B_n^b(\omega)}}\right)^{1/d}}\,,$$
with boxes which may be smaller along the boundaries of~$\Lambda(n)$. The number of boxes is at most
$$\Ceil{\frac{n}{L_n}}^d
\ \leqslant\ \Ceil{\demi\left(\frac{\abs{B_n^b(\omega)}}{2n^b}\right)^{1/d}}^d
\ <\ \left[\demi\left(\frac{\abs{B_n^b(\omega)}}{2n^b}\right)^{1/d}+1\right]^d
\ \leqslant\ \frac{\abs{B_n^b(\omega)}}{2n^b}\,.$$
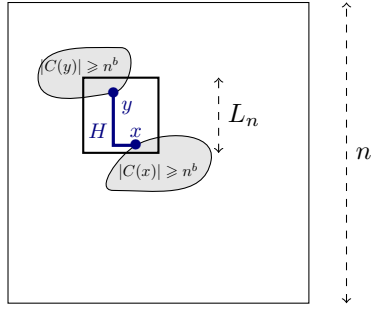
\begin{figure}
\begin{center}
\begin{tikzpicture}
\draw (0,0) -- (0,4) -- (4,4) -- (4,0) -- cycle;
\draw[<->,dashed] (4.5,0) -- node[midway,right]{$n$} (4.5,4);
\filldraw[fill=gray!20!white] (1.7,2.1) to[out=30, in=80] (2.7,1.9) to[out=-100,in=0] (1.4,1.5) to[out=180,in=-150] (1.7,2.1);
\filldraw[fill=gray!20!white] (1.4,2.8) to[out=10,in=0] (1.5,3.4) to[out=180,in=90] (0.4,3) to[out=-90,in=-170] (1.4,2.8);
\draw[thick] (1,2) --(2,2) -- (2,3) -- (1,3) -- cycle;
\draw[very thick,blue!50!black] (1.7,2.1) node{$\bullet$} node[above,scale=0.8]{$x$} -- (1.4,2.1) -- (1.4,2.8) node{$\bullet$} node[below right,scale=0.8]{$y$};
\draw[blue!50!black] (1.2,2.3) node[scale=0.8]{$H$};
\draw (2,1.75) node[scale=0.6]{$\abs{C(x)}\geqslant n^b$};
\draw (0.95,3.15) node[scale=0.6]{$\abs{C(y)}\geqslant n^b$};
\draw[<->,dashed] (2.8,2) -- node[midway,right]{$L_n$} (2.8,3);
\end{tikzpicture}
\caption{If there is no cluster with size~$\geqslant 2n^b$, we reconstitute one by joining two clusters of size at least~$n^b$, with a path using less than~$dL_n$ edges.}
\end{center}
\end{figure}

\noindent Hence, there are strictly less boxes than the number of clusters with size at least~$n^b$.
Therefore, by the pigeonhole principle, at least one of these boxes must intersect two such clusters, which means that we can find~$x,\,y\in\Lambda(n)$ such that
$$\abs{C(x,\,\omega)}\ \geqslant\ n^b\,,
\qquad
\abs{C(y,\,\omega)}\ \geqslant\ n^b\,,
\qquad
x\stackrel{\omega}{\centernot{\longleftrightarrow}}y
\qquadet
\norme{x-y}_{\infty}\ \leqslant\ L_n\,.$$
We then have~$\norme{x-y}_1\leqslant dL_n$, implying that there exists a path~$H\subset\En$ with at most~$dL_n$ edges which connects~$x$ and~$y$.
Opening the edges of this path in~$\omega$ creates a connection between two different clusters of size at least~$n^b$, whence~\smash{$\abs{C_{max}\big(\omega^H\big)}\geqslant 2n^b$}.
What's more, we have
$$\abs{H}
\ \leqslant\ dL_n
\ \leqslant\ d\left(\frac{2^{1+1/d}n^{1+b/d}}{\abs{B_n^b(\omega)}^{1/d}}+1\right)
\ \leqslant\ \big(2^{1+1/d}+1\big)d\frac{n^{1+b/d}}{\abs{B_n^b(\omega)}^{1/d}}
\ \leqslant\ 4d\frac{n^{1+b/d}}{\abs{B_n^b(\omega)}^{1/d}}\,,$$
which completes the proof of this lemma.
\end{proof}

The second geometrical lemma will tell us how many edges we need to close to adjust the size of~$B_n^b$:

\begin{lemma}
\label{lemmeSurgeryBnb}
Let~$b\in(0,d)$.
There exists~$K_1=K_1(d)>0$ such that, for~$n\geqslant 1$, for any configuration~\smash{$\omega\in\acc{0,1}^{\En}$} and any~$s\in\N$, if
\begin{equation}
\label{hyplem}
12n^b\ \leqslant\ s\ \leqslant\ \abs{B_n^b(\omega)}\ \leqslant\ s+6n^b
\qquadet
\abs{C_{max}(\omega)}\ \geqslant\ 2n^b\,,
\end{equation}
then there exists~$H\subset\En$ such that
$$\abs{B_n^b\big(\omega_H\big)}\ =\ s
\qquadet
\abs{H}\ \leqslant\ K_1s^{\frac{d-1}{d}}\,.$$
\end{lemma}

\souspreuve{Sketch of the proof:}
If we have at our disposal a big enough cluster, then we may reach~$\abs{B_n^b}=s$ by only closing edges of this cluster.
In this case, using the geometrical results of section~\ref{section_intermede}, the idea is to cut~$C_{max}(\omega)$ into one large piece of size~$m\geqslant n^b$ and remaining pieces all of size~$< n^b$ (see the left part of figure~\ref{surgeryBn}).
Adjusting the cutting so that~$m=\abs{C_{max}(\omega)}-\abs{B_n^b(\omega)}+s$ yields the desired result.
However, for this technique to work, we need~$m$ to be greater than~$n^b$, so that the large piece of size~$m$ still belongs to~$B_n^b$ after the cutting.
This is the case if~$\abs{C_{max}(\omega)}\geqslant 7n^b$, because then~$m\geqslant 7n^b-6n^b=n^b$.
In the case where~$\abs{C_{max}(\omega)}< 7n^b$, we proceed differently.
In this case, we first cut other intermediate clusters, to reach a situation where~$s\leqslant \abs{B_n^b}< s+n^b$ (see the right part of figure~\ref{surgeryBn}).
Then, we can use~$C_{max}(\omega)$ to reach exactly~$\abs{B_n^b}=s$, by disconnecting~$\abs{B_n^b}-s$ vertices from~$C_{max}(\omega)$.
Because we have assumed that~$\abs{C_{max}(\omega)}\geqslant 2n^b$ and~$\abs{B_n^b}-s<n^b$, we can ensure that the resulting cluster still contains at least~$n^b$ vertices, and thus still belongs to~$B_n^b$.

\begin{proof}
Let us now implement the strategy presented above.
Let~$b\in(0,d)$,~$n\geqslant 1$,~\smash{$\omega\in\acc{0,1}^{\En}$} and~$s\in\N$ such that~(\ref{hyplem}) holds.
We distinguish between two cases :

\point\textbf{First case:} Assume that~$\abs{C_{max}(\omega)}\geqslant 7n^b$.
Letting~$m=\abs{C_{max}(\omega)}-\abs{B_n^b(\omega)}+s$, we can use lemma~\ref{lemme_boucher} to find~$H_1\subset \En$ with
$$\abs{H_1}
\ \leqslant\ K\abs{C_{max}(\omega)}^{(d-1)/d}
\ \leqslant\ K\big(s+6n^b\big)^{(d-1)/d}
\ \leqslant\ 2Ks^{(d-1)/d}$$
and such that closing the edges of~$H_1$ divides~$C_{max}(\omega)$ into one connected component of size exactly~$m$, and one or several other pieces, whose total size is~$\abs{C_{max}(\omega)}-m=\abs{B_n^b(\omega)}-s\leqslant 6n^b$.
Using the butcher's lemma (lemma~\ref{lemme_sandwich_jambon}), this remaining part can be cut into pieces smaller than~$3n^b$.
Using again the butcher's lemma on the connected subpieces which contain strictly more than~$(3/2)n^b$ vertices (there are at most~$3$ such subpieces), we can cut them into pieces smaller than~$(3/2)n^b$.
Repeating the operation on the pieces containing strictly more than~$(3/4)n^b$ vertices (there are at most~$7$ such subpieces), we can cut them into pieces smaller than~$(3/4)n^b$.
Thus, using at most~$11$ times the butcher's lemma, we obtain~$H_2\subset\En$ such that~\smash{$\abs{H_2}\leqslant 11\times 4^{d+1}d^2\big(6n^b\big)^{(d-1)/d}$} and such that in the configuration~$\omega_{H_1\cup H_2}$, the vertices of~$C_{max}(\omega)$ are separated into one cluster of size exactly~$m$ and the remaining clusters which are all smaller than~$(3/4)n^b<n^b$ (see the left part of figure~\ref{surgeryBn}).
Therefore, writing~$H=H_1\cup H_2$ and using that
$$m
\ =\ \abs{C_{max}(\omega)}-\abs{B_n^b(\omega)}+s
\ \geqslant\ 7n^b-6n^b\ =\ n^b\,,$$
we obtain
$$\abs{B_n^b\big(\omega_{H}\big)}
\ =\ \abs{B_n^b(\omega)}-\abs{C_{max}(\omega)}+m
\ =\ s$$
and
$$\abs{H}\ \leqslant\ 2Ks^{(d-1)/d}+11\times 4^{d+1}d^2\big(6n^b\big)^{(d-1)/d}
\ \leqslant\ \big(2K+11\times 4^{d+1}d^2\big)s^{(d-1)/d}\,.$$
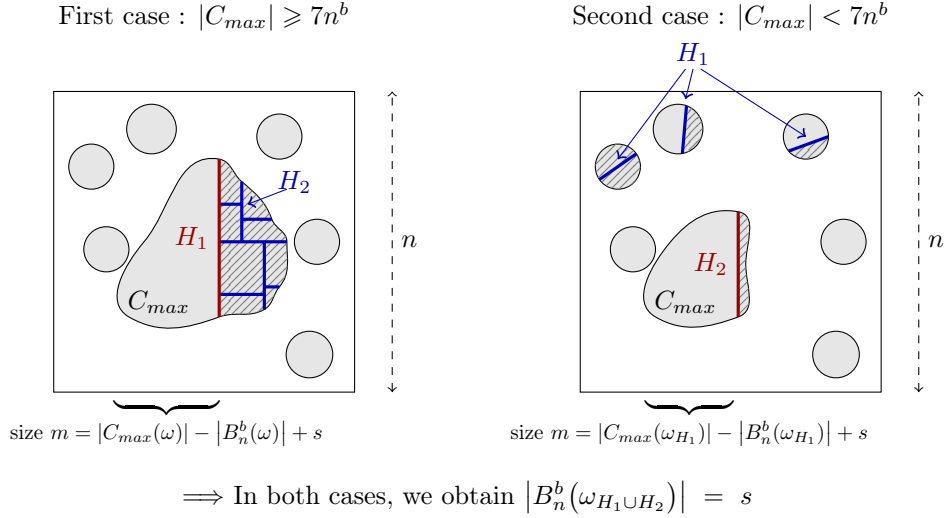
\begin{figure}
\begin{center}
\begin{tikzpicture}
\draw (0,0) -- (0,4) -- (4,4) -- (4,0) -- cycle;
\draw[<->,dashed] (4.5,0) -- node[midway,right]{$n$} (4.5,4);
\filldraw[fill=gray!20!white] (1,1) to[out=-30,in=-160] (2.2,1) to[out=20,in=-150] (2.8,1.1) to[out=30,in=-130] (3,1.4) to[out=50,in=-90] (3.1,2) to[out=90,in=-70] (2.9,2.3) to[out=110,in=-40] (2.5,2.8) to[out=140,in=-10] (2.2,3.1) to[out=170,in=50] (1.2,2) to[out=-130,in=150] (1,1);
\fill[pattern=north east lines, pattern color=gray] (2.2,1) to[out=20,in=-150] (2.8,1.1) to[out=30,in=-130] (3,1.4) to[out=50,in=-90] (3.1,2) to[out=90,in=-70] (2.9,2.3) to[out=110,in=-40] (2.5,2.8) to[out=140,in=-10] (2.2,3.1) to[out=-90,in=90] (2.2,1);
\draw (1.5,-0.2) node{$\underbrace{\qquad\qquad}$};
\draw (1.5,-0.55) node[scale=0.8]{size $m=\abs{C_{max}(\omega)}-\abs{B_n^b(\omega)}+s$};
\draw[very thick,blue!70!black] (3.1,2) -- (2.2,2) (2.8,1.1) -- (2.8,2) (2.5,2) -- (2.5,2.8) (3,1.4) -- (2.8,1.4) (2.9,2.3) -- (2.5,2.3) (2.2,1.3) -- (2.8,1.3) (2.2,2.5) --(2.5,2.5);
\draw[blue!70!black] (3.2,2.8) node{$H_2$};
\draw[blue!70!black,decoration={markings,mark=at position 1 with {\arrow[thick]{>}}},postaction={decorate}] (3.1,2.7) -- (2.58,2.5);
\draw[very thick,red!60!black] (2.2,1) -- node[midway,left]{$H_1$} (2.2,3.1);
\filldraw[fill=gray!20!white] (0.5,3) circle(0.3) (3,3.4) circle(0.3) (1.3,3.5) circle(0.33) (3.5,2) circle(0.3) (3.4,0.5) circle(0.31) (0.7,1.9) circle(0.3);
\draw (1.4,1.2) node{$C_{max}$};
\draw (2,5) node{First case : $\abs{C_{max}}\geqslant 7n^b$};
\draw (5.5,-1.4) node{$\Longrightarrow$ In both cases, we obtain $\abs{B_n^b\big(\omega_{H_1\cup H_2}\big)}\ =\ s$};

\begin{scope}[xshift=7cm]
\draw (0,0) -- (0,4) -- (4,4) -- (4,0) -- cycle;
\draw[<->,dashed] (4.5,0) -- node[midway,right]{$n$} (4.5,4);
\filldraw[fill=gray!20!white] (1,1) to[out=-30,in=-160] (2.1,1) to[out=20,in=-90] (2.2,1.7) to[out=90,in=-10] (2.1,2.4) to[out=170,in=50] (1.2,2) to[out=-130,in=150] (1,1);
\draw (1.45,-0.2) node{$\underbrace{\qquad\quad\ }$};
\draw (1.45,-0.55) node[scale=0.8]{size $m=\abs{C_{max}(\omega_{H_1})}-\abs{B_n^b(\omega_{H_1})}+s$};
\fill[pattern=north east lines, pattern color=gray] (2.1,1) to[out=20,in=-90] (2.2,1.7) to[out=90,in=-10] (2.1,2.4) to[out=-90,in=90] (2.1,1);
\draw[very thick,red!60!black] (2.1,1) -- node[midway,left]{$H_2$} (2.1,2.4);
\filldraw[fill=gray!20!white] (0.5,3) circle(0.3) (3,3.4) circle(0.3) (1.3,3.5) circle(0.33) (3.5,2) circle(0.3) (3.4,0.5) circle(0.31) (0.7,1.9) circle(0.3);
\fill[pattern=north east lines, pattern color=gray] (0.5,3) circle(0.3);
\draw[very thick,blue!70!black] ({0.5-0.3*cos(35)},{3-0.3*sin(35)}) -- ({0.5+0.3*cos(35)},{3+0.3*sin(35)});
\fill[pattern=north east lines, pattern color=gray] ({1.3-0.33*cos(100)},{3.5-0.33*sin(100)}) -- ({1.3+0.33*cos(70)},{3.5+0.33*sin(70)}) arc(70:-80:0.33);
\draw[very thick,blue!70!black] ({1.3-0.33*cos(100)},{3.5-0.33*sin(100)}) -- ({1.3+0.33*cos(70)},{3.5+0.33*sin(70)});
\fill[pattern=north east lines, pattern color=gray] (3.3,3.4) -- ({3+0.3*cos(-140)},{3.4+0.3*sin(-140)}) arc(-140:0:0.3);
\draw[very thick,blue!70!black] (3.3,3.4) -- ({3+0.3*cos(-140)},{3.4+0.3*sin(-140)});
\draw[blue!70!black] (1.5,4.5) node{$H_1$};
\draw[blue!70!black,decoration={markings,mark=at position 1 with {\arrow[thick]{>}}},postaction={decorate}] (1.4,4.3) -- (0.5,3.1);
\draw[blue!70!black,decoration={markings,mark=at position 1 with {\arrow[thick]{>}}},postaction={decorate}] (1.5,4.3) -- (1.4,3.9);
\draw[blue!70!black,decoration={markings,mark=at position 1 with {\arrow[thick]{>}}},postaction={decorate}] (1.6,4.3) -- (3,3.4);
\draw (1.4,1.2) node{$C_{max}$};
\draw (2,5) node{Second case : $\abs{C_{max}}< 7n^b$};
\end{scope}
\end{tikzpicture}
\caption{\label{surgeryBn}First case: when~$\abs{C_{max}}\geqslant 7 n^b$ (picture on the left), we cut a piece of~$C_{max}$ with the desired size~$m$ (by closing~$H_1$) and we divide the remaining part of~$C_{max}$ into pieces smaller than~$n^b$ (by closing~$H_2$). Second case: when~$\abs{C_{max}}<7 n^b$ (picture on the right), we cut some intermediate clusters (by closing~$H_1$) and we cut a piece of~$C_{max}$ with the desired size~$m$ (by closing~$H_2$). In both figures, the hatched region is~$B_n^b(\omega)\backslash B_n^b(\omega_{H_1\cup H_2})$, i.e., the vertices whose cluster is no longer larger than~$n^b$ after the cutting procedure.}
\end{center}
\end{figure}

\point \textbf{Second case:} Now assume that~$2n^b\leqslant \abs{C_{max}(\omega)}< 7n^b$.
The first step in this case is to find~$H_1\subset\En$ such that
\begin{equation}
\label{objH1}
s\ \leqslant\ \abs{B_n^b\big(\omega_{H_1}\big)}\ <\ s+n^b\,.
\end{equation}
If we already have~$\abs{B_n^b(\omega)}< s+n^b$, then we take~$H_1=\varnothing$.
Assume that~$\abs{B_n^b(\omega)}\geqslant s+n^b$.
As explained in the sketch of the proof, the idea to obtain~(\ref{objH1}) is to cut one or several intermediate clusters. By intermediate, we mean clusters containing at least~$n^b$ vertices but which are distinct from~$C_{max}$.
Each cutting of an intermediate cluster will have to either yield directly~(\ref{objH1}) or to decrease~$\abs{B_n^b}$ by at least~\smash{$\Ceil{n^b}-1$}, maintaining~$\abs{B_n^b}\geqslant s$.
Because we start from~$s+n^b\leqslant\abs{B_n^b(\omega)}\leqslant s+6n^b$, at most~$6$ such cuttings are necessary to eventually obtain~(\ref{objH1}).
Let us now detail the cutting we perform on these intermediate clusters.
We look for~$H_0\subset\En$ such that
\begin{equation}
\label{objHH2}
s\ \leqslant\ \abs{B_n^b\big(\omega_{H_0}\big)}\ \leqslant\ \Big(\abs{B_n^b(\omega)}-\big(\Ceil{n^b}-1\big)\Big)\vee\big(s+\Ceil{n^b}-1\big)\,.
\end{equation}
Notice that
$$\abs{B_n^b(\omega)}
\ \geqslant\ s
\ \geqslant\ 12n^b
\ >\ \abs{C_{max}(\omega)}\,,$$
which allows us to choose~$x\in B_n^b(\omega)\backslash C_{max}(\omega)$, meaning that the cluster~$C_x=C(x,\,\omega)$ contains at least~\smash{$\Ceil{n^b}$} vertices but it is not the selected largest cluster (it is what we call an intermediate cluster).
We now distinguish between several subcases depending on the size of this cluster~$C_x$:

\noindent $\circ\ $ If~$\abs{C_x}\geqslant 2\Ceil{n^b}-2$, then lemma~\ref{lemme_boucher} provides us with~$H_0\subset\En$ such that
$$\big|C\big(x,\,\omega_{H_0}\big)\big|\ =\ \abs{C_x}-\big(\Ceil{n^b}-1\big)\,.$$
We then have
$$\big|B_n^b\big(\omega_{H_0}\big)\big|\ =\ \abs{B_n^b(\omega)}-\big(\Ceil{n^b}-1\big)\,.$$

\noindent $\circ\ $ If~$\Ceil{n^b}\leqslant\abs{C_x}<2\Ceil{n^b}-2$ and~$\abs{B_n^b(\omega)}-\abs{C_x}\geqslant s$, then we choose a set of edges~$H_0\subset\En$ such that~$\big|C\big(x,\,\omega_{H_0}\big)\big|=\Ceil{n^b}-1$.
In this case, we get
$$\big|B_n^b\big(\omega_{H_0}\big)\big|\ =\ \abs{B_n^b(\omega)}-\abs{C_x}
\ \leqslant\ \abs{B_n^b(\omega)}-\Ceil{n^b}\,.$$

\noindent $\circ\ $ Otherwise, if~$\Ceil{n^b}\leqslant\abs{C_x}<2\Ceil{n^b}-2$ and~$\abs{B_n^b(\omega)}-\abs{C_x}<s$, then we choose~$H_0$ such that~$\big|C\big(x,\,\omega_{H_0}\big)\big|=\Ceil{n^b}$, which entails that
$$\big|B_n^b\big(\omega_{H_0}\big)\big|
\ =\ \abs{B_n^b(\omega)}-\big(\abs{C_x}-\Ceil{n^b}\big)\ <\ s+\Ceil{n^b}\,.$$

\noindent In all three cases,~$H_0$ satisfies~(\ref{objHH2}) and lemma~\ref{lemme_boucher} ensures that~$H_0$ can be chosen with
$$\abs{H_0}
\ \leqslant\ K\abs{C_x}^{(d-1)/d}
\ \leqslant\ K\abs{C_{max}(\omega)}^{(d-1)/d}
\ \leqslant\ K\big(7n^b\big)^{(d-1)/d}
\ \leqslant\ Ks^{b(d-1)/d}\,.$$
After closing this set of edges~$H_0$, we still have~$\abs{B_n^b}\geqslant s$ and, either~$\abs{B_n^b}<s+n^b$ or~$\abs{B_n^b}$ has decreased by at least~$n^b-1$.
Therefore, we can repeat this operation, and after at most 6 steps, we obtain~\smash{$s\leqslant \abs{B_n^b}< s+n^b$}.
Thus, we end up with~$H_1\subset\En$ satisfying~(\ref{objH1}) and such that~\smash{$\abs{H_1}\leqslant 42Kn^{b(d-1)/d}$}.
As we have not touched~$C_{max}(\omega)$ during this procedure, we still have
$$2n^b\ \leqslant\ \big|C_{max}\big(\omega_{H_1}\big)\big|\ <\ 7n^b\,.$$
Letting now
$$m\ =\ \big|C_{max}\big(\omega_{H_1}\big)\big|-\big|B_n^b\big(\omega_{H_1}\big)\big|+s\,,$$
it follows from~(\ref{objH1}) that
$$n^b
\ \leqslant\ m\ \leqslant\ \big|C_{max}\big(\omega_{H_1}\big)\big|\,.$$
Hence, using again lemma~\ref{lemme_boucher}, we can find~$H_2\subset\En$ with~\smash{$\abs{H_2}\leqslant K\big(7n^b\big)^{(d-1)/d}\leqslant K s^{(d-1)/d}$} and such that closing the edges of~$H_2$ divides the cluster~$C_{max}\big(\omega_{H_1}\big)$ into one connected component of size exactly~$m$ and one or several other connected components, whose total size is
$$\big|C_{max}\big(\omega_{H_1}\big)\big|-m
\ =\ \big|B_n^b\big(\omega_{H_1}\big)\big|-s
\ <\ n^b\,.$$
Therefore, writing~$H=H_1\cup H_2$, we have~\smash{$\abs{H}\leqslant 2Ks^{(d-1)/d}$} and
$$\big|B_n^b\big(\omega_H\big)\big|
\ =\ \big|B_n^b\big(\omega_{H_1}\big)\big|-\big|C_{max}\big(\omega_{H_1}\big)\big|+m
\ =\ s\,.$$

\noindent In both cases, we obtain the claimed result, with~$K_1=2K+11\times 4^{d+1}d^2$.
\end{proof}

We are now in a position to prove our second lower bound on~$Z_n$.

\begin{proof}[Proof of lemma~\ref{lemme_minoration_Z_n_bis}]
As explained above, we define two couplings, in order to be able not only to close edges, but also to reopen edges.

\souspreuve{Definition of the two couplings:}
The first coupling is defined as in the previous proofs. We write~$\En = \acc{e_1,\,\ldots,\,e_r}$ with~$r=\abs{\En}$, and we consider i.i.d.\ random variables
$$\left(X_{t,e}\right)_{t\in\acc{0,\,\ldots,\,n^d-1},\,e\in \En}$$
with Bernoulli law of parameter~$\exp(-1/n^a)$.
For~$t_0\in\acc{0,\,\ldots,\,n^d}$, we define
$$\omega(t_0)\ :\ e\in \En\ \longmapsto\ \min\limits_{0\leqslant t<t_0}\,X_{t,e}\,.$$
In addition to this, we draw a uniform random~$M\in\acc{0,\,\ldots,\,n^d}$, uniform independent edges~$\varepsilon_1,\,\ldots,\,\varepsilon_M\in\En$ and i.i.d.\ random variables~$\left(X'_{t,e}\right)_{t\leqslant n^d-1,\,e\in \En}$ again with Bernoulli law of parameter~$\exp(-1/n^a)$.
The second coupling of configurations is then defined by
$$\forall\, t_0\in\acc{0,\,\ldots,\,n^d}\qquad
\omega'(t_0)\ :\ e\in \En\ \longmapsto\ \left\{\begin{aligned}
&\min\limits_{0\leqslant t<t_0}\,X'_{t,e}\quad\text{if }e\in\acc{\varepsilon_1,\,\ldots,\,\varepsilon_M}\,,\\
&\omega(t_0)\quad\text{otherwise,}
\end{aligned}\right.$$
with again intermediate configurations defined for all~$t\in\acc{0,\,\ldots,\,n^d-1}$ and~$s_0\in\acc{0,\,\ldots,\,r}$ by
$$
\omega'(t,\,s_0)\ :\ e_s\in \En\ \longmapsto\ 
\left\{
\begin{aligned}
&\omega'(t+1)(e_s)\ &\text{if }s\leqslant s_0\,,\\
&\omega'(t)(e_s)\ &\text{otherwise.}
\end{aligned}
\right.$$
Hence, the two decreasing couplings have the same law, with~$\omega(t)\stackrel{d}{=}\omega'(t)\stackrel{d}{=}\Proba_{\varphi_n(t)}$, and the second coupling differs from the first one only on the edges~$\varepsilon_1,\,\ldots,\,\varepsilon_M$.
This set of edges is chosen at random, but we will be interested in the event that~$\acc{\varepsilon_1,\,\ldots,\,\varepsilon_M}=H_1\cup H_2$, where~$H_1$ is a set of edges which we want to leave open longer in the second coupling, and~$H_2$ is a set of edges which we want to close sooner in the second coupling.
Thus, this double coupling will allow us to perform the surgery procedure of lemmas~\ref{lemmeCmaxdouble} (which involves opening edges) and~\ref{lemmeSurgeryBnb} (which involves closing edges), starting from a given configuration in the first coupling.
Note that, instead of constructing such a double coupling, we could also have used the standard estimate of, for example, lemma~6.3 in~\cite{CerfPisztora2000}, about the price to open or close specific edges.
\\

\souspreuve{Reconstitution of a big enough cluster:}
As in the proof of lemma~\ref{lemme_minoration_Z_n}, we consider the instant~$(T,\,S)$ defined by~(\ref{defTS}), which is such that
$$T+1\ \leqslant\ \abs{B_n^b\big(\omega(T,\,S)\big)}\ \leqslant\ T+1+2n^b\,.$$
With~$\tau_n^+$ defined as in~(\ref{defTau}) and~$\tau_n^-$ given by
$$\tau_n^-\ =\ \Ent{n^a\left(-\ln\left(\frac{p_c+1}{2}\right)\right)}\,,$$
we have
\begin{align*}
\Proba\big(T\notin[\tau_n^-,\,\tau_n^+]\big)
\ &\leqslant\ \Proba\Big(\abs{B_n^b\big(\omega(\tau_n^-)\big)}\leqslant \tau_n^- +1+2n^b\quador \abs{B_n^b\big(\omega(\tau_n^+)\big)}> \tau_n^++1+2n^b\Big)\\
\ &\leqslant\ \Proba_{(p_c+1)/2}\Big(\abs{B_n^b}\leqslant \tau_n^-+1+2n^b\Big)+\Proba_{p_c/2}\Big(\abs{B_n^b}> \tau_n^++1+2n^b\Big)
\ \cvninfty\ 0\,,\numberthis\label{Ptn}
\end{align*}
thanks to lemmas~\ref{lemme_majo_souscritique} and~\ref{lemme_majo_surcritique}.
This allows us, in the sequel, to reason conditionally on the event that~$\tau_n^-\leqslant T\leqslant \tau_n^+$.
Thus, we have
$$\abs{B_n^b\big(\omega(T,\,S)\big)}
\ \geqslant\ T+1
\ \geqslant\ \tau_n^-+1
\ \geqslant\ 2^{d+1}n^b$$
for~$n$ large enough, given that~$a>b$.
This allows us to apply lemma~\ref{lemmeCmaxdouble}, which provides us with~$H_1=H_1\big(\omega(T,\,S)\big)\subset\En$ such that the configuration~$\omega(T,\,S)^{H_1}$, where the edges of~$H_1$ are reopened, contains a cluster with at least~$2n^b$ vertices, and such that
\begin{equation}
\label{boundH1}
\abs{H_1}
\ \leqslant\ 4d\frac{n^{1+b/d}}{\abs{B_n^b\big(\omega(T,\,S)\big)}^{1/d}}
\ \leqslant\ 4d\frac{n^{1+b/d}}{(\tau_n^-+1)^{1/d}}
\ \leqslant\ K'n^{1+b/d-a/d}\,,
\end{equation}
where~$K'$ is a positive constant.
We can choose this~$H_1$ minimal in the sense of inclusion, so that either~$H_1=\varnothing$ if we already had~$\abs{C_{max}\big(\omega(T,\,S)\big)}\geqslant 2n^b$ or all the edges of~$H_1$ belong to~\smash{$C_{max}\big(\omega(T,\,S)^{H_1}\big)$} and moreover~\smash{$C_{max}\big(\omega(T,\,S)^{H_1}\big)$} contains at most~$4n^b$ vertices (otherwise a smaller set~$H_1$ would work), so that in both cases,
$$T+1\ \leqslant\ \abs{B_n^b\big(\omega(T,\,S)^{H_1}\big)}
\ \leqslant\ \abs{B_n^b\big(\omega(T,\,S)\big)}+4n^b
\ \leqslant\ T+1+6n^b\,.$$

\souspreuve{Surgery step:}
The next step is to find a set of edges~$H_2=H_2\big(\omega(T,\,S)^{H_1}\big)\subset\En$ such that
\begin{equation}
\label{objH2}
\big|B_n^b\big(\big(\omega(T,\,S)^{H_1}\big)_{H_2}\big)\big|\ =\ T+1\,.
\end{equation}
Still assuming~$\tau_n^-\leqslant T\leqslant \tau_n^+$, applying lemma~\ref{lemmeSurgeryBnb} with~$s=T+1\geqslant \tau_n^-\geqslant 12n^b$ (for~$n$ large enough), we can construct~$H_2\subset\En$, which can be defined as a function of the configuration~$\omega(T,\,S)^{H_1}$, satisfying~(\ref{objH2}) and whose cardinality is bounded by
\begin{equation}
\label{boundH2}
\abs{H_2}\ \leqslant\ K_1\big(T+1\big)^{(d-1)/d}
\ \leqslant\ K_1\big(\tau_n^++1\big)^{(d-1)/d}
\ \leqslant\ K''n^{a(d-1)/d}\,,
\end{equation}
where~$K''$ is a positive constant, since~$\tau_n^+=O(n^a)$.
\\

\souspreuve{The happy event:}
We now consider the event (where~$H_1=H_2=\varnothing$ if~$T\notin[\tau_n^-,\,\tau_n^+]$)
$$\mathcal{E}\ =\ \acc{\begin{aligned}
&M=\abs{H_1\cup H_2}\,,\quad
\acc{\varepsilon_1,\,\ldots,\,\varepsilon_M}=H_1\cup H_2\,,\\
&\forall s>S\qquad e_s\in \Edge{B_n^b\big(\omega(T,\,S)\big)}\ \Rightarrow\ X_{T,\,e_s}=1\,,\\
&\forall e\in H_1\backslash H_2\quad\forall t\in\acc{0,\,\ldots,\,T}\qquad X'_{t,\,e}=1\,,\\
&\forall e\in H_2\qquad X'_{0,\,e}=0
\end{aligned}}\,.$$
If this event occurs and~$\tau_n^-\leqslant T\leqslant \tau_n^+$, then we have
$$\abs{B_n^b\big(\omega'(T+1,\,0)\big)}\ =\ \big|B_n^b\big(\big(\omega(T,\,S)^{H_1}\big)_{H_2}\big)\big|\ =\ T+1\,,$$
whence
\begin{equation}
\label{eq1679}
Z_n\ \geqslant\ \Proba\Big(\,\abs{B_n^b\big(\omega'(T+1,\,0)\big)}=T+1\,\Big)
\ \geqslant\ \Proba\Big(\,\mathcal{E}\cap\acc{\tau_n^-\leqslant T\leqslant \tau_n^+}\,\Big)\,.
\end{equation}
As in the proof of lemma~\ref{Cmax_lemme_minoration_Z_n}, we now take~$(t_0,\,s_0)$ and~\smash{$\omega_0\in\acc{0,1}^{\En}$} such that~$\tau_n^-\leqslant t_0\leqslant \tau_n^+$ and
$$\Proba\big(\mathcal{C}_{t_0,\,s_0,\,\omega_0}\big)\ >\ 0
\qquadou
\mathcal{C}_{t_0,\,s_0,\,\omega_0}\ =\ \Big\{\,(T,\,S)=(t_0,\,s_0)\quadet\omega(T,\,S)=\omega_0\,\Big\}\,.$$
Because~$H_1$ and~$H_2$ only depend on~$T$ and~$\omega(T,\,S)$, we may consider the deterministic sets~$H_1$ and~$H_2$ associated with~$T=t_0$ and~$\omega(T,\,S)=\omega_0$.
Then we consider the event
$$\widetilde{\mathcal{E}}_{t_0,\,s_0,\,\omega_0}\ =\ \acc{\begin{aligned}
&M=\abs{H_1\cup H_2}\,,\quad
\acc{\varepsilon_1,\,\ldots,\,\varepsilon_M}=H_1\cup H_2\,,\\
&\forall s>s_0\qquad e_s\in \Edge{B_n^b\big(\omega_0\big)}\ \Rightarrow\ X_{t_0,\,e_s}=1\,,\\
&\forall e\in H_1\backslash H_2\quad\forall t\in\acc{0,\,\ldots,\,t_0}\qquad X'_{t,\,e}=1\,,\\
&\forall e\in H_2\qquad X'_{0,\,e}=0
\end{aligned}}\,,$$
which is independent of~$\mathcal{C}_{t_0,\,s_0,\,\omega_0}$ because~$\mathcal{C}_{t_0,\,s_0,\,\omega_0}$ depends only on the variables~$X_{t,\,e_s}$ with~$(t,\,s)\leqslant(t_0,\,s_0)$.
Therefore, we can write
\begin{align*}
\Proba\big(\mathcal{E}\,\big|\,\mathcal{C}_{t_0,\,s_0,\,\omega_0}\big)
\ &=\ \Proba\left(\widetilde{\mathcal{E}}_{t_0,\,s_0,\,\omega_0}\,\Big|\,\mathcal{C}_{t_0,\,s_0,\,\omega_0}\right)
\ =\ \Proba\left(\widetilde{\mathcal{E}}_{t_0,\,s_0,\,\omega_0}\right)\\
\ &=\ \frac{1}{n^d+1}\left(\frac{1}{n^d}\right)^{\abs{H_1}+\abs{H_2}}
\prod_{\substack{s>s_0\\ e_s\in\Edge{B_n^b(\omega_0)}}}{\Proba\big(\,X_{t_0,\,e_s}=1\,\big)}
\prod_{e\in H_1}\prod_{t=0}^{t_0}{\Proba\big(\,X'_{t,\,e}=1\,\big)}
\prod_{e\in H_2}\Proba\big(\,X'_{0,\,e}=0\,\big)\\
\ &\geqslant\ \left(\frac{1}{n^d+1}\right)^{1+\abs{H_1}+\abs{H_2}}\big(e^{-1/n^a}\big)^{\abs{\Edge{B_n^b(\omega_0)}}+\abs{H_1}(t_0+1)}\big(1-e^{-1/n^a}\big)^{\abs{H_2}}\,.
\end{align*}
We now use the bounds~(\ref{boundH1}) and~(\ref{boundH2}) on~$\abs{H_1}$ and~$\abs{H_2}$, the upper bound
$$\abs{\Edge{B_n^b(\omega_0)}}
\ \leqslant\ d\abs{B_n^b(\omega_0)}
\ \leqslant\ d\big(t_0+1+2n^b\big)
\ \leqslant\ d\big(\tau_n^++1+2n^b\big)$$
and the fact that, for~$n$ large enough,~$1-e^{-1/n^a}\geqslant 1/(n^d+1)$, to obtain
$$\Proba\big(\mathcal{E}\,\big|\,\mathcal{C}_{t_0,\,s_0,\,\omega_0}\big)
\ \geqslant\ \left(\frac{1}{n^d+1}\right)^{1+K'n^{1+b/d-a/d}+2K''n^{a(d-1)/d}}\exp\left(-\frac{d(\tau_n^++1+2n^b)+(\tau_n^++1)K'n^{1+b/d-a/d}}{n^a}\right)\,.$$
This bound being uniform with respect to~$t_0$,~$s_0$ and~$\omega_0$ (as long as~$\tau_n^-\leqslant t_0\leqslant \tau_n^+$), we obtain that
$$\Proba\big(\,\mathcal{E}\,\big|\,\tau_n^-\leqslant T\leqslant \tau_n^+\,\big)
\ \geqslant\ \left(\frac{1}{n^d+1}\right)^{1+K'n^{1+b/d-a/d}+2K''n^{a(d-1)/d}}\exp\left(-\frac{(d+K'n^{1+b/d-a/d})(\tau_n^++1)+2dn^b}{n^a}\right)\,.$$
Plugging this into~(\ref{eq1679}) and recalling~(\ref{Ptn}), we get
\begin{align*}
\ln Z_n
\ &\geqslant\ \ln \Proba\big(\tau_n^-\leqslant T\leqslant \tau_n^+\big)
-\big(1+K'n^{1+b/d-a/d}+2K''n^{a(d-1)/d}\big)\ln\big(n^d+1\big)\\
&\qquad\qquad-\frac{\big(d+K'n^{1+b/d-a/d}\big)\big(\tau_n^++1\big)+2dn^b}{n^a}\\
\ &=\ o(1)+O\big(n^{1+b/d-a/d}\ln n\big)+O\big(n^{a(d-1)/d}\ln n\big)
\ =\ O\big(n^c\ln n\big)\,,
\end{align*}
with~$c=(1+b/d-a/d)\vee(a-a/d)$.
\end{proof}

\subsection{Proof of the convergence result}

We now obtain the third case of theorem~\ref{thm_clusters_convergence}, proceeding as in section~\ref{section_Cmax_preuve_resultat}.

\begin{proof}[Proof of theorem~\ref{thm_clusters_convergence}, case~($iii$)]
Let~$\varepsilon>0$ and~$0<b<a<d$.
The upper bound
\begin{equation}
\label{prthm5}
\limsupn\,\frac{1}{n^a}\,\ln\mu_n\big(p_n<p_c-\varepsilon\big)
\ <\ 0
\end{equation}
follows from the exponential decay in the subcritical regime (lemma~\ref{lemme_majo_souscritique}) and the lower bound on~$Z_n$ given by lemma~\ref{lemme_minoration_Z_n}, using that~$b<a$.
Similarly, lemma~\ref{lemme_majo_surcritique} together with the other lower bound on~$Z_n$ given by lemma~\ref{lemme_minoration_Z_n_bis} implies
\begin{equation}
\label{prthm6}
\limsupn\,\frac{1}{n^{d-b/d}}\,\ln\mu_n\big(p_n>p_c+\varepsilon\big)
\ <\ 0\,,
\end{equation}
using that
$$c\ =\ \left(1-\frac{a}{d}+\frac{b}{d}\right)\vee\left(a-\frac{a}{d}\right)
\ <\ 1\vee\left(d-\frac{a}{d}\right)
\ <\ d-\frac{b}{d}\,.$$
To obtain a lower bound on~\smash{$\mu_n\big(p_n< p_c-\varepsilon\big)$}, we use the same technique as in section~\ref{section_Cmax_preuve_resultat}, choosing this time a parameter~$a'\in(a,\,2a-b)$.
Using the notations of the proof of lemma~\ref{lemme_minoration_Z_n} and~$t_n^-=\Ent{n^a\big(-\ln(p_c-\varepsilon)\big)}$, we write
$$\mu_n\big(p_n< p_c-\varepsilon\big)
\ \geqslant\ \Proba\Big(\,\mathcal{E}\,\cap\,\acc{T\geqslant t_n^-}\,\Big)
\ \geqslant\ \Proba\Big(\,\mathcal{E}\,\cap\,\big\{t_n^-\leqslant T\leqslant n^{a'}\big\}\,\Big)\,.$$
As in section~\ref{section_Cmax_preuve_resultat}, using that~$\varphi_n(t_n^-)\geqslant p_c-\varepsilon$ and~\smash{$\varphi_n\big(n^{a'}\big)<p_c/2$} for~$n$ large enough, and using the exponential estimate of lemma~\ref{lemme_majo_souscritique}, we have, for~$n$ large enough,
$$\Proba\big(t_n^-\leqslant T\leqslant n^{a'}\big)
\ \geqslant\ \Proba_{p_c-\varepsilon}\Big(\,\abs{B_n^b}>t_n^-+1+2n^b\,\Big)
-\Proba_{p_c/2}\Big(\,\abs{B_n^b}>n^{a'}\,\Big)
\ \geqslant\ e^{-Cn^a}-e^{-C'n^{a'}}
\ \geqslant\ \frac{e^{-Cn^a}}{2}\,.$$
Combining this with the lower bound~(\ref{eq761}) on the conditional probability~\smash{$\Proba\big(\mathcal{E}\,\big|\,T\big)$}, we obtain
\begin{align*}
\frac{1}{n^a}\ln\mu_n\big(p_n< p_c-\varepsilon\big)
\ &\geqslant\ -C-\frac{\ln 2}{n^a}-\frac{\big(2n^b+1\big)d\big(n^{a'}+1+2n^b\big)}{n^{2a}}\\
\ &=\ -C-\frac{\ln 2}{n^a}+O\left(\frac{1}{n^{2a-b-a'}}\right)
\ \cvninfty\ -C
\ >\ -\infty\,.\numberthis\label{prthm7}
\end{align*}
We now turn to the other lower bound. With~$t_n^+$ given by~(\ref{defTn}) we have
\begin{equation}
\label{eq6628}
\mu_n\big(p_n> p_c+\varepsilon\big)
\ \geqslant\ \Proba_{\varphi_n(t_n^+-1)}\Big(\abs{B_n^b}=t_n^+-1 \Big)\,.
\end{equation}
We now consider a configuration~\smash{$\omega\in\acc{0,1}^{\En}$} such that~$\abs{B_n^b(\omega)}\geqslant t_n^+-1$, and we use again the surgery procedure detailed in section~\ref{sectionBnminoZn} to force~$\abs{B_n^b}=t_n^+-1$.
Dividing the box~$\Lambda(n)$ into hypercubic boxes of side~\smash{$N_n=\Ceil{n^{b/d}}-1$} (and hence of volume~$<n^b$), and closing all the edges on the boundaries of these boxes, we can find~$H_1\subset\En$ such that~$B_n^b\big(\omega_{H_1}\big)=\varnothing$ and
$$\abs{H_1}\ =\ O\left(\frac{n}{N_n}n^{d-1}\right)\ =\ O\big(n^{d-b/d}\big)\,.$$
Rather than closing all this set of edges~$H_1$, we choose a maximal subset~$H_2\subset H_1$ such that~\smash{$\abs{B_n^b\big(\omega_{H_2}\big)}\geqslant t_n^+-1$}.
The maximality of~$H_2$ then ensures that
$$t_n^+-1\ \leqslant\ \abs{B_n^b\big(\omega_{H_2}\big)}\ \leqslant\ t_n^+-1+2n^b\,,$$
since closing one edge can at most diminish~$\abs{B_n^b}$ by~$2n^b$, as explained in the proof of lemma~\ref{lemme_minoration_Z_n}.
Using now the geometrical lemmas~\ref{lemmeCmaxdouble} and~\ref{lemmeSurgeryBnb}, we can find~$H_3,\,H_4\subset \En$ such that
$$\abs{B_n^b\big(\big((\omega_{H_2})^{H_3}\big)_{H_4}\big)}
\ =\ t_n^+-1
\qquadavec
\abs{H_3}\ =\ O\big(n^{1+b/d-a/d}\big)
\qquadet
\abs{H_4}\ =\ O\big(n^{a(d-1)/d}\big)\,.$$
Following lemma~6.3 in~\cite{CerfPisztora2000}, we obtain
\begin{align*}
\Proba_{\varphi_n(t_n^+-1)}\Big(\abs{B_n^b}=t_n^+-1 \Big)
\ &\geqslant\ \left(\frac{1}{O(n^d)}\right)^{\abs{H_2}+\abs{H_3}+\abs{H_4}}
\Proba_{\varphi_n(t_n^+-1)}\Big(\abs{B_n^b}\geqslant t_n^+-1 \Big)\\
\ &\geqslant\ \exp\Big(-O\big((\ln n)n^{d-b/d}\big)\Big)\Proba_{p_c+\varepsilon}\Big(\abs{B_n^b}\geqslant t_n^+-1 \Big)\,.
\end{align*}
Plugging this into~(\ref{eq6628}) and recalling that the probability on the right-hand side tends to~$1$ according to lemma~\ref{lemme_majo_surcritique}, we obtain
\begin{equation}
\label{prthm8}
\liminfn\,\frac{1}{(\ln n)n^{d-b/d}}\ln\mu_n\big(p_n> p_c+\varepsilon\big)
\ >\ -\infty\,.
\end{equation}
The last case of theorem~\ref{thm_clusters_convergence} then follows from~(\ref{prthm5}),~(\ref{prthm6}),~(\ref{prthm7}) and~(\ref{prthm8}).
\end{proof}

\subsection{A control on the convergence speed}
\label{prthmvitesse}

We now make the previous arguments more precise, in order to obtain an estimate on the convergence speed of~$p_n$ towards the critical point~$p_c$. We assume that there exist real numbers~$\beta>0$,~$\gamma>0$ such that
$$\limsup\limits_{\substack{p\rightarrow p_c\\p>p_c}}\,\,\frac{\ln\theta(p)}{\ln(p-p_c)}\ \leqslant\ \beta
\qquadet
\liminf\limits_{\substack{p\rightarrow p_c\\p<p_c}}\,\,\frac{\ln\chi(p)}{\ln(p_c-p)}\ \geqslant\ -\gamma\,,$$
and we fix for all this part some real numbers~$a$,~$b$ and~$c$ such that
$$0\ <\ b\ <\ a\ <\ d
\qquadet
0\ <\ c\ <\ {\rm min}\left(\frac{b}{2\gamma},\,\frac{a-b}{2\gamma},\,\frac{d-a}{\beta},\,\frac{d-bd-b}{2\beta}\right)\,.$$
We also choose~$\beta'$ and~$\gamma'$ such that
$$\beta\ <\ \beta'\ <\ \frac{1-b}{c}\wedge\frac{d-a}{c}
\qquadet
\gamma\ <\ \gamma'\ <\ \frac{b}{2c}\wedge\frac{a-b}{2c}\,.$$
Therefore, we can find~$\varepsilon_0>0$ such that, for~$0<\varepsilon<\varepsilon_0$,
$$\theta\left(p_c+\varepsilon\right)\ \geqslant\ \varepsilon^{\beta'}
\qquadet
\chi\left(p_c-\varepsilon\right)\ \leqslant\ \frac{1}{\varepsilon^{\gamma'}}\,.$$

\subsubsection{Subcritical phase}
\begin{lemma}
\label{vitesse_souscrit}
We have the upper bound
$$\forall\varepsilon>0\quad
\forall A>0\qquad
\limsup\limits_{n\rightarrow\infty}\,\,\frac{1}{n^{a-2\gamma'c}}\ln\Proba_{p_c-\varepsilon/n^c}\Big(\abs{B_n^b}>An^a\Big)
\ <\ 0\,.$$
\end{lemma}
\begin{proof}
Take~$A>0$ and~$0<\varepsilon<p_c$. Without loss of generality, we can assume that~$\varepsilon<\varepsilon_0$.
We repeat the proof of lemma~\ref{lemme_majo_souscritique}, but replacing~$p$ with~$p_c-\varepsilon/n^c$.
To control~$\Proba_p\big(\abs{C_{\Lambda(n)}(x_i)}\geqslant n_i\big)$, the upper bound~(\ref{eq675}) is no longer sufficient, because we would need to specify the dependence in~$n$ of~$\lambda(p_c-\varepsilon/n^c)$.
Thus, we use another inequality provided by the same theorem~6.75 in~\cite{Grimmett}, which states that
\begin{equation}
\label{exponential_decay_chi}
\forall p<p_c\qquad
\forall n>\chi(p)^2\qquad
\Proba_p\big(\abs{C(0)}\geqslant n\big)\ \leqslant\ 2\exp\left(-\frac{n}{2\chi(p)^2}\right)\,.
\end{equation}
With our choice of~$\gamma'$, we have that
$$\chi\left(p_c-\frac{\varepsilon}{n^c}\right)^2\ \leqslant\ \frac{n^{2\gamma'c}}{\varepsilon^{2\gamma'}}\ =\ o\left(n^b\right)\,.$$
Hence, the condition~$n^b \geqslant \chi\left(p_c-\varepsilon/n^c\right)^2$ is satisfied for~$n$ large enough.
This allows us to apply~(\ref{exponential_decay_chi}) to get, with~\smash{$N_n=1+\Ent{An^{a-b}}$} as in the proof of lemma~\ref{lemme_majo_souscritique},
\begin{align*}
\Proba_{p_c-\varepsilon/n^c}\Big(\abs{B_n^b}>An^a\Big)
\ &\leqslant\ 
\sum_{k=1}^{N_n}\,
\sum_{x_1,\,\ldots,\,x_k\in\Lambda(n)}
\,\sum_{\substack{n^b\leqslant n_1,\,\ldots,\,n_k\leqslant n^d\\n_1+\cdots+n_k>An^a}}
\prod_{i=1}^k \,2\exp\left(-\frac{n_i}{2\chi(p_c-\varepsilon/n^c)^2}\right)\\
\ &\leqslant\ N_n n^{2dN_n} \exp\big(-A\varepsilon^{2\gamma'}n^{a-2\gamma'c}\big)\,.
\end{align*}
Therefore, we obtain
\begin{align*}
\frac{1}{n^{a-2\gamma'c}}\ln\Proba_{p_c-\varepsilon/n^c}\Big(\abs{B_n^b}>An^a\Big)
\ \leqslant\ \frac{\ln N_n}{n^{a-2\gamma'c}}+\frac{2N_nd\ln n}{n^{a-2\gamma'c}}-A\varepsilon^{2\gamma'}
\ &=\ O\left(\frac{\ln n}{n^{a-2\gamma'c}}\right)+O\left(\frac{1}{n^{b-2\gamma'c}}\right)-A\varepsilon^{2\gamma'}\\
\ &=\ o(1)-A\varepsilon^{2\gamma'}\,,
\end{align*}
which proves the desired upper bound.
\end{proof}

\subsubsection{Supercritical phase}
\begin{lemma}
\label{vitesse_surcrit}
We have the upper bound
$$\forall\varepsilon>0\quad
\forall A>0\qquad
\limsup\limits_{n\rightarrow\infty}\,\,\frac{1}{n^{d-bd-2\beta'c}}\ln\Proba_{p_c+\varepsilon/n^c}\Big(\abs{B_n^b}<An^a\Big)
\ <\ 0\,.$$
\end{lemma}

\noindent This bound is rougher than the bound we proved in lemma~\ref{lemme_majo_surcritique}, but it presents the advantage of using only~$\theta(p)$ which we have assumed to scale as~$(p-p_c)^{\beta}$.
The counterpart is that it only works with~$b<1$.
\begin{proof}
We let~$N=\Ceil{n^b}$, and we divide the box~$\Lambda(n)$ into smaller boxes of side~$3N$, leaving apart the remainder, meaning that we write
$$\Lambda(n)\ \supset\ \bigsqcup_{i=1}^{\Ent{n/(3N)}^d}B_i\,,$$
where the boxes~$B_i=\Lambda(3N)+\tau_i$ are disjoint translates of~$\Lambda(3N)$.
If a vertex~$x\in\Lambda(N)$ is connected to the boundary~$\partial\Lambda(3N)$, then the cluster of~$x$ in the box~$\Lambda(3N)$ contains at least~$N\geqslant n^b$ vertices, whence
$$S_n\ \stackrel{\text{def}}{=}\ \sum_{i=1}^{\Ent{n/(3N)}^d}\abs{\acc{\,x\in\big(\Lambda(N)+\tau_i\big)\ :\ x\connecte\partial B_i\text{ inside }B_i\,}}
\ \leqslant\ \abs{B_n^b}\,.$$
The boxes~$B_i$ being disjoint, the variables in the above sum are pairwise independent.
Besides, the expectation of this sum is
$$\mathbb{E}_{p_c+\varepsilon/n^c}\big[S_n\big]
\ \geqslant\ \Ent{\frac{n}{3N}}^d N^d \theta\left(p_c+\frac{\varepsilon}{n^c}\right)
\ \geqslant\ \Ent{\frac{n}{3N}}^d N^d \frac{\varepsilon^{\beta'}}{n^{\beta'c}}
\ \eqninfty\ \frac{\varepsilon^{\beta'}}{3^d}n^{d-\beta'c}\,.$$
Using that~$d-c\beta'>a$, we deduce that, for~$n$ large enough, we have
$$\mathbb{E}_{p_c+\varepsilon/n^c}\big[S_n\big]\ \geqslant 2An^a\,.$$
Therefore, applying Hoeffding's inequality (see~\cite{Hoeffding}) yields that, for~$n$ large enough,
\begin{align*}
\Proba_{p_c+\varepsilon/n^c}\Big(B_n^b< An^a\Big)
\ &\leqslant\ \Proba_{p_c+\varepsilon/n^c}\bigg(S_n-\mathbb{E}_{p_c+\varepsilon/n^c}\big[S_n\big]
\,<\,-\demi\mathbb{E}_{p_c+\varepsilon/n^c}\big[S_n\big]\bigg)\\
\ &\leqslant\ \exp\left(-\frac{\mathbb{E}_{p_c+\varepsilon/n^c}\big[S_n\big]^2}{2\Ent{n/(3N)}^d N^{2d}}\right)
\ \leqslant\ \exp\left(-\Ent{\frac{n}{3N}}^d\frac{\varepsilon^{2\beta'}}{2n^{2\beta'c}}\right)\,,
\end{align*}
which concludes the proof, using that~$\Ent{n/(3N)}^d\sim n^{d-bd}/3^d$.
\end{proof}

\subsubsection{Conclusion}
Combining the lower bound on~$Z_n$ obtained in lemma~\ref{lemme_minoration_Z_n} and the results of lemmas~\ref{vitesse_souscrit} and~\ref{vitesse_surcrit}, we get the convergence of~$n^c(p_n-p_c)$ to~$0$, using that
$$a-2\gamma'c\ >\ b
\qquadet
d-bd-2\beta'c\ >\ b\,.$$
In fact, using also the lower bound of lemma~\ref{lemme_minoration_Z_n_bis}, a slightly larger admissible window for~$c$ can be obtained, namely
$$0<c<\min\left(\frac{b}{2\gamma},\,\frac{d-a}{\beta},\,\frac{\max\big[a-b,\,a/d+\min(a-1-b/d,0)\big]}{2\gamma},\,\frac{d-bd+\max\big[-b,\,a/d-\max(1+b/d,a)\big]}{2\beta}\right)\,,$$
but we have preferred to present the simpler condition in the statement of the theorem, since none of them is optimal anyway.

\subsection{An alternative model with cluster diameters}

The variant obtained by replacing~$B_n^b$ with the function~$\widetilde{B}_n^b$ defined by~(\ref{Bntilde}) can be dealt with using the same techniques.
The main difference is that, instead of using theorem~6.75 of~\cite{Grimmett}, we use the theorem~5.4 therein, which states that
$$\forall p<p_c\quad
\exists\,\psi(p)>0\quad
\forall n\geqslant 1\qquad
\Proba_p\left(0\connecte\partial\Lambda(n)\right)\ \leqslant\ e^{-n\psi(p)}\,.$$

\noindent \textbf{Acknowledgments:}
We wish to thank two anonymous referees, one who suggested an improvement of our previous version of lemma~\ref{lemme_majo_surcritique}, and another one for numerous remarks which helped us to improve the presentation of our paper.

%\nocite{*}
\bibliography{Some_toy_models_of_SOC_in_percolation.bib}

\end{document}